\documentclass{siamltex} 

\usepackage[usenames,dvipsnames]{xcolor}
\usepackage{ulem,cancel}
\usepackage{graphicx,color}
\usepackage{amsmath,amssymb}
\usepackage{hyperref}
\usepackage{comment}
\usepackage{verbatim}
\usepackage{array,multirow}
\usepackage{tipa}
\usepackage{pifont}
\usepackage{caption}
\usepackage{subcaption}
\usepackage{tikz}
\usetikzlibrary[patterns]

\newcommand{\oname}[1]{\operatorname{#1}}
\newcommand{\eqdef}{\overset{\mathrm{def}}{=\joinrel=}}
\newcommand{\bs}[1]{\boldsymbol{#1}}
\def\abs#1{\left|#1\right|}

\title{A general approach to enhance slope limiters on non-uniform rectilinear grids}

\author{Xianyi Zeng\thanks{Postdoctoral associate, Department of Civil and Environmental Engineering, Duke University, Durham, North Carolina 27708 ({\tt xy.zeng@duke.edu}).}}

\begin{document}

\maketitle

\setlength{\unitlength}{1cm}
\setlength{\tabcolsep}{1pt}

\begin{abstract}
Most slope limiter functions in high-resolution finite volume methods to solve hyperbolic conservation laws are designed assuming one-dimensional uniform grids, and they are also used to compute slope limiters in computations on non-uniform rectilinear grids.
However, this strategy may lead to either loss of total variation diminishing (TVD) stability for 1D linear problems or the loss of formal second-order accuracy if the grid is highly non-uniform.
This is especially true when the limiter function is not piecewise linear.
Numerical evidences are provided to support this argument for two popular finite volume strategies: MUSCL in space and method of lines in time (MUSCL-MOL), and capacity-form differencing.
In order to deal with this issue, this paper presents a general approach to study and enhance the slope limiter functions for highly non-uniform grids in the MUSCL-MOL framework.
This approach extends the classical reconstruct-evolve-project procedure to general grids, and it gives sufficient conditions for a slope limiter function leading to a TVD stable, formal second-order accuracy in space, and symmetry preserving numerical scheme on arbitrary grids.
Several widely used limiter functions, including the smooth ones by van Leer and van Albada, are enhanced to satisfy these conditions.
These properties are confirmed by solving various one-dimensional and two-dimensional benchmark problems using the enhanced limiters on highly non-uniform rectilinear grids.
\end{abstract}



\pagestyle{myheadings}
\thispagestyle{plain}
\markboth{X. ZENG}{SLOPE LIMITERS ON NON-UNIFORM GRIDS}

\section{Introduction}
\label{sec:intro}
Limiting strategies are often employed in high-resolution finite volume methods (FVM) to solve scientific and engineering problems that are governed by time-dependent hyperbolic conservation laws. 
These problems are usually characterized by discontinuous solutions at finite time even when the initial conditions are smooth.
Practical limiting strategies include the slope limiters~\cite{BvanLeer:1979a, SOsher:1985a, PColella:1985a, MEHubbard:1999a} and the flux limiters~\cite{JPBoris:1973a, STZalesak:1979a,CRDeVore:1991a}, and they are equivalent to each other in certain situations.
In both cases, the limiters are computed by evaluating a limiter function with a smoothness monitor as input, which measures the local solution regularity.
These limiter functions are originally designed using one-dimensional uniform grids.
It is in this simple situation a number of properties, such as total variation diminishing (TVD) stability, formal second-order accuracy in space for smooth solutions, and the symmetry-preserving property are guaranteed for solving 1D linear problems by using these limiter functions.
However, in practical computations they are often applied in computations using non-uniform grids.
The limitations of this strategy are well documented in literature~\cite{VVenkatakrishnan:1995a,MBerger:2005a}.

The limitations are further shown in this paper (Sections~\ref{sec:case}, \ref{sec:alt}, and \ref{sec:example}) by numerical examples.
In particular, two different finite volume methodologies are considered in these examples: (1) the monotonic upstream-centered scheme for conservation laws (MUSCL) in space and method of lines (MOL) in time, referred to as MUSCL-MOL in this paper, and (2) the capacity-form differencing.
Note that different implementations may lead to different numerical properties on non-uniform grids, but the general observation is that either the TVD stability or the formally second-order spatial accuracy will be lost, unless the piecewise linear limiter functions (such as the {\it minmod}, {\it superbee}, or {\it monotonized central (MC)} limiters) are used with care (more details are offered in Section~\ref{sec:alt}).
Because non-smooth limiter functions cause convergence difficulties for steady-state computations~\cite{VVenkatakrishnan:1995a}, using piecewise linear limiters is not always a solution to circumvent the aforementioned limitations.

This brings out the main question to be answered in the present paper: how to analyze and enhance the slope limiter functions on general rectilinear grids, so that all the desired properties (TVD stability, formal second-order accuracy, and symmetry-preserving) are maintained on these grids.
To this end, this paper starts with the MUSCL-MOL framework, and chooses a representative implementation, which leads to the loss of second-order accuracy but retaining the TVD stability (in the 1D linear case) if conventional limiter functions are used.
The effects of limiters are studied by extending the classical reconstruct-evolve-project (REP) procedure~\cite{JBGoodman:1988a} to take into account of the non-uniformity of the grids.
The analysis leads to sufficient conditions that a slope limiter function should satisfy, so that the TVD stability, second-order accuracy, and the symmetry-preserving property are preserved on general one-dimensional grids.
Several most widely used limiter functions are enhanced, so that they satisfy these conditions.
In particular, the enhanced {\it minmod}, {\it superbee}, and {\it MC} limiters have exactly the same form as given in~\cite{MBerger:2005a}; 
the enhanced {\it van Leer} limiter is smoother than the one in the same reference; 
and the enhanced {\it van Albada} limiter, to the knowledge of the author, is the first time proposed.

These sufficient conditions may not be valid for other MUSCL implementation, but the methodology extends naturally.
However, the methodology does not apply to the capacity-form differencing, mainly due to the reason that the latter uses the classical Riemann problem to approximate the Riemann problem of the governing equation rewritten in capacity-form~\cite{RJLeVeque:2002a}.
For this reason, using conventional limiters may not be fully responsible to the loss of stability or accuracy of the capacity-form differencing on non-uniform grids, as observed in the numerical examples in this paper. 

The remainder of the paper is organized as follows. 
Section~\ref{sec:case} briefly reviews MUSCL-MOL and presents a case study to demonstrate the limitation of conventional limiter functions for a particular implementation of MUSCL on highly non-uniform grids. 
The slope limiters are studied using an extension of the REP procedure in Section~\ref{sec:math}, in which sufficient conditions are derived for designing slope limiter functions such that the resulting schemes are TVD stable, second-order accurate in space and preserve symmetric solutions.
Section~\ref{sec:limiter} contains enhancements to several most widely used limiter functions, so that they satisfy these conditions.
Further 1D and 2D examples are present in Section~\ref{sec:alt} and Section~\ref{sec:example}, respectively, to confirm that the enhanced limiter functions perform as expected, and to compare the results with alternative strategies including MUSCL with a different implementation and the capacity-form differencing.
Finally, Section~\ref{sec:concl} concludes this paper.

\section{Review of MUSCL and Case study}
\label{sec:case}

\subsection{The MUSCL scheme}
Consider applying the MUSCL scheme~\cite{BvanLeer:1979a} to solve the 1D scalar conservation law 
\begin{equation}\label{eq:sec_case_1dcl}
u_t + f(u)_x = 0
\end{equation}
The computational domain $\Omega$ is divided into cells $\Omega_i = [x_{i-1/2},x_{i+1/2}]$, where $\cdots<x_{i-1/2}<x_{i+1/2}<x_{i+3/2}<\cdots$ are cell faces.
The center of $\Omega_i$ is $x_i = (x_{i-1/2}+x_{i+1/2})/2$, and the size of the cell is $\Delta x_i\eqdef|\Omega_i| = x_{i+1/2}-x_{i-1/2}$.
The mesh is supposed to be fixed in time.

The following notations are used in this paper: 
$u(x,t)$ denotes the exact solution of the PDE of a variable $u$, $\bar{u}_i(t)$ designates the exact cell-averaged data at time-instance $t$ over the cell $\Omega_i$, $u_i^n$ designates the numerical approximation of the value $\bar{u}_i(t^n)$, and $u_i$ denotes the semi-discretized approximation of $\bar{u}_i$.

Integrating Eq.~(\ref{eq:sec_case_1dcl}) in space over $\Omega_i$, one obtains
\begin{equation}\label{eq:sec_case_int}
\frac{d\bar{u}_i(t)}{dt} + \frac{f(u(x_{i+1/2},t))-f(u(x_{i-1/2},t))}{\Delta x_i} = 0
\end{equation}
where $\bar{u}_i(t)$ is defined by
\begin{equation}\label{eq:sec_case_ave}
\bar{u}_i(t) = \frac{1}{\Delta x_i}\int_{x_{i-1/2}}^{x_{i+1/2}}u(x,t)dx
\end{equation}
Eq.~(\ref{eq:sec_case_int}) leads to the semi-discretized equations for the approximated solution $u_i$
\begin{equation}\label{eq:sec_case_semi}
\frac{du_i}{dt} + \frac{F_{i+1/2}-F_{i-1/2}}{\Delta x_i} = 0
\end{equation}
Here $F_{i+1/2}$ is the numerical flux approximating $f(u(x_{i+1/2},t))$; and it is calculated using the semi-discrete solutions $u_k$. 
In the MUSCL approach, these fluxes are obtained from exact or approximated Riemann solvers, namely
\begin{equation}\label{eq:sec_case_linr}
F_{i+1/2} = F^{Riem}(u_i+\frac{1}{2}\sigma_i\Delta x_i, u_{i+1}-\frac{1}{2}\sigma_{i+1}\Delta x_{i+1})
\end{equation}
Here $F^{Riem}(u_l,u_r)$ is the Riemann solver that takes the left value $u_l$ and right value $u_r$ as inputs, and $\sigma_i$ is the numerical slope in the cell $\Omega_i$.
The canonical Godunov scheme~\cite{SKGodunov:1959a} fits in this framework by setting $\sigma_{i}\equiv0$ and using the exact Riemann solver.
Due to the high computational cost of the exact Riemann solver for Euler equations, the Roe solver~\cite{PLRoe:1981a} $F^{Roe}$ is used in this paper.

When $\sigma_i$ is consistent with the local gradient $u_x(x_i,t)$, Eq.~(\ref{eq:sec_case_linr}) leads to a second-order accurate discretization in space~\cite{SOsher:1985a}.
One of the most natural choices for such $\sigma_i$ on uniform meshes is given by
\begin{equation}\label{eq:sec_case_slpe}
\sigma_i = D_x^+u_i\eqdef\frac{1}{\Delta x_i}(u_{i+1}-u_i)
\end{equation}
To avoid spurious oscillations near discontinuities caused by Eq.~(\ref{eq:sec_case_slpe}), 
the MUSCL approach scales $D_xu_i$ with a factor $\phi$, called the slope limiter
\begin{equation}\label{eq:sec_case_muscl}
\sigma_i = \phi_i D_xu_i
\end{equation}
Here and in the remains of the paper, the symbol $\phi$ is reserved for slope limiters. 

Once the limited slopes $\sigma_i$ are computed, one may use any preferred ODE solver, such as the forward Euler (for simplicity of analysis), to discretize Eq.~(\ref{eq:sec_case_semi}) in time
\begin{equation}\label{eq:sec_case_fw}
\frac{u_i^{n+1}-u_i^n}{\Delta t} + \frac{F_{i+1/2}^n-F_{i-1/2}^n}{\Delta x_i} = 0
\end{equation}

The nature of this numerical approximation is clear by integrating Eq.~(\ref{eq:sec_case_int}) over the time interval $[t^n, t^{n+1}]$ to obtain an exact integral form of (\ref{eq:sec_case_1dcl}), namely
\begin{align}
\notag
  \frac{\bar{u}_i(t^{n+1})-\bar{u}_i(t^n)}{\Delta t^n}
+ \frac{1}{\Delta x_i}&\left(
\frac{1}{\Delta t^n}\int_{t^n}^{t^{n+1}}f(u(x_{i+1/2},t))dt \right.\\
\label{eq:sec_case_int_form}
&\left.- \frac{1}{\Delta t^n}\int_{t^n}^{t^{n+1}}f(u(x_{i-1/2},t))dt
\right) = 0
\end{align}
where $\Delta t^n = t^{n+1}-t^n$. 

Thus the numerical flux $F_{i+1/2}^n$ is interpreted as an approximation to the time-averaged flux at the cell face
\begin{equation}\label{eq:sec_case_num_flux}
F_{i+1/2}^n\approx\frac{1}{\Delta t^n}\int_{t^n}^{t^{n+1}}f(u(x_{i+1/2},t))dt
\end{equation}
This interpretation plays an important role in subsequent analysis in Section~\ref{sec:math}.

\subsection{Slope limiters}
In practice, $\phi_i$ in Eq.~(\ref{eq:sec_case_slpe}) is a function of a local smoothness monitor $\theta_i$, of which a popular choice \cite{BvanLeer:1973a, MBerger:2005a} is 
\begin{equation}\label{eq:sec_case_monitor}
\theta_i = \frac{u_i-u_{i-1}}{u_{i+1}-u_i}
\end{equation}
Conventional practice of MUSCL defines the limiters $\phi_i$ as
\begin{equation}\label{eq:sec_case_conv}
\phi_i = \phi(\theta_i)
\end{equation}
where the absence of the subscript of $\phi()$ highlights the common strategy that a single limiter function is applied everywhere in the computational domain. 

On the one hand, A. Harten~\cite{AHarten:1983a} showed that on uniform meshes, a sufficient condition for the fully discretized system~(\ref{eq:sec_case_fw}) to be TVD stable is 
\begin{equation}\label{eq:sec_case_tvd}
0\le\phi(\theta)\le2,\quad 0\le\frac{\phi(\theta)}{\theta}\le 2,\quad\forall\theta\in\mathbb{R}
\end{equation}

On the other hand, one may study the effectiveness of limiters in retaining formal second-order accuracy for smooth solutions by writing Eq.~(\ref{eq:sec_case_linr}) in the flux-corrected form
\begin{equation}\label{eq:sec_case_fc}
F(u_i+\frac{1}{2}\sigma_i\Delta x_i, u_{i+1}-\frac{1}{2}\sigma_{i+1}\Delta x_{i+1}) = F^L_{i+1/2}+\varphi_{i+1/2}\left(F_{i+1/2}^H-F^L_{i+1/2}\right)
\end{equation}
Here $F^L$ is any first-order numerical flux $F^L_{i+1/2} = F(u_i,u_{i+1})$; 
and $F^H$ designates a second-order numerical flux.
The flux limiters $\varphi$ are defined at the cell faces, and they make the solution-adapted transition from low-order flux to second-order flux possible.
Typically, $\varphi_{i+1/2}$ is a function of another local smoothness monitor defined at cell faces $\theta_{i+1/2}$ 
\begin{equation}\label{eq:sec_case_fluxlim}
\varphi_{i+1/2} = \varphi(\theta_{i+1/2})
\end{equation}
where the absence of the subscript of $\varphi()$ again highlights the practice that a single flux limiter is utilized everywhere.
A sufficient condition to preserve second-order spatial accuracy is $\varphi_{i+1/2}=1$ for linear data (linearity preserving), 
which leads to 
\begin{equation}\label{eq:sec_case_2nd}
\varphi(1) = 1
\end{equation}
in the case of uniform grids.

For linear problems, one can often define $\theta_j$ and $\theta_{j+1/2}$ properly so that Eqs.~(\ref{eq:sec_case_fc}) and (\ref{eq:sec_case_linr}) are equivalent to each other. 
This provides a convenient tool to study the formal order of accuracy a limiter could deliver, as presented in Section~\ref{sec:math}.

Table~\ref{tb:sec_case_lim} lists several most widely used limiter functions~\cite{PLRoe:1986a, BvanLeer:1974a, BvanLeer:1977a, GDvanAlbada:1982a}. 
Here the superscript $^+$ designates the positive part of a real number.
All these limiter functions satisfy~(\ref{eq:sec_case_tvd}) and (\ref{eq:sec_case_2nd}). 
\begin{table}\centering
\caption{Limiter functions $\phi^{\it \textrm{minmod}}, \phi^{\it \textrm{superbee}}, \phi^{\it \textrm{MC}}, \phi^{\it \textrm{van Leer}}$, and $\phi^{\it \textrm{van Albada}}$} 
\label{tb:sec_case_lim}
{\small
\begin{tabular}{|c||c|c|c|c|c|}
\hline
Limiter        
& minmod                   
& superbee                                 
& MC 
& van Leer 
& van Albada \\ \hline
$\phi(\theta)$ 
& $\max(\theta,1)^+$
& $\max(\min(2\theta,1),\min(\theta,2))^+$
& $\min(2\theta,\frac{1+\theta}{2},2)^+$
& $\frac{\theta+\abs{\theta}}{1+\abs{\theta}}$ 
& $\frac{\theta+\theta^2}{1+\theta^2}$ \\ \hline
\end{tabular} 
}
\end{table}

\subsection{Loss of accuracy on non-uniform grids}
The preceding methods, however, lead to loss of accuracy on highly irregular grids, as demonstrated by the subsequent example.
Consider the 1D Euler equations with periodic boundary conditions
\begin{equation}\label{eq:sec_case_1deuler}
\mathbf{w}_t + \mathbf{f}(\mathbf{w})_x = 0,\quad
\mathbf{w} = 
\left[\begin{array}{c}
\rho \\ \rho u \\ E
\end{array}\right],\quad
\mathbf{f}(\mathbf{w}) = 
\left[\begin{array}{c}
\rho u \\ \rho u^2+p \\ u(E+p)
\end{array}\right],\quad x\in[-1,1]
\end{equation}
where $\mathbf{w}$ is the conservative fluid state vector and $\mathbf{f}$ is the physical flux function. 
$\rho$, $u$ and $p$ are density, velocity, and pressure, respectively.
$E = p/(\gamma-1) + \rho u^2/2$ is the total energy, with $\gamma=1.4$ being the heat capacity ratio.
The initial condition is 
\begin{equation}\label{eq:sec_case_ic}
\rho(x,0) = 1 + 0.5\sin(\pi x),\quad
u(x,0) = 2 + 0.5\sin(\pi x),\quad
p(x,0) = 1 + 0.5\sin(\pi x)
\end{equation}

Given the number of cells $N$, the irregular mesh is constructed by perturbing the uniform grid as follows
{\small
\begin{enumerate}
\item Let $x_{i+1/2}',\ i=0,1,\cdots,N$ be the cell faces of the uniform mesh.
\item Set a fixed ratio $r: 0\le r<0.5$. 
For each grid point $x_{i+1/2}', i=1,\cdots,N-1$, define $x_{i+1/2} = x_{i+1/2}'+r\delta_{i+1/2}$ where $\delta_{i+1/2}$ are independent random variables obeying uniform distribution on $[-h, h]$, $h=2/N$.
\item The two endpoints are fixed: $x_{1/2}=x_{1/2}'$ and $x_{(2N+1)/2}=x_{(2N+1)/2}'$.
\end{enumerate}
}
In all tests, MUSCL-MOL with Roe flux and limiter functions in Table~\ref{tb:sec_case_lim} in space, and second-order TVD Runge-Kutta (TVD RK2)~\cite{SGottlieb:1998a} in time is used.
The time step size is computed using the Courant number $0.6$.

Given any limiter function and ratio $r$, solutions on five grids with number of cells ranging from $100$ to $1600$ are computed. 
The global $L_1$ errors in $\rho$ calculated with $\phi^{\it \textrm{van Albada}}$ are reported in Table~\ref{tb:sec_case_l1_den}.
\begin{table}\centering
\caption{$L_1$ errors in $\rho$ using RK2, MUSCL with $\phi^{\it \textrm{van Albada}}$, and grids with $r=0, 0.2$, and $0.3$}
\label{tb:sec_case_l1_den}
{\small
\setlength{\tabcolsep}{5pt}
\begin{tabular}{|c||cc|cc|cc|}
\hline
& \multicolumn{2}{c|}{$r=0.0$} 
& \multicolumn{2}{c|}{$r=0.2$}
& \multicolumn{2}{c|}{$r=0.3$} \\ \hline
Mesh & error & rate & error & rate & error & rate \\ \hline
$100$  & $3.40\text{\sc{e}-}3$ &        & $4.55\text{\sc{e}-}3$ &        & $7.82\text{\sc{e}-}3$ &        \\
$200$  & $9.66\text{\sc{e}-}4$ & $1.82$ & $2.12\text{\sc{e}-}3$ & $1.10$ & $4.02\text{\sc{e}-}3$ & $0.96$ \\
$400$  & $2.15\text{\sc{e}-}4$ & $2.17$ & $9.82\text{\sc{e}-}4$ & $1.11$ & $2.04\text{\sc{e}-}3$ & $0.98$ \\
$800$  & $4.61\text{\sc{e}-}5$ & $2.22$ & $4.68\text{\sc{e}-}4$ & $1.07$ & $9.72\text{\sc{e}-}4$ & $1.07$ \\
$1600$ & $1.04\text{\sc{e}-}5$ & $2.15$ & $2.41\text{\sc{e}-}4$ & $0.96$ & $5.03\text{\sc{e}-}4$ & $0.95$ \\ \hline
\end{tabular}
\setlength{\tabcolsep}{1pt}
}
\end{table}
Similarly, Table~\ref{tb:sec_case_l1_err} summarizes the convergence rates (using the two coarsest meshes) in $\rho$, $u$, and $p$ by various limiter functions; 
the rates on irregular grids are computed according to the method in Appendix~\ref{app:rates}.
\begin{table}\centering
\caption{Convergence rates in $\rho$, $u$, and $p$ by various limiter functions}
\label{tb:sec_case_l1_err}
{\small
\setlength{\tabcolsep}{5pt}
\begin{tabular}{|c||ccc|ccc|ccc|}
\hline
& \multicolumn{3}{c|}{$r=0.0$} 
& \multicolumn{3}{c|}{$r=0.2$}
& \multicolumn{3}{c|}{$r=0.3$} \\ \hline
Limiter & $\rho$ & $u$ & $p$ & $\rho$ & $u$ & $p$ & $\rho$ & $u$ & $p$ \\ \hline
$\phi^{\it \textrm{minmod}}$      & $1.52$ & $1.70$ & $1.68$ & $0.98$ & $1.05$ & $1.02$ & $0.95$ & $1.02$ & $1.00$ \\
$\phi^{\it \textrm{superbee}}$    & $1.58$ & $1.95$ & $1.82$ & $1.08$ & $1.06$ & $1.06$ & $1.08$ & $1.11$ & $1.08$ \\
$\phi^{\it \textrm{MC}}$          & $2.13$ & $2.15$ & $2.22$ & $1.22$ & $1.33$ & $1.36$ & $1.06$ & $1.15$ & $1.19$ \\
$\phi^{\it \textrm{van Leer}}$   & $1.80$ & $1.99$ & $2.00$ & $1.12$ & $1.29$ & $1.28$ & $1.00$ & $1.09$ & $1.06$ \\
$\phi^{\it \textrm{van Albada}}$ & $1.82$ & $2.00$ & $1.99$ & $1.10$ & $1.26$ & $1.22$ & $0.96$ & $1.07$ & $1.01$ \\ \hline
\end{tabular}
\setlength{\tabcolsep}{1pt}
}
\end{table}

On the one hand, the results reveal that the convergence rates degrade towards first-order when the grids become more and more irregular.
On the other hand, the stability seems to be unaffected by the fact that irregular grids are utilized.
This is observed by long term computations (up to $T=6.0$), shown in Figure~\ref{fg:sec_case_lt}.
\begin{figure}\centering
\begin{minipage}[t]{.48\textwidth}
\includegraphics[width=\textwidth]{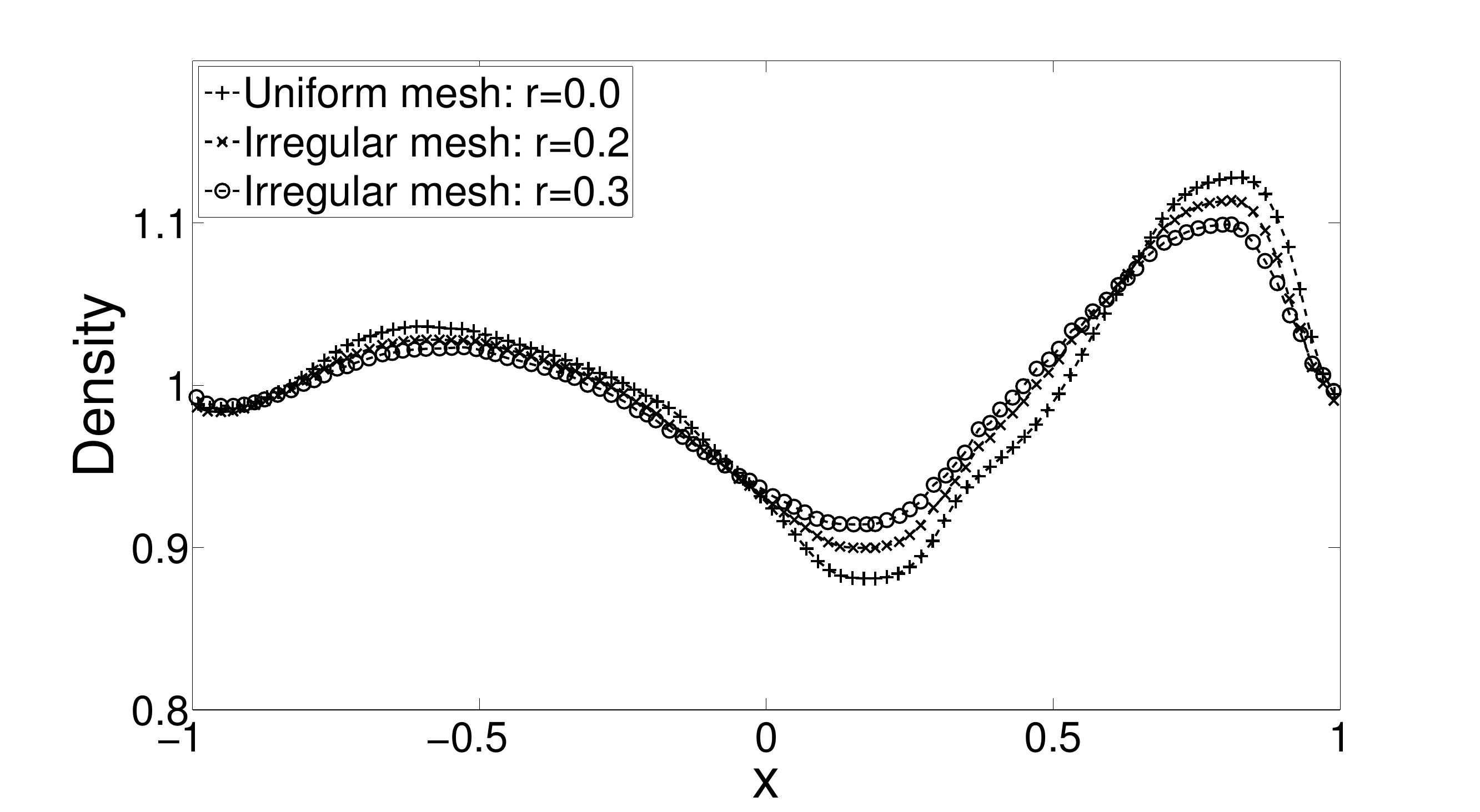}
\end{minipage}
\begin{minipage}[t]{.04\textwidth}
\end{minipage}
\begin{minipage}[t]{.48\textwidth}
\includegraphics[width=\textwidth]{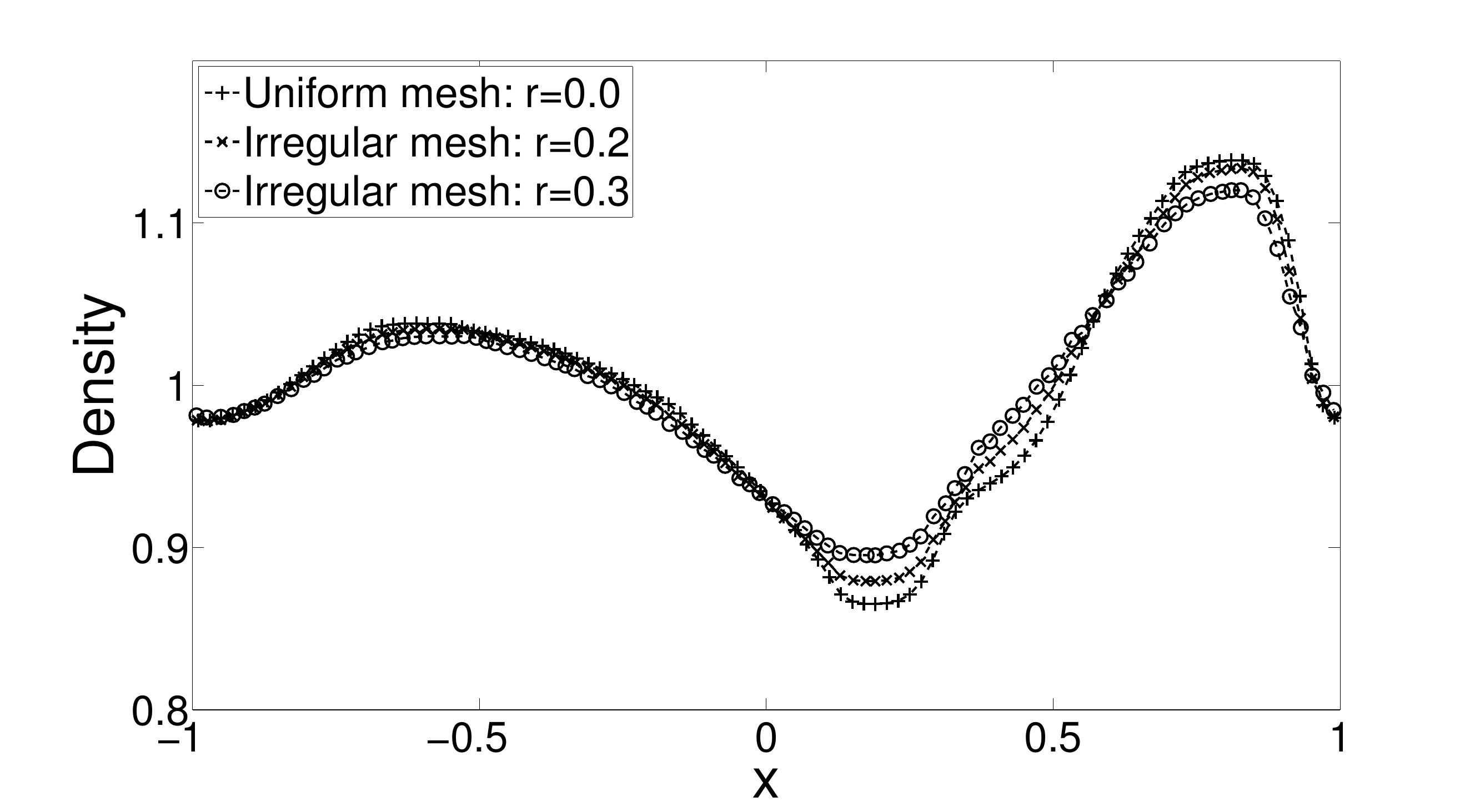}
\end{minipage}
\vglue -0.1 truein
\caption{Densities at $T=6.0$ using $\phi^{\it \textrm{minmod}}$ (left) and $\phi^{\it \textrm{van Albada}}$ (right)}
\label{fg:sec_case_lt}
\end{figure}

{\bf Remark 1:} 
In practice, the slope limiter is sometimes implemented in a way such that it is effective only if discontinuities are detected.
This strategy, however, falls into the range of discussion of current paper by noticing that it is equivalent to a limiter function, such as
\begin{displaymath}
\phi_{switch}(\theta) = 1 + \chi_{\{|\theta-1|>\delta\}}(\phi(\theta)-1)
\end{displaymath}
where $\chi$ is the characteristic function, and $\delta$ is the threshold to determine whether discontinuities present.

{\bf Remark 2:}
The particular choices of $D_xu_i$ and $\theta_i$, given by (\ref{eq:sec_case_slpe}) and (\ref{eq:sec_case_monitor}), respectively, are supposed throughout this paper, except in Section~\ref{sec:alt}, where alternative strategies are briefly discussed.

\section{Mathematical analysis}
\label{sec:math}
This section focuses on studying the mechanism behind which the conventional limiters lead to loss of accuracy on non-uniform grids while retaining the nonlinear stability.

\subsection{Slope limiters and TVD stability}
Applying the classical REP approach (Figure~\ref{fg:sec_math_rep}) to solve the 1D scalar advection equation
\begin{equation}\label{eq:sec_math_1dadv}
u_t+cu_x = 0,\quad c\equiv\textrm{constant}
\end{equation}
using non-uniform meshes, updating the data from $t^n$ to $t^{n+1}$ involves four components
{\small
\begin{enumerate}
\item Obtaining the cell averages $u_i^n$ from previous time step.
\item (Reconstruct.)
Linearly reconstructing the data on each cell to obtain the function $\tilde u(x,t^n)$ (Figure~\ref{fg:sec_math_rep_r})
\begin{equation}\label{eq:sec_math_linear_r}
\tilde u(x,t^n) = u_i^n+\sigma_i^n(x-x_i),\quad 
x\in(x_{i-1/2},\ x_{i+1/2})
\end{equation}
\item (Evolve.)
Solving Eq.~(\ref{eq:sec_math_1dadv}) exactly over $[t^n,t^{n+1}]$ with initial condition~(\ref{eq:sec_math_linear_r}) to have the solution at $t^{n+1}$, denoted by $\tilde u(x,t^{n+1})$ (Figure~\ref{fg:sec_math_rep_e}).
\item (Project.)
Computing cell averages at $t^{n+1}$ (Figure~\ref{fg:sec_math_rep_p}) as
\begin{equation}\label{eq:sec_math_ave}
u_i^{n+1} = \frac{1}{\Delta x_i}\int_{x_{i-1/2}}^{x_{i+1/2}}\tilde u(x,t^{n+1})dx
\end{equation}
\end{enumerate}
}
\begin{figure}\centering
\begin{subfigure}[b]{0.32\textwidth}\centering
\includegraphics[trim=2.0in 0.8in 1.1in 0.0in, clip, width=\textwidth]{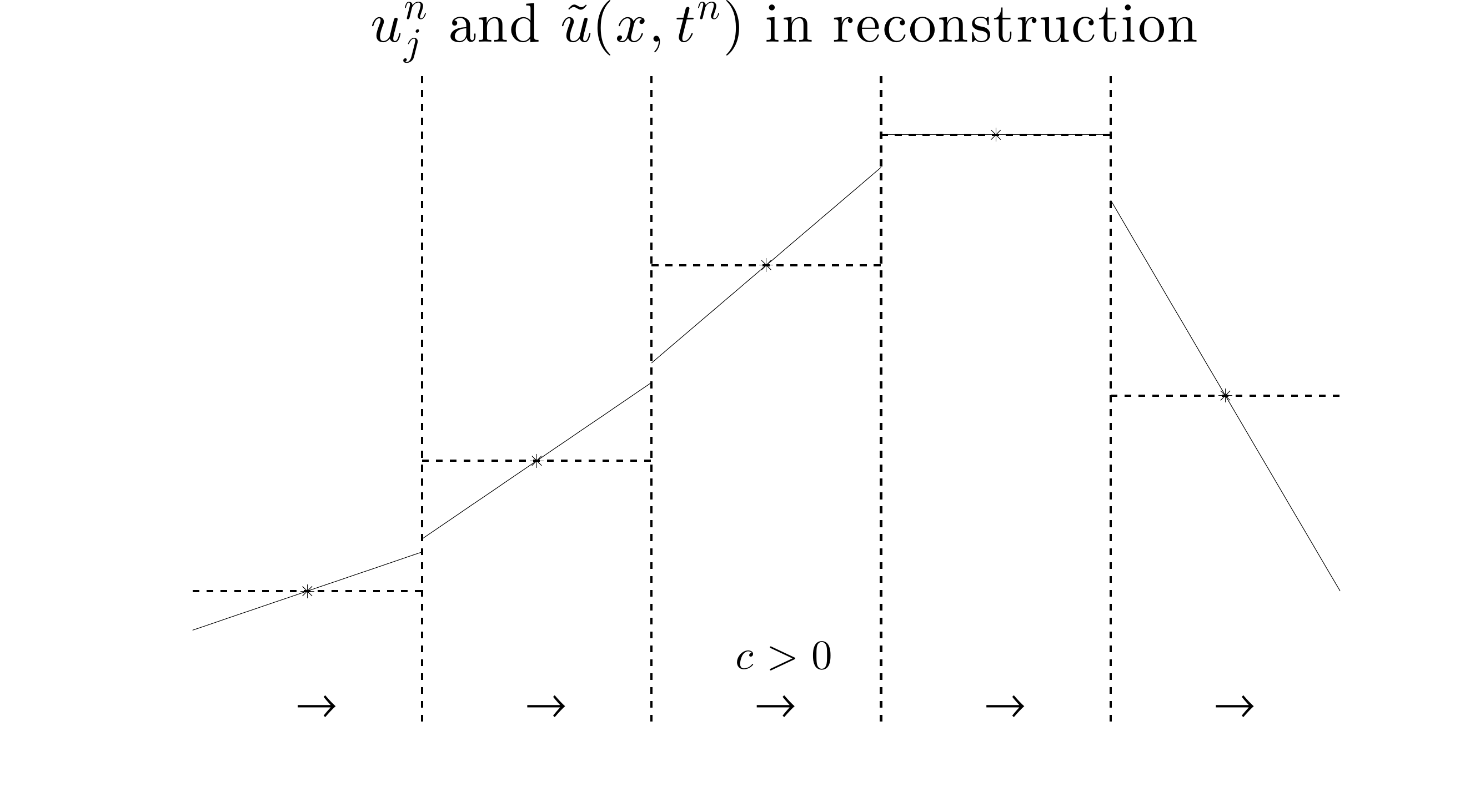}
\vglue -0.05 truein
\caption{REP-1}
\label{fg:sec_math_rep_r}
\end{subfigure}
\begin{subfigure}[b]{0.32\textwidth}\centering
\includegraphics[trim=2.0in 0.8in 1.1in 0.0in, clip, width=\textwidth]{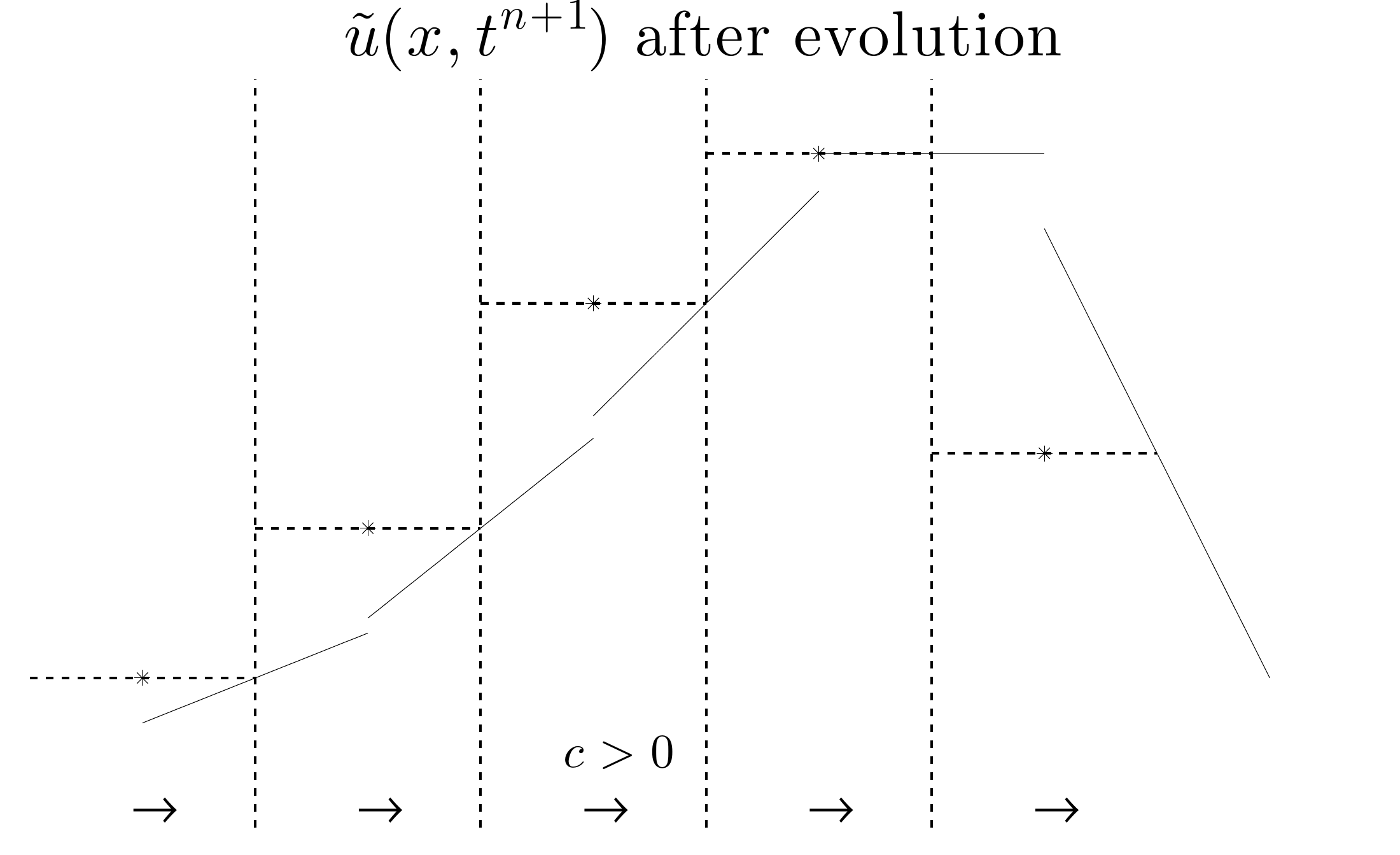}
\vglue -0.05 truein
\caption{REP-2}
\label{fg:sec_math_rep_e}
\end{subfigure}
\begin{subfigure}[b]{0.32\textwidth}\centering
\includegraphics[trim=2.0in 0.8in 1.1in 0.0in, clip, width=\textwidth]{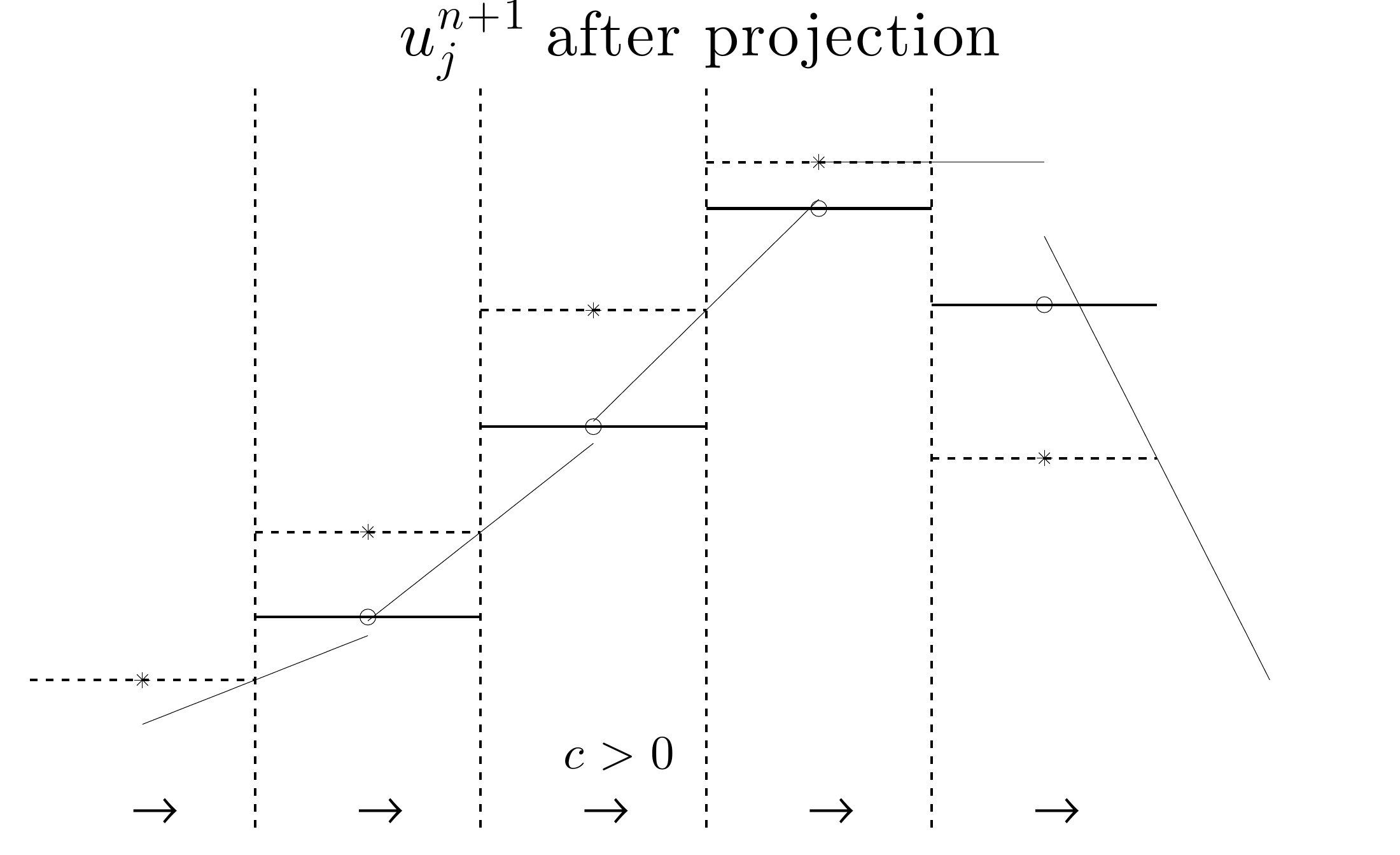}
\vglue -0.05 truein
\caption{REP-3}
\label{fg:sec_math_rep_p}
\end{subfigure}
\vglue -0.2 truein
\caption{REP approach: (\ref{fg:sec_math_rep_r}) cell averages and reconstructed data $\tilde u$ at $t^n$, (\ref{fg:sec_math_rep_e}) solution of $\tilde u$ at $t^{n+1}$, and (\ref{fg:sec_math_rep_p}) cell averages at $t^{n+1}$}
\label{fg:sec_math_rep}
\end{figure}
For general hyperbolic conservation laws, one may perform the ``evolve'' and ``project'' steps approximately.

A classical result concerning the TVD stability is due to Harten~\cite{AHarten:1983a}:
\begin{theorem}[Harten]\label{thm:sec_math_harten}
If the numerical scheme can be written, for one time step, in the following form:
\begin{equation}\label{eq:sec_math_harten_form}
u_i^{n+1} = u_i^n-C_{i-1}^n(u_i^n-u_{i-1}^n)+D_i^n(u_{i+1}^n-u_i^n)
\end{equation}
Then a sufficient condition for the scheme to be TVD is:
\begin{equation}\label{eq:sec_math_harten_cond}
C_{i-1}^n\ge0,\quad
D_i^n\ge0,\quad
C_i^n+D_i^n\le1;\quad
\forall i,n
\end{equation}
here $C_i^n$, $D_i^n$ can be any numbers including those that are data dependent.\\
\end{theorem}

Uniform grids are assumed in the original proof; however, this theorem can be applied to non-uniform meshes naturally, by observing that the definition of discrete total variance (Eq.~(\ref{eq:sec_math_dtv})) is independent of the particular cell sizes. 
\begin{equation}\label{eq:sec_math_dtv}
\textrm{TV}(\{u_i\})\eqdef\sum_i\left|u_i-u_{i+1}\right|
\end{equation}
thus the result in Theorem~\ref{thm:sec_math_harten} carries directly to any 1D grids.

In the subsequent discussion, the upwind cell index of $\Omega_i$ is denoted by $i'$, namely $i'=i+1$ if $c<0$ and $i'=i-1$ if $c>0$.
Then the exact solution to Eq.~(\ref{eq:sec_math_1dadv}) obtained from the REP approach is
\begin{equation}\label{eq:sec_math_sol}
u_i^{n+1} = u_i^n - \lambda_i(u_i^n-u_{i'}^n) - \oname{sgn}(c)\frac{1}{2}\lambda_i\left[(1-\lambda_i)\Delta x_i\sigma_i^n - (1-\lambda_{i'})\Delta x_{i'}\sigma_{i'}^n\right]
\end{equation}
Here $\lambda_i=|c\Delta t^n|/\Delta x_i$ is the absolute value of the local Courant number, and $\oname{sgn}(c)$ is $+1$ if $c\ge0$ and $-1$ if $c<0$.
Assuming the smoothness monitor~(\ref{eq:sec_case_monitor}) and the approximated gradients~(\ref{eq:sec_case_slpe}), the limited slopes are 
\begin{equation}\label{eq:sec_math_slpe}
\sigma_i^n = \frac{1}{\Delta x_i}\phi_i(\theta_i)(u_{i+1}^n-u_i^n)
\end{equation}
Note that the limiter function $\phi$ is equipped with subscript $i$ to allow variance among cells. 
Plugging (\ref{eq:sec_math_slpe}) into (\ref{eq:sec_math_sol}), one may choose the parameters $C_i^n$ and $D_i^n$ in Theorem~\ref{thm:sec_math_harten} as 
\begin{equation}\label{eq:sec_math_c_d_rght}
C_{i-1}^n = \lambda_i+\frac{1}{2}\lambda_i(1-\lambda_i)\frac{\phi_i(\theta_i)}{\theta_i}-\frac{1}{2}\lambda_i(1-\lambda_{i-1})\phi_{i-1}(\theta_{i-1}),\quad
D_i^n = 0
\end{equation}
if $c>0$, and
\begin{equation}\label{eq:sec_math_c_d_left}
C_{i-1}^n = 0,\quad
D_i^n = \lambda_i+\frac{1}{2}\lambda_i(1-\lambda_i)\phi_i(\theta_i)-\frac{1}{2}\lambda_i(1-\lambda_{i+1})\frac{\phi_{i+1}(\theta_{i+1})}{\theta_{i+1}}
\end{equation}
if $c<0$. 
Then the theorem leads to the conclusion that the REP approach is TVD stable given that 
\begin{displaymath}
0\le C_{i-1}^n\le1,\ \textrm{if }c>0;\quad\textrm{ and }\quad0\le D_{i}^n\le1,\ \textrm{if }c<0
\end{displaymath}
These conditions hold whenever the Courant numbers are subject to the usual condition $0\le\lambda_i\le1$ and
\begin{equation}\label{eq:sec_math_tvd}
0\le\phi_i(\theta)\le2,\quad 0\le\frac{\phi_i(\theta)}{\theta}\le2,\quad\forall i,\theta
\end{equation}
For example, in the case $c>0$ and assuming (\ref{eq:sec_math_tvd}), one has
\begin{displaymath}
C_{i-1}^n\ge\lambda_i+\frac{1}{2}\lambda_i(1-\lambda_i)\cdot 0 - \frac{1}{2}\lambda_i(1-\lambda_{i-1})\cdot 2 = \lambda_i\lambda_{i-1}\ge0
\end{displaymath}
and
\begin{displaymath}
C_{i-1}^n\le\lambda_i+\frac{1}{2}\lambda_i(1-\lambda_i)\cdot 2 - \frac{1}{2}\lambda_i(1-\lambda_{i-1})\cdot 0 = 2\lambda_i - \lambda_i^2\le1
\end{displaymath}
The proof of $0\le D_i^n\le 1$ in the case $c<0$ is similar.

Hence (\ref{eq:sec_math_tvd}) is a sufficient condition for the limiter functions $\phi_i$ to lead to TVD stability on non-uniform grids.
This condition is the same as (\ref{eq:sec_case_tvd}), which explains why TVD stability is retained even if the conventional slope limiters are applied with non-uniform meshes as observed in Section~\ref{sec:case}.

\subsection{Second-order spatial accuracy}
Introducing the numerical fluxes
\begin{equation}\label{eq:sec_math_flux}
F_{i+1/2}^n = cu_{i''}^n + \frac{1}{2}\abs{c}(1-\lambda_{i''})\Delta x_{i''}\sigma_{i''}^n
\end{equation}
where $i''$ is the upwind cell index corresponding to a cell face $x_{i+1/2}$, namely $i''=i+1$ if $c<0$ and $i''=i$ if $c>0$, 
the exact solutions~(\ref{eq:sec_math_sol}) is rewritten as
\begin{displaymath}
\frac{1}{\Delta t^n}\left(u_i^{n+1}-u_i^n\right) + 
\frac{1}{\Delta x_i}\left(F_{i+1/2}^n-F_{i-1/2}^n\right) = 0
\end{displaymath}
Plugging the limited slopes (\ref{eq:sec_math_slpe}) into Eq.~(\ref{eq:sec_math_flux}) leads to
\begin{equation}\label{eq:sec_math_flux_lform}
F_{i+1/2}^n = cu_{i''}^n+\frac{\phi_{i''}(\theta_{i''})}{B_{i+1/2}}\left(cu_{i''}^n+\frac{1}{2}B_{i+1/2}\abs{c}(1-\lambda_{i''})(u_{i''+1}^n-u_{i''}^n)-cu_{i''}^n\right)
\end{equation}
Here $B_{i+1/2}$ are coefficients to be determined later.

In order to express these fluxes in the flux-corrected form (\ref{eq:sec_case_fc}), the low-order flux $F^L$ is chosen to be the first-order upwind flux
\begin{equation}\label{eq:sec_math_flux_l}
F_{i+1/2}^{L,n} = cu_{i''}^n
\end{equation}
and the ``second-order'' fluxes 
\begin{equation}\label{eq:sec_math_flux_h}
F_{i+1/2}^{H,n} = 
cu_{i''}^n+\frac{1}{2}B_{i+1/2}\abs{c}\left(1-\frac{\abs{c}\Delta t^n}{\Delta x_{i''}}\right)(u_{i''+1}^n-u_{i''}^n) 
\end{equation}
In order that $F_{i+1/2}^{H,n}$ is indeed second-order, $B_{i+1/2}$ needs to be chosen carefully.
Applying Taylor series expansion in time to the right hand side of (\ref{eq:sec_case_num_flux})
\begin{align}
\notag
  \frac{1}{\Delta t^n}\int_{t^n}^{t^{n+1}}cu(x_{i+1/2},t)dt
&= \frac{1}{\Delta t^n}\int_{t^n}^{t^{n+1}}c\left[u^*+(t-t^n)u_t^*\right]dt+O((\Delta t^n)^2)\\
\label{eq:sec_math_time_ave}
&= cu^*+\frac{1}{2}c\Delta t^nu_x^*+O((\Delta t^n)^2)
\end{align}
Here the superscript $*$ indicates that the values are evaluated at $(x_{i+1/2},t^n)$.
Feeding the exact data at $t^n$ to (\ref{eq:sec_math_flux_h}) and noticing that $\bar{u}(x_j,t) = u(x_j,t)+O(\Delta x_j^2), j=i'', i''+1$, one obtains
\begin{align}
\notag
F_{i+1/2}^{H,n} =& cu^*+\frac{1}{2}\abs{c}u^*_x\left[\frac{1}{2}B_{i+1/2}\left(1-\frac{\abs{c}\Delta t^n}{\Delta x_{i''}}\right)(\Delta x_{i''}+\Delta x_{i''+1})-\Delta x_{i''}\right] \\
\label{eq:sec_math_flux_trunc}
&+O(\Delta x_{i''}^2+\Delta x_{i''+1}^2)
\end{align}
Comparing (\ref{eq:sec_math_flux_trunc}) and (\ref{eq:sec_math_time_ave}) and using $u_t^*+cu_x^*=0$, a proper value for $B_{i+1/2}$ is
\begin{equation}\label{eq:sec_math_b}
B_{i+1/2} = \frac{2\Delta x_{i''}}{\Delta x_{i''}+\Delta x_{i''+1}}
\end{equation}
By comparing (\ref{eq:sec_case_fc}) and (\ref{eq:sec_math_flux_lform}), (\ref{eq:sec_math_b}), one has the following relation between the flux limiter and slope limiter
\begin{equation}\label{eq:sec_math_fluxlim}
\varphi_{i+1/2} = \frac{\Delta x_{i''}+\Delta x_{i''+1}}{2\Delta x_{i''}}\phi_{i''}(\theta_{i''})
\end{equation}
On a uniform mesh, (\ref{eq:sec_math_fluxlim}) indicates that the flux limiters and slope limiters are equivalent to each other, which is, however, not true if the mesh is non-uniform.
A sufficient condition for second-order accuracy in space is the linear preserving property, namely $\varphi_{i+1/2}=1$ for linear data, in which case
\begin{displaymath}
\theta_i = \frac{u_i-u_{i-1}}{u_{i+1}-u_i} = \frac{\Delta x_{i-1}+\Delta x_i}{\Delta x_i+\Delta x_{i+1}}
\end{displaymath}
Thus a sufficient condition for second-order spatial accuracy is 
\begin{equation}\label{eq:sec_math_2nd}
\phi_i\left(\frac{\Delta x_{i-1}+\Delta x_i}{\Delta x_i+\Delta x_{i+1}}\right) = \frac{2\Delta x_i}{\Delta x_i+\Delta x_{i+1}},\quad\forall i
\end{equation}
It follows that $\phi_i(\cdot)$ must be defined locally, such that local grid sizes are taken into account, for non-uniform grids.
As a last note, the condition~(\ref{eq:sec_math_2nd}) is always compatible with the TVD stability condition~(\ref{eq:sec_math_tvd}) when $\Delta x_i>0,\ \forall i$. 

\subsection{The Sweby's diagram}
The Sweby's diagram~\cite{PKSweby:1984a} for limiter functions that satisfy both (\ref{eq:sec_math_tvd}) and (\ref{eq:sec_math_2nd}) is revisited to take non-uniform grids into account.
In the view of Eq.~(\ref{eq:sec_math_2nd}), one may characterize the limiter function $\phi_{A,B}(\cdot)$ by two parameters $A$ and $B$ that satisfy
\begin{equation}\label{eq:sec_math_a_b}
0<B<\min(2,2A)
\end{equation}
and
\begin{equation}\label{eq:sec_math_lim_a_b}
\phi_{A,B}(A)=B
\end{equation}
Using this notation, the conventional limiters are written as $\phi(\cdot) = \phi_{1,1}(\cdot)$, and the corresponding Sweby's diagram is shown in Figure~\ref{fg:sec_math_sweby_conv}. 
In this diagram, the admissible region of the limiter function is shaded.
This region is bounded by four straight edges: (1) $\phi(\theta)=2$, (2) $\phi(\theta)=2\theta$, (3) $\phi(\theta) = 1$, and (4) $\phi(\theta)=\theta$.
\begin{figure}\centering
\begin{subfigure}[b]{.48\textwidth}\centering
    \includegraphics[trim=0.5in 0.2in 0.5in 0.5in, clip, width=\textwidth]{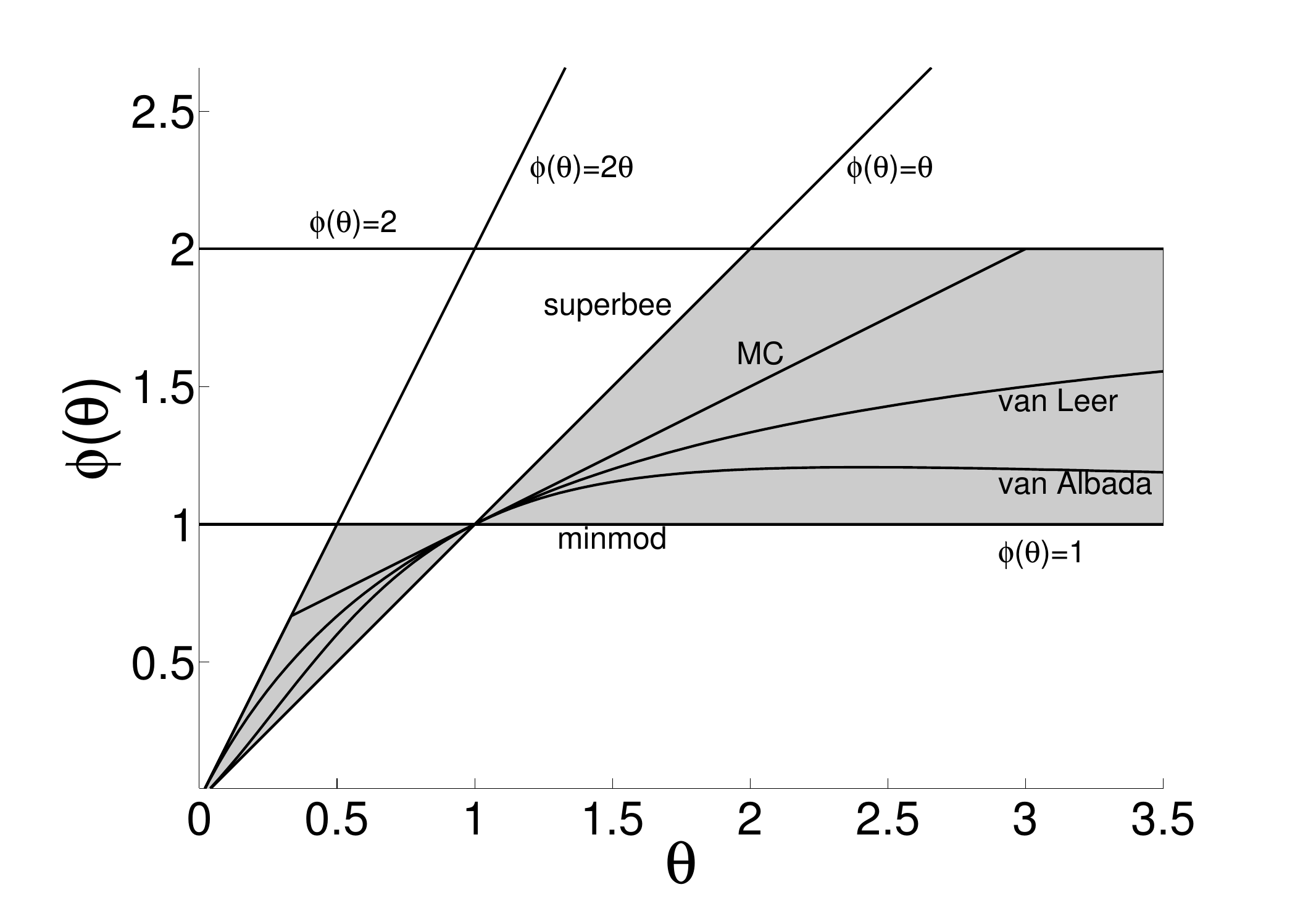}
    \vglue -0.05 truein
    \caption{Conventional limiters}
    \label{fg:sec_math_sweby_conv}
\end{subfigure}
\begin{subfigure}[b]{.48\textwidth}\centering
    \includegraphics[trim=0.5in 0.2in 0.5in 0.5in, clip, width=\textwidth]{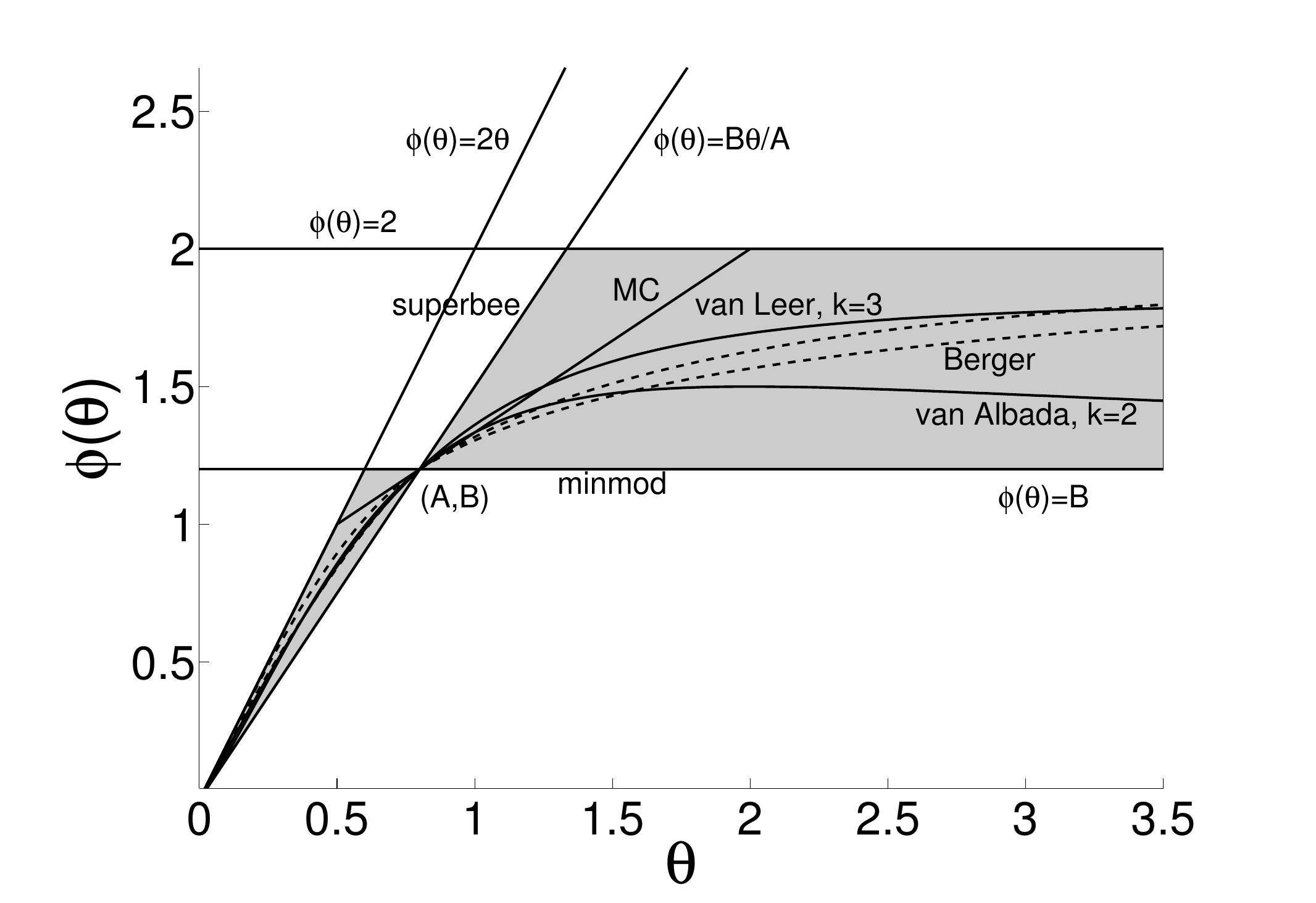}
    \vglue -0.05 truein
    \caption{Enhanced limiters}
    \label{fg:sec_math_sweby_modf}
\end{subfigure}
\vglue -0.2 truein
\caption{Sweby's diagrams for limiter functions: (\ref{fg:sec_math_sweby_conv}) conventional limiters $\phi(\cdot)=\phi_{1,1}(\cdot)$, (\ref{fg:sec_math_sweby_modf}) enhanced limiters $\phi_{A,B}(\cdot)$}
\label{fg:sec_math_sweby}
\end{figure}
The first two edges correspond to the conventional TVD stability conditions~(\ref{eq:sec_case_tvd}); 
and the other two edges lead to the Lax-Wendroff method~\cite{PLax:1960a} and the Beam-Warming method~\cite{RFWarming:1976a}, which correspond to applying using the following ``slope limiters'' on uniform meshes
\begin{displaymath}
\phi^{LW}(\theta) = 
\left\{\begin{array}{lcc}
1 & & c>0 \\
\\
\theta & & c<0
\end{array}\right.,\qquad
\phi^{BW}(\theta) = 
\left\{\begin{array}{lcc}
\theta & & c>0 \\
\\
1 & & c<0
\end{array}\right.
\end{displaymath}
Thus one may interpret the Sweby's diagram such that the limiter function should be a convex combination of $\phi^{LW}$ and $\phi^{BW}$, in addition to satisfying~(\ref{eq:sec_case_tvd}) and (\ref{eq:sec_case_2nd}).

This interpretation may be extended to the limiter function $\phi_{A,B}$ as follows.
First, the limiter function $\phi^{LW}$ and $\phi^{BW}$ are generalized to satisfy~(\ref{eq:sec_math_lim_a_b})
\begin{displaymath}
\phi^{LW}_{A,B}(\theta) = 
\left\{\begin{array}{lcc}
B & & c>0 \\
\\
\frac{B\theta}{A} & & c<0
\end{array}\right.,\quad
\phi^{BW}(\theta) = 
\left\{\begin{array}{lcc}
\frac{B\theta}{A} & & c>0 \\
\\
B & & c<0 
\end{array}\right.
\end{displaymath}
Next, the shaded admissible region is surrounded by the four edges: (1) $\phi_{A,B}(\theta)=2$, (2) $\phi_{A,B}(\theta)=2\theta$, (3) $\phi_{A,B}(\theta)=B$, and (4) $\phi_{A,B}(\theta)=B\theta/A$ (Figure~\ref{fg:sec_math_sweby_modf}).

{\bf Remark:} by introducing $A$ and $B$ that satisfy~(\ref{eq:sec_math_a_b}), one generalizes one particular limiter function, such as $\phi^{\it \textrm{minmod}}$, to a class of limiter functions $\phi_{A,B}^{\it \textrm{minmod}}$ (see Section~\ref{sec:limiter}).
This convention is used throughout the remainder of the paper.

\subsection{Preserving symmetric solutions}
Another desired property of slope limiters is to preserving symmetric solutions.
This property, however, need to be elaborated for non-uniform meshes first.
In particular, it is defined for {\it a class of} limiter functions $\phi_{A,B}(\cdot)$ rather than a single one.
\begin{definition}[Symmetry-preserving]\label{def:sec_math_sym}
A class of limiter functions $\phi_{A,B}(\cdot)$ preserves symmetric solutions, if the next two problems lead to $u_i^{n+1} = v_i^{n+1}, \forall i$.
{\small
\begin{enumerate}
\item[] Problem 1: given the grid with cell sizes $\Delta x_{i+k}^u$ and corresponding data at $t^n$: $u_{i+k}^n,\ k=0,\pm1,\cdots$, solving $u_t+cu_x = 0$ for $u_i^{n+1}$ using the limiter functions $\phi_{A,B}(\cdot)$ such that 
\begin{displaymath}
\phi_{j}^u(\cdot) = \phi_{A^u_{j},B^u_{j}}(\cdot),\quad
A^u_{j}=\frac{\Delta x_{j-1}^u+\Delta x_j^u}{\Delta x_j^u+\Delta x_{j+1}^u},\ 
B^u_{j}=\frac{2\Delta x_j^u}{\Delta x_j^u+\Delta x_{j+1}^u}
\end{displaymath}
\item[] Problem 2: given the grid with cell sizes $\Delta x_{i+k}^v$ and corresponding data at $t^n$: $v_{i+k}^n,\ k=0,\pm1,\cdots$, such that 
\begin{equation}\label{eq:sec_math_sym_data}
\Delta x_{i+k}^v=\Delta x_{i-k}^u,\qquad v_{i+k}^n=u_{i-k}^n
\end{equation}
 solving $v_t-cv_x = 0$ for $v_i^{n+1}$ using the limiter functions $\phi_{A,B}(\cdot)$ satisfying 
\begin{displaymath}
\phi_{j}^v(\cdot) = \phi_{A^v_{j},B^v_{j}}(\cdot),\quad
A^v_{j}=\frac{\Delta x_{j-1}^v+\Delta x_j^v}{\Delta x_j^v+\Delta x_{j+1}^v},\ 
B^v_{j}=\frac{2\Delta x_j^v}{\Delta x_j^v+\Delta x_{j+1}^v}
\end{displaymath}
\end{enumerate}
}
\end{definition}
Here the superscript $u$ or $v$ designates one of the two problems to which a particular quantity is relevant.
Figure~\ref{fg:sec_math_sym_nunf} demonstrates this definition in the case $c>0$.
\begin{figure}\centering
\begin{subfigure}[b]{0.48\textwidth}\centering
    \includegraphics[trim=0.5in 0.5in 0.5in 0.1in, clip, width=\textwidth]{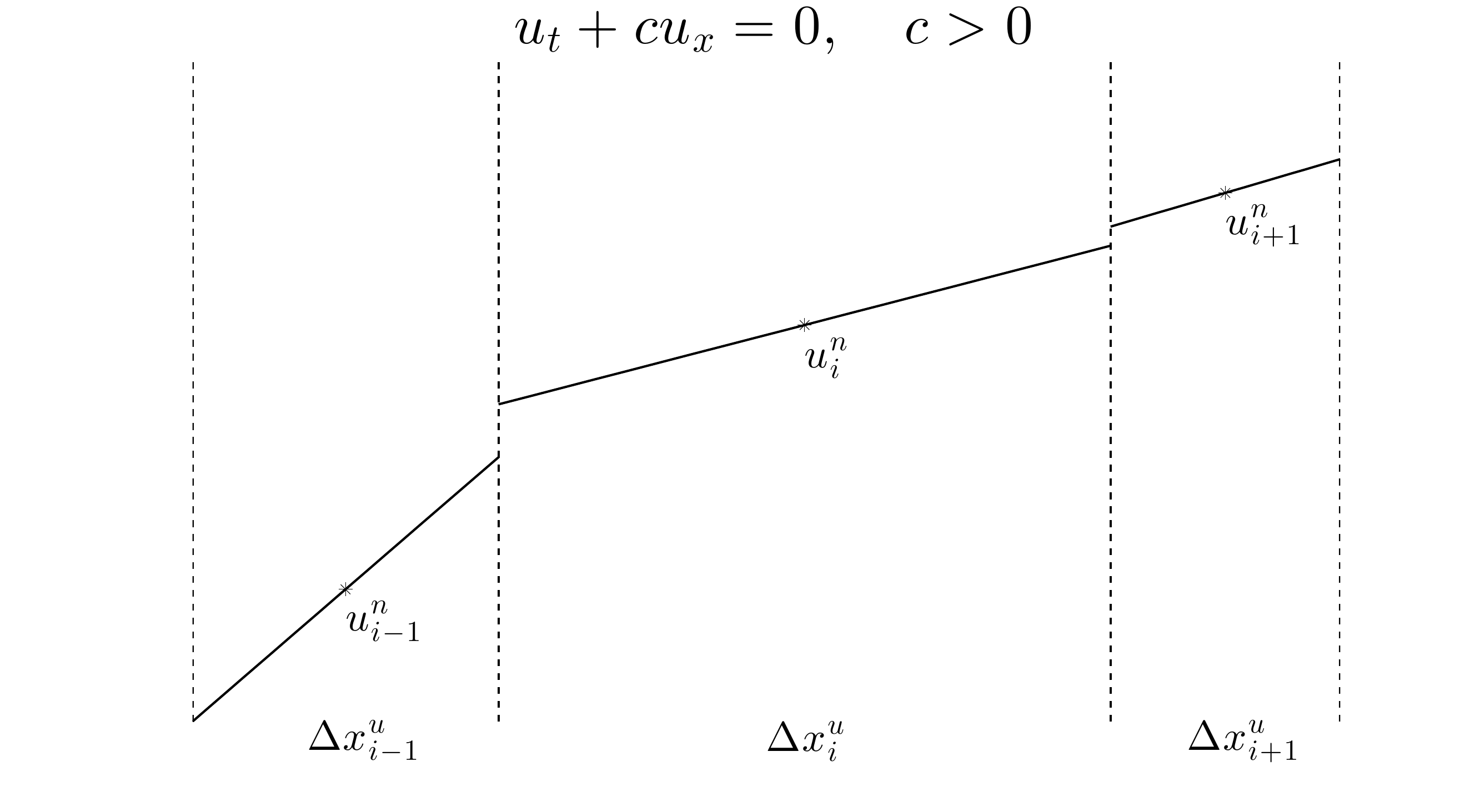}
    \vglue -0.05 truein
    \caption{Problem 1}
    \label{fg:sec_math_sym_nunf_u}
\end{subfigure}
\begin{subfigure}[b]{0.48\textwidth}\centering
    \includegraphics[trim=0.5in 0.5in 0.5in 0.1in, clip, width=\textwidth]{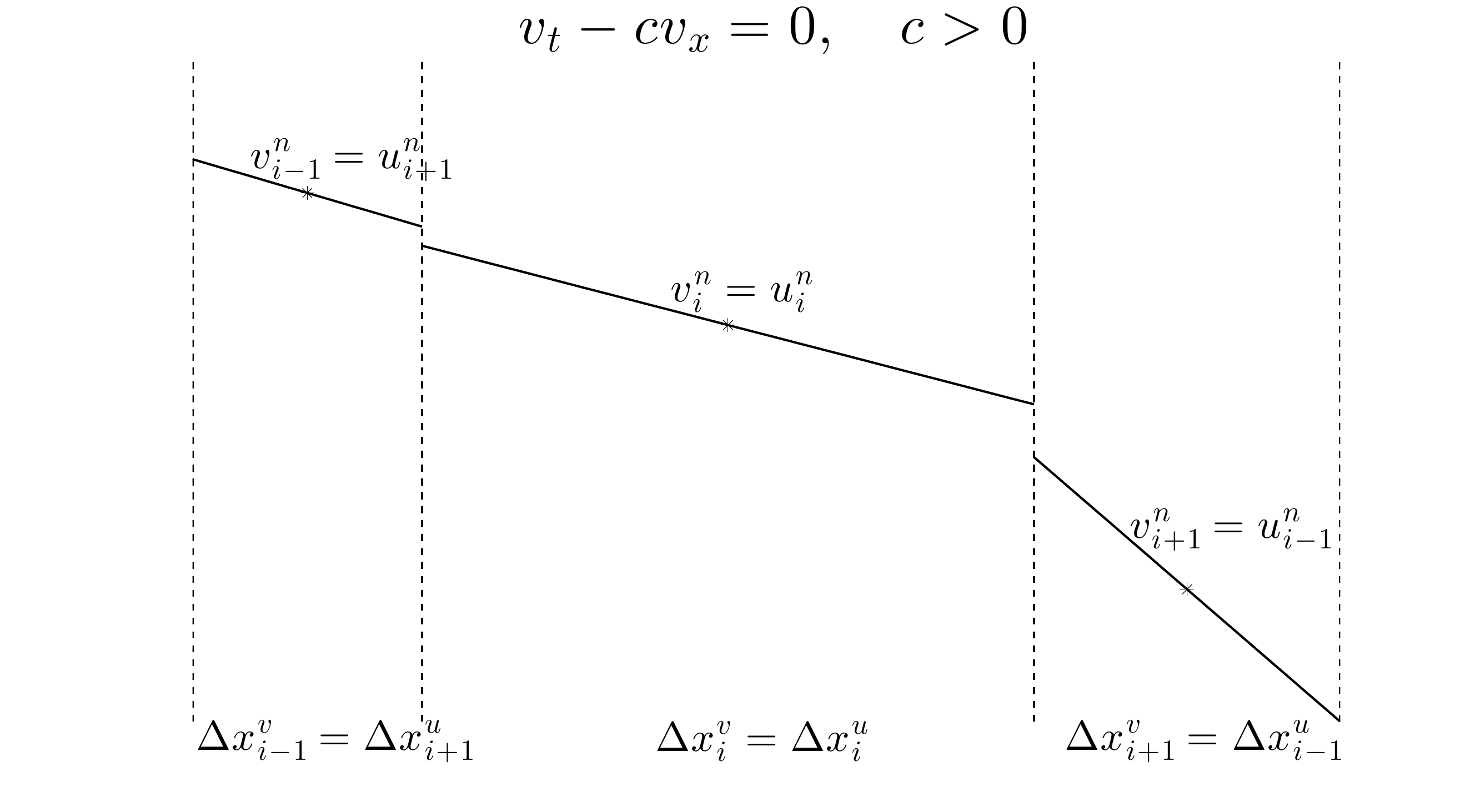}
    \vglue -0.05 truein
    \caption{Problem 2}
    \label{fg:sec_math_sym_nunf_v}
\end{subfigure}
\vglue -0.2 truein
\caption{Symmetry-preserving property on non-uniform meshes: (\ref{fg:sec_math_sym_nunf_u}) problem 1, and (\ref{fg:sec_math_sym_nunf_v}) problem 2}
\label{fg:sec_math_sym_nunf}
\end{figure}
Supposing $c>0$, $u_i^{n+1}=v_i^{n+1}$ is equivalent to
\begin{align}
\notag
 &u_i^n-\lambda_i^u(u_i^n-u_{i-1}^n)-\\
\label{eq:sec_math_sym_eqn}
 &\frac{1}{2}\lambda_i^u\left[(1-\lambda_i^u)\phi_i^u(\theta_i^u)(u_{i+1}^n-u_i^n)-(1-\lambda_{i-1}^u)\phi_{i-1}^u(\theta_{i-1}^u)(u_i^n-u_{i-1}^n)\right]\\
\notag
=&u_i^n-\lambda_i^u(u_i^n-u_{i-1}^n)-\\
\notag
 &\frac{1}{2}\lambda_i^u\left[(1-\lambda_i^u)\phi_i^v\left(\frac{1}{\theta_i^u}\right)\theta_i^u(u_{i+1}^n-u_{i}^n)-(1-\lambda_{i-1}^u)\phi_{i+1}^v\left(\frac{1}{\theta_{i-1}^u}\right)\theta_{i-1}^u(u_{i}^n-u_{i-1}^n)\right]
\end{align}
Here the following identities are utilized
\begin{displaymath}
\textrm{(\ref{eq:sec_math_sym_data})}\quad
\Rightarrow\quad
\lambda_{i+k}^v = \lambda_{i-k}^u,\quad 
\theta_{i+k}^v = \frac{1}{\theta_{i-k}^u},\quad 
A_{i+k}^v = \frac{1}{A_{i-k}^u},\quad
B_{i+k}^v = \frac{B_{i-k}^u}{A_{i-k}^u}
\end{displaymath}

Clearly, a sufficient condition for Eq.~(\ref{eq:sec_math_sym_eqn}) to hold is 
\begin{displaymath}
\frac{\phi_{i-k}^u(\theta)}{\theta} = \phi_{i+k}^v\left(\frac{1}{\theta}\right), k=0,1
\end{displaymath}
which is equivalent to saying
\begin{equation}\label{eq:sec_math_sym_cond}
\frac{\phi_{A^u_{i-k},B^u_{i-k}}(\theta)}{\theta} =
\phi_{A^v_{i+k},B^v_{i+k}}\left(\frac{1}{\theta}\right) =
\phi_{1/A^u_{i-k},B^u_{i-k}/A^u_{i-k}}\left(\frac{1}{\theta}\right)
\end{equation}
\begin{definition}[Conjugate limiter function]\label{def:sec_math_conjg}
Given a class of limiter functions $\phi_{A,B}$ characterized by the two parameters $A$ and $B$ as before, the conjugate class of limiter functions $\phi^*_{A,B}$ is defined by
\begin{equation}\label{eq:sec_math_sym_conjg}
\phi^*_{A,B} = \phi_{1/A,B/A}
\end{equation}
\end{definition}
It is easy to verify that $\phi^*$ is well-defined, in the sense that:(1) for any $(A,B)$ satisfying (\ref{eq:sec_math_a_b}), the pair $(1/A,B/A)$ does so as well, (2) $\phi^{**}_{A,B} = \phi_{A,B}$ for any pair $(A,B)$.
In the view of Eq.~(\ref{eq:sec_math_sym_cond}), a sufficient condition for the symmetry-preserving property is 
\begin{equation}\label{eq:sec_math_sym}
\frac{\phi_{A,B}(\theta)}{\theta} = \phi^*_{A,B}\left(\frac{1}{\theta}\right)
\end{equation}
In the case of conventional limiter functions $\phi=\phi_{1,1}$, there is $\phi^*=\phi$, and Eq.~(\ref{eq:sec_math_sym}) reduces to the classical symmetry-preserving condition, namely $\phi(\theta)/\theta=\phi(1/\theta)$.

{\bf Remark 1:} M. Berger~\cite{MBerger:2005a} proposes an alternative theory to study the symmetry-preserving property, which writes the limiters as functions of a symmetry variable instead of the classical smooth monitors.
This strategy is not explored here.

{\bf Remark 2:} as a final remark of this section, although the analysis remains similar if alternative slopes and smoothness monitors are used, the conclusions may be different.
For example, one may use an alternative strategy to apply the conventional limiters on irregular meshes without destroying the formal second-order accuracy, but losing the TVD stability, as will be discussed in more detail in Section~\ref{sec:alt}.

\section{Enhanced limiter functions}
\label{sec:limiter}
This section proposes enhancements to the limiters listed in Table~\ref{tb:sec_case_lim} such that the improved limiters satisfy (\ref{eq:sec_math_tvd}), (\ref{eq:sec_math_2nd}), and (\ref{eq:sec_math_sym}).
In particular, the enhanced version of each limiter function $\phi^{\it \textrm{name}}$ is named $\phi^{\it \textrm{name}}_{A,B}$.
Examples of the enhanced limiters are plotted in Figure~\ref{fg:sec_math_sweby_modf}, among which the three piecewise linear limiters are constructed naturally, as listed in Table~\ref{tb:sec_limiter_lin}.
\begin{table}\centering
\caption{Enhanced limiter functions $\phi^{\it \textrm{minmod}}_{A,B}, \phi^{\it \textrm{superbee}}_{A,B}$, and $\phi^{\it \textrm{MC}}_{A,B}$}
\label{tb:sec_limiter_lin}
{\small
\begin{tabular}{|c||c|c|c|}
\hline
Enhanced Limiter        
& minmod                   
& superbee                                 
& MC \\ \hline
$\phi_{A,B}(\theta)$ 
& $\frac{B}{A}\min(\theta,A)^+$ 
& $\max(\min(2\theta,B),\min(\frac{B\theta}{A},2))^+$
& $\min(2\theta,\frac{B}{A+1}(\theta+1),2)^+$ \\ \hline
\end{tabular} 
}
\end{table}
Verifying that these enhanced limiters satisfy the three conditions is straight forward.

Constructing the enhanced {\it van Leer} and {\it van Albada} limiters is more elaborated, as described below.

\subsection{Enhanced {\it \textbf{van Leer}} limiter}
The reference~\cite{MBerger:2005a} proposes two generalized {\it van Leer} limiters using symmetry variables, which may be rewritten as
\begin{equation}\label{eq:sec_lim_berger_1}
\phi^{\it \textrm{Berger-1}}_{A,B}(\theta) = \left\{\begin{array}{lcl}
2\theta\left[1-\left(1-\frac{B}{2A}\right)\left[\frac{\theta}{\theta+1}\cdot\frac{A+1}{A}\right]^{\frac{B}{2A-B}}\right] & & \theta\le A \\
		\\
2\left[1-\left(1-\frac{B}{2}\right)\left[\frac{A+1}{\theta+1}\right]^{\frac{B}{2-B}}\right] & & \theta> A
\end{array}\right.
\end{equation}
\begin{equation}\label{eq:sec_lim_berger_2}
\phi^{\it \textrm{Berger-2}}_{A,B}(\theta) = \left\{\begin{array}{lcl}
\frac{B(\theta+1)}{A+1}\left[1-\left[1-\frac{\theta}{\theta+1}\cdot\frac{A+1}{A}\right]^{\frac{2A}{B}}\right] & & \theta\le A \\
\\
\frac{B(\theta+1)}{A+1}\left[1-\left[1-\frac{A+1}{\theta+1}\right]^{\frac{2}{B}}\right] & & \theta> A \\
\end{array}\right.
\end{equation}
Both limiters reduce to the conventional one in the case $A=B=1$; 
and both functions satisfy the TVD stability condition, order condition and symmetry-preserving condition presented in this paper.
They are plotted in Figure~\ref{fg:sec_math_sweby_modf} by dotted lines.

A drawback of these two modifications is that they are not smooth at the point $\theta=A$; 
on the contrary, a smooth enhancement $\phi^{\it \textrm{van Leer}}_{A,B}$ is presented below.

Noticing that $\lim_{k\to\infty}\sum_{l=1}^kA^l/\sum_{l=0}^kA^l = \min(1,A)$, Eq.~(\ref{eq:sec_math_a_b}) suggests the existence of an integer $k>0$ such that
\begin{equation}\label{eq:sec_lim_vanleer_k}
B \le \frac{2\sum_{l=1}^kA^l}{\sum_{l=0}^kA^l}
\end{equation}
To this end, the enhanced {\it van Leer} limiter is defined as
\begin{equation}\label{eq:sec_lim_vanleer}
\phi^{\it \textrm{van Leer}}_{A,B}(\theta) = 
\frac{B\sum_{l=1}^k\theta^l}{\sum_{l=0}^k\theta^l}\cdot
\frac{\sum_{l=0}^kA^l}{\sum_{l=1}^kA^l}\quad\textrm{if}\quad\theta\ge0,\quad\textrm{and}\quad0\quad\textrm{o.w.}
\end{equation}
Clearly $\phi_{A,B}^{\it \textrm{van Leer}}(A)=B$; and it is easy to verify the TVD stability condition as follows: given $\theta\ge0$ and using (\ref{eq:sec_lim_vanleer_k}) 
\begin{displaymath}
\phi_{A,B}^{\it \textrm{van Leer}}\le\frac{2\sum_{l=1}^k\theta^l}{\sum_{l=0}^k\theta^l}<2\min(1,\theta)
\end{displaymath}
Verifying the symmetry-preserving condition~(\ref{eq:sec_math_sym}) is straight forward and omitted.

A sample curve calculated with $k=3$ is plotted in Figure~\ref{fg:sec_math_sweby_modf}. 
And clearly, if $A=B=1$, Eq.~(\ref{eq:sec_lim_vanleer_k}) is satisfied by $k=1$, in which case the enhanced limiter Eq.~(\ref{eq:sec_lim_vanleer}) reduces to the conventional {\it van Leer} limiter.

{\bf Remark:}
This particular limiter can be used as a good example to show that, the issue raised in this paper cannot always be resolved by choosing a different smoothness monitor and at the same time applying the conventional limiter function. 
For example, see next section and Appendix~\ref{app:vanleer}.

\subsection{Enhanced {\it \textbf{van Albada}} limiter}
The enhanced {\it van Albada} limiter proposed here is a smooth function of the smoothness monitor, and it is based on the inequality $\lim_{k\to\infty}(k-1)^{k-1}/k^k=0$.
Thus given $A$ and $B$ satisfying (\ref{eq:sec_math_a_b}), there always exists $k$ such that
\begin{equation}\label{eq:sec_lim_vanalbada_k}
B \le 2\left(1+\frac{(k-1)^{k-1}}{k^k}\right)^{-1}\min(A,1)
\end{equation}
Then the limiter function is defined as
\begin{equation}\label{eq:sec_lim_vanalbada}
\phi_{A,B}^{\it \textrm{van Albada}}(\theta) = \frac{B(\theta^k+\theta)}{\theta^k+A}
\end{equation}
Verifying the conditions~(\ref{eq:sec_math_2nd}) and (\ref{eq:sec_math_sym}) is straight forward.
Now considering the TVD stability condition~(\ref{eq:sec_math_tvd}), the following inequalities are utilized
\begin{displaymath}
\theta 
\le 
k\left(\frac{2}{B}-1\right)^{\frac{1}{k}}\left(\frac{2A/B}{k-1}\right)^{\frac{k-1}{k}}\theta
\le
\left(\frac{2}{B}-1\right)\theta^k+\frac{(k-1)\cdot2A}{(k-1)B}
= \left(\frac{2}{B}-1\right)\theta^k+\frac{2A}{B}
\end{displaymath}
and similarly
\begin{displaymath}
\frac{1}{\theta}\le 
k\left(\frac{2A}{B}-1\right)^{\frac{1}{k}}\left(\frac{2/B}{k-1}\right)^{\frac{k-1}{k}}\frac{1}{\theta}
\le
\left(\frac{2A}{B}-1\right)\frac{1}{\theta^k}+\frac{(k-1)\cdot2}{(k-1)B}
= \left(\frac{2A}{B}-1\right)\frac{1}{\theta^k}+\frac{2}{B}
\end{displaymath}
An example with $k=2$ is plotted in Figure~\ref{fg:sec_math_sweby_modf}.
Furthermore, the enhanced limiter Eq.~(\ref{eq:sec_lim_vanalbada}) coincides with the conventional one by setting $k=1$ and $A=B=1$.

{\bf Remark:} in practice, evaluating the right hand side of Eq.~(\ref{eq:sec_lim_vanalbada_k}) can be expensive.
An economical version may be derived using the fact that $(k-1)^{k-1}/k^k \le 1/k$, thus one may choose $k$ such that
\begin{displaymath}
B\le2\left(1+\frac{1}{k}\right)^{-1}\min(1,A)
\end{displaymath}
This approach is adopted in all the numerical examples in this paper.

\section{1D examples and comparison with alternative strategies}
\label{sec:alt}
The analysis so far supposes the particular numerical slope~(\ref{eq:sec_case_slpe}) and the smoothness monitor~(\ref{eq:sec_case_monitor}).
In practice, one may choose different strategies to handle non-uniform meshes with conventional limiters and still observe reasonably good results.
This section discusses two such alternatives, namely MUSCL-MOL using consistent numerical slope and the capacity-form differencing~\cite{RJLeVeque:2002a}, and compares them with the proposed method using one-dimensional examples.

\subsection{MUSCL-MOL using consistent slopes}
\label{sec:alt_consist}
One way to explain why using (\ref{eq:sec_case_slpe}) with conventional limiters leads to the loss of second-order accuracy is that when the mesh is highly non-uniform, the unlimited $\sigma_i$ is not a consistent approximation to the solution's slope, that is
\begin{displaymath}
D_xu_i = \frac{1}{\Delta x_i}(u_{i+1}-u_i)\approx \frac{\Delta x_i+\Delta x_{i+1}}{2\Delta x_i}u_x|_{x=x_i}+O(h)
\end{displaymath}
here $h$ is the reference cell size of the mesh.
Based on this observation, an obvious alternative to (\ref{eq:sec_case_slpe}) is to use a consistent numerical slope, such as
\begin{equation}\label{eq:alt_consist_slpe}
D_xu_i = \frac{u_{i+1}-u_i}{(\Delta x_i+\Delta x_{i+1})/2}
\end{equation}
The MUSCL equipped with this slope and conventional limiters is second-order accurate in space for smooth problems.
In fact, using the piecewise linear limiters $\phi^{\it \textrm{minmod}}$, $\phi^{\it \textrm{superbee}}$, and $\phi^{\it \textrm{MC}}$, the resulting method is exactly the same as the one obtained by using (\ref{eq:sec_case_slpe}) and the enhanced limiters in Section~\ref{sec:limiter} for linear advection equations.

However, the two strategies are in general different if nonlinear limiters are used.
In particular, using (\ref{eq:alt_consist_slpe}) with conventional {\it van Leer} limiter leads to the loss the TVD stability property, as demonstrated by the subsequent example.

Considering the following advection problem
\begin{displaymath}
u_t + u_x = 0,\qquad x\in [0,2]
\end{displaymath}
with periodic boundary condition and the initial data
\begin{displaymath}
u(x,0) = 100\ \textrm{ if }\ 0.75\le x\le 1.25,\ \textrm{ and }\ 0\ \textrm{ o.w. }
\end{displaymath}
At $t=2.0$, the exact solution is the same as the initial data.
Using $200$ cells and $r=0.4$ to mesh the domain, Figure~\ref{fg:sec_alt_slpe} presents the solutions obtained by: (a) slope~(\ref{eq:sec_case_slpe}) and $\phi^{\it \textrm{van Leer}}$, (b) slope~(\ref{eq:alt_consist_slpe}) and $\phi^{\it \textrm{van Leer}}$, and (c) slope~(\ref{eq:sec_case_slpe}) and $\phi^{\it \textrm{van Leer}}_{A,B}$.
The Rusanov flux, TVD RK2, and Courant number $0.8$ are used in all three tests.
\begin{figure}\centering
\begin{subfigure}[b]{0.32\textwidth}\centering
\includegraphics[trim=0.8in 0.1in 1.1in 0.0in, clip, width=\textwidth]{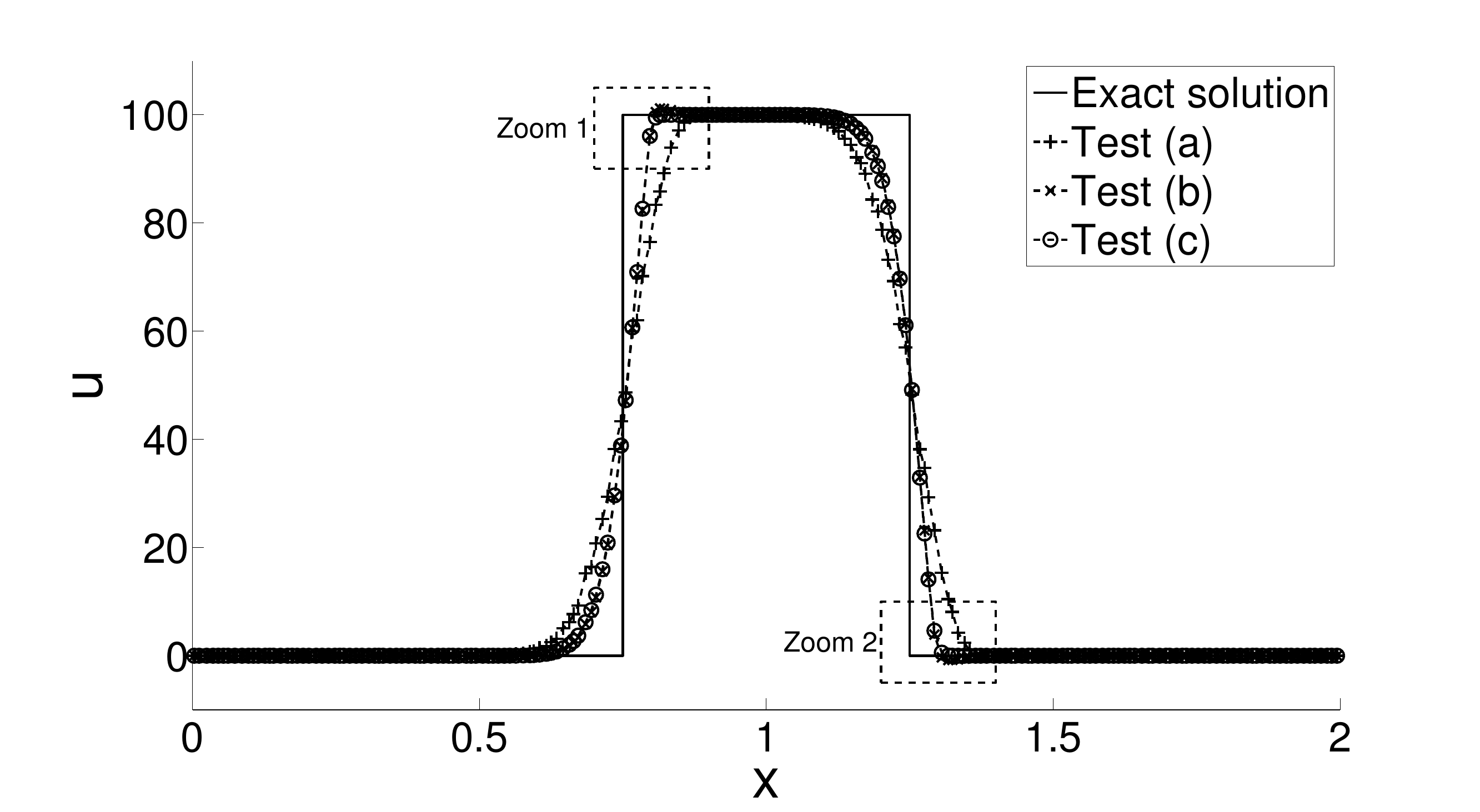}
\vglue -0.05 truein
\caption{Global view}
\label{fg:sec_alt_slpe_full}
\end{subfigure}
\begin{subfigure}[b]{0.32\textwidth}\centering
\includegraphics[trim=0.8in 0.1in 1.1in 0.0in, clip, width=\textwidth]{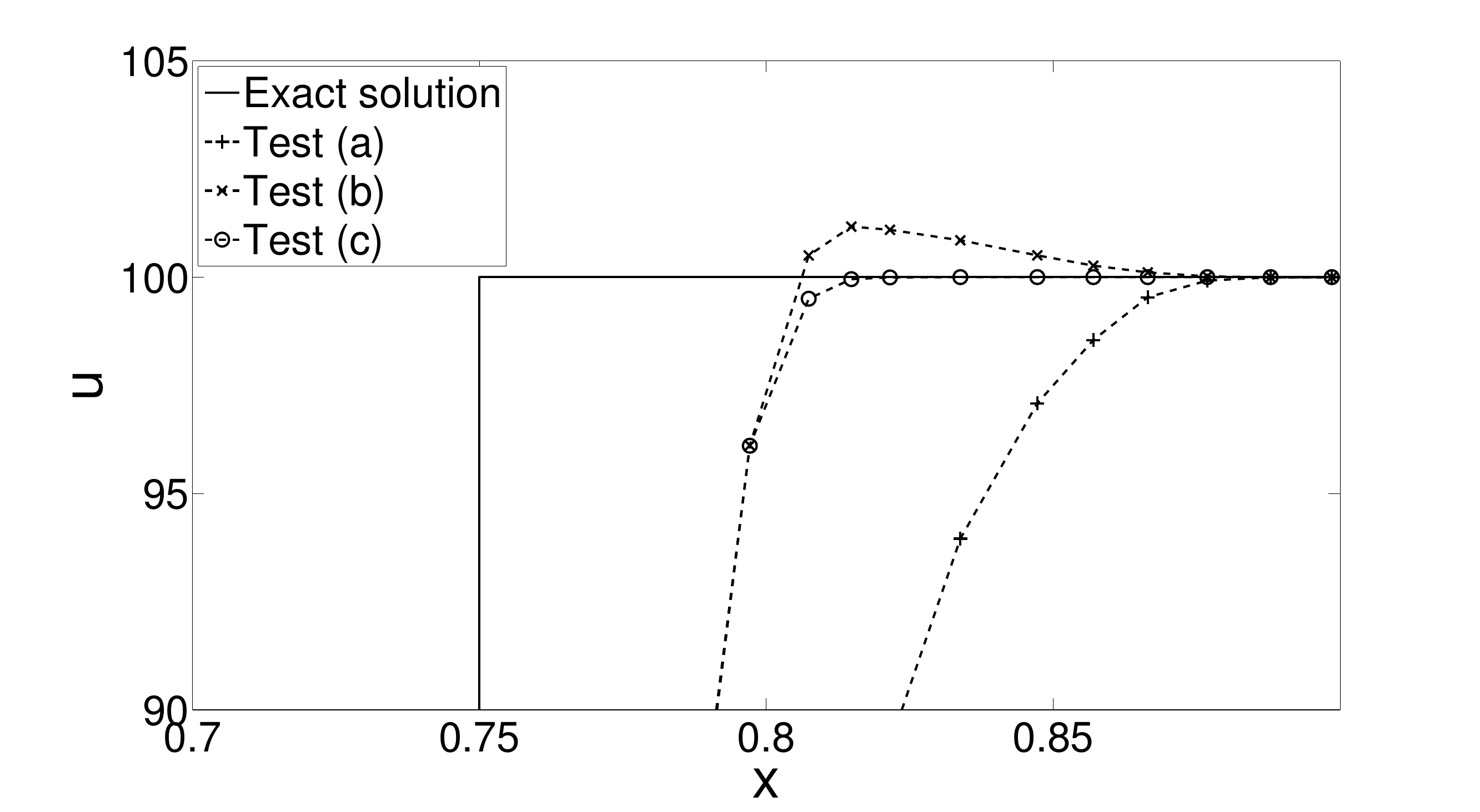}
\vglue -0.05 truein
\caption{Local view: zoom 1}
\label{fg:sec_alt_slpe_zoom_1}
\end{subfigure}
\begin{subfigure}[b]{0.32\textwidth}\centering
\includegraphics[trim=1.1in 0.1in 1.1in 0.0in, clip, width=\textwidth]{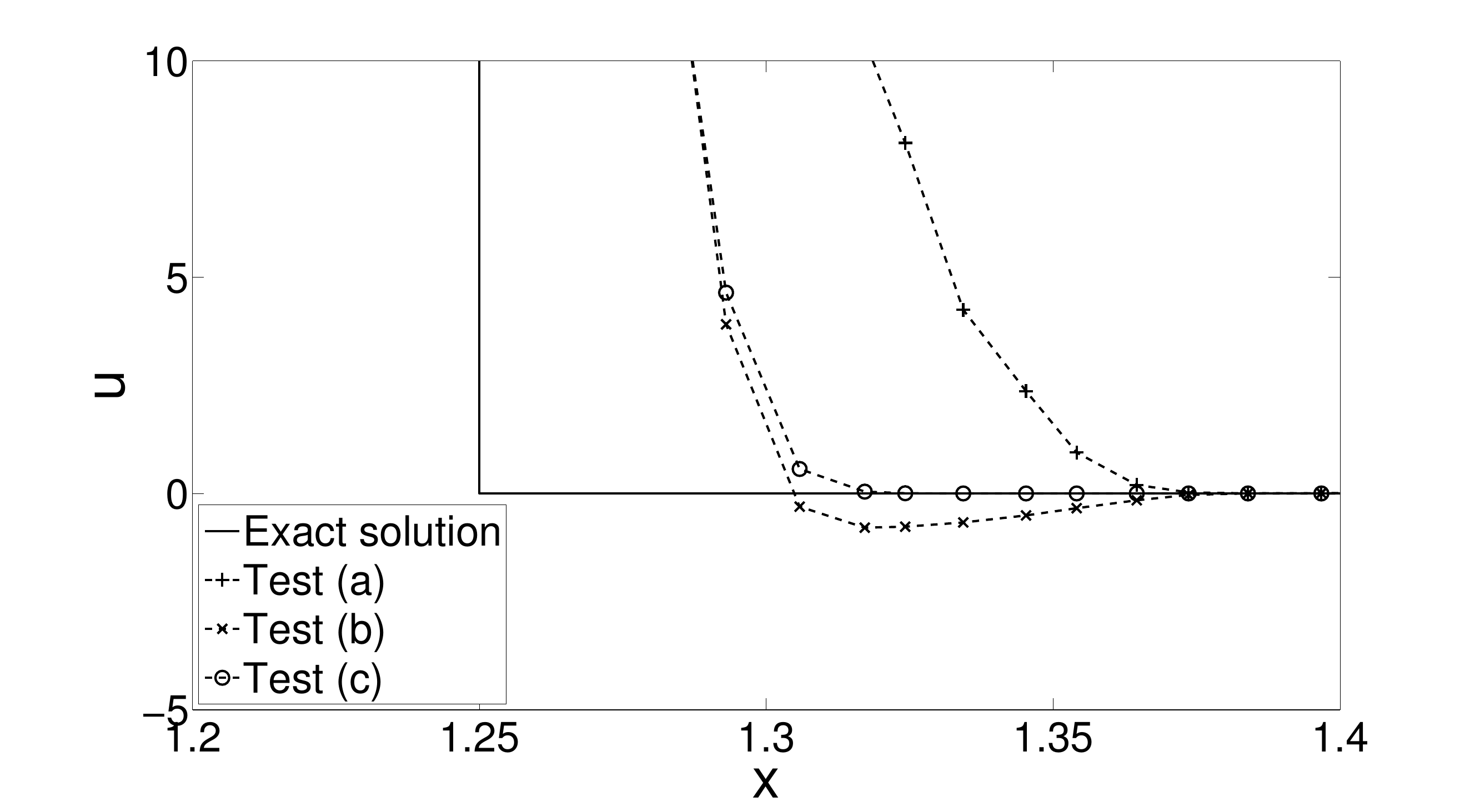}
\vglue -0.05 truein
\caption{Local view: zoom 2}
\label{fg:sec_alt_slpe_zoom_2}
\end{subfigure}
\vglue -0.2 truein
\caption{Numerical solutions at $t=2.0$ in (\ref{fg:sec_alt_slpe_full}) global view, and (\ref{fg:sec_alt_slpe_zoom_1}--\ref{fg:sec_alt_slpe_zoom_2}) local views, by: 
    test (a) -- slope~(\ref{eq:sec_case_slpe}) and $\phi^{\it \textrm{van Leer}}$, 
    test (b) -- slope~(\ref{eq:alt_consist_slpe}) and $\phi^{\it \textrm{van Leer}}$, 
    test (c) -- slope~(\ref{eq:sec_case_slpe}) and $\phi^{\it \textrm{van Leer}}_{A,B}$.}
\label{fg:sec_alt_slpe}
\end{figure}

On the one hand, the tests (b) and (c) show similar dispersion and dissipation properties, indicating that both are second-order accurate in space.
On the other hand, the alternative strategy (test (b)) shows both overshoot and undershoot near the jumps, whereas tests (a) and (c) exhibit TVD property as expected.
In fact, for a large class of possible alternative strategies, applying the conventional {\it van Leer} limiter leads to the loss of either second-order accuracy or TVD stability, as proved in Appendix~\ref{app:vanleer}.

Using the slope~(\ref{eq:alt_consist_slpe}) with the conventional {\it van Albada} limiter show similar tendency to produce overshoots and undershoots on highly irregular grids.
But the problem is much less severe than the {\it van Leer} limiter, because $\phi^{\it \textrm{van Albada}}$ is designed such that it stays away from the TVD stability bounds (see Figure~\ref{fg:sec_math_sweby}).
In particular, the magnitude of undershoots is about $0.04$ in absolute value by solving the same problem using this strategy on a extremely irregular mesh generated by $r=0.4995$.
Note that Section~\ref{sec:example} contains a 2D example showing non-physical solution that are caused by using the conventional {\it van Albada} limiter with a formal second-order accurate method on non-uniform grids.

\subsection{Capacity-form differencing}
\label{sec:alt_cap}
Another popular strategy to handle non-uniform rectilinear grids is to solve the capacity form equation~\cite{RJLeVeque:1997a, RJLeVeque:2002a, DACalhoun:2008a}.
Unlike the method of lines considered in this paper, the capacity-form differencing incorporates the time-integration explicitly in constructing the numerical flux, and it behaves very differently in one- and two-dimensional cases.

In one space dimension, let $\xi = \xi(x)$ be an increasing continuously differentiable function such that $\xi_i\eqdef\xi(x_i) = i\Delta\xi$.
Then the equation~(\ref{eq:sec_case_1dcl}) is rewritten as
\begin{equation}\label{eq:sec_cap_eqn}
\kappa(\xi)u_t + f(u)_\xi = 0
\end{equation}
where $\kappa(\xi) = x'(\xi)>0$.
Integrating (\ref{eq:sec_cap_eqn}) over a space-time slab $\Omega_i\times[t^n,t^{n+1}]$, and approximating $\kappa|_{\Omega_i}$ by $\kappa_i\eqdef\Delta x_i/\Delta\xi$, one arrives at
\begin{equation}\label{eq:sec_cap_full}
\frac{u^{n+1}-u^n}{\Delta t^n} + \frac{F_{i+1/2}^n-F_{i-1/2}^n}{\kappa_i\Delta\xi} = 0
\end{equation}
Here $F_{i+1/2}^n$ is a numerical approximation to the time-averaged flux across $\xi_{i+1/2}=\xi(x_{i+1/2})$; this flux explicitly incorporates $\Delta t^n$ in its high-resolution version, and it uses the conventional limiters to enhance nonlinear stability of the resulting scheme.
In principal, $F_{i+1/2}$ should be constructed using Riemann solvers to (\ref{eq:sec_cap_eqn}) rather than the original conservation law; but in practice, the latter is always used.
As explained in~\cite{RJLeVeque:2002a}, this strategy does not cause accuracy issue when $x(\xi)$ is sufficiently smooth.
When the map $x(\xi)$ is not smooth, such as the random meshes considered in this paper, the second-order accuracy is lost in the 1D case, as demonstrated by the example in Section~\ref{sec:alt_accuracy}.
It is difficult to say whether this reduction in accuracy is due to the usage of the wrong approximate Riemann solver or the limiter function that is designed for uniform grids; and an analysis of the capacity-form differencing method is beyond the scope of this paper.

The 2D capacity-form differencing method~\cite{RJLeVeque:2002a, DACalhoun:2008a}, however, behaves very differently from the 1D one.
In particular, because it is a single-stage method, it cannot rely on the multiple Runge-Kutta stages to account for contributions between diagonally adjacent cells (such as fluxes between $\Omega_{i,j}$ and $\Omega_{i\pm1,j\pm1}$).
These transverse contributions must be explicitly included in the numerical fluxes for high-order methods; this is a situation that does not appear for 1D problems.
In the author's experience, these transverse fluxes are the crucial components to achieve formally second-order on both uniform and non-uniform grids.
Furthermore, even on highly non-uniform meshes, the effects of these fluxes seem to counteract the reduced accuracy caused by using the original Riemann solver and/or conventional limiters and lead to a formally second-order method, see Section~\ref{sec:example} for more details.

\subsection{Accuracy test of the enhanced limiters}
\label{sec:alt_accuracy}
This example concerns about solving the problem in Section~\ref{sec:case} using the capacity-form differencing and the proposed MUSCL-MOL (i.e., the test (c) in Section~\ref{sec:alt_consist} with possibly different enhanced limiters).
The same meshes ($r=0.2\ \textrm{or}\ 0.3$), numerical flux (Roe flux), and Courant number ($0.6$) as in previous test are used here.
Tables~\ref{tb:sec_alt_cap}--\ref{tb:sec_alt_mod} summarize the $L_1$-errors and convergence rates obtained by the two methods, respectively.
The results only show the effects of the {\it van Leer} and {\it van Albada} limiters; other limiter functions lead to similar conclusion.
\begin{table}\centering
\caption{$L_1$ errors and convergence rates by capacity-form differencing to solve (\ref{eq:sec_case_1deuler}) and (\ref{eq:sec_case_ic}) using $\phi^{\it \textrm{van Leer}}$ or $\phi^{\it \textrm{van Albada}}$}
\label{tb:sec_alt_cap}
{\small
\setlength{\tabcolsep}{2pt}
\begin{tabular}{|c|c||cc|cc|cc|cc|cc|cc|}
\hline
 &
 & \multicolumn{6}{c|}{$r=0.2$} 
 & \multicolumn{6}{c|}{$r=0.3$} \\ \cline{3-14}
 &
 & \multicolumn{2}{c|}{$\rho$} & \multicolumn{2}{c|}{$u$} & \multicolumn{2}{c|}{$p$}
 & \multicolumn{2}{c|}{$\rho$} & \multicolumn{2}{c|}{$u$} & \multicolumn{2}{c|}{$p$} \\ \cline{2-14}
 & Mesh 
 & error & rate 
 & error & rate 
 & error & rate 
 & error & rate 
 & error & rate 
 & error & rate  \\ \hline
 \parbox[c]{4mm}{\multirow{5}{*}{\rotatebox[origin=c]{90}{$\phi^{\it \textrm{van Leer}}$}}} 
 & $100$  & $1.51\text{\sc{e}-}3$ &        & $1.64\text{\sc{e}-}3$ &        & $1.90\text{\sc{e}-}3$ &        
          & $1.99\text{\sc{e}-}3$ &        & $2.16\text{\sc{e}-}3$ &        & $2.52\text{\sc{e}-}3$ &        \\ 
 & $200$  & $4.93\text{\sc{e}-}4$ & $1.62$ & $5.32\text{\sc{e}-}4$ & $1.63$ & $5.87\text{\sc{e}-}4$ & $1.70$ 
          & $7.20\text{\sc{e}-}4$ & $1.47$ & $7.73\text{\sc{e}-}4$ & $1.48$ & $8.51\text{\sc{e}-}4$ & $1.57$ \\ 
 & $400$  & $1.83\text{\sc{e}-}4$ & $1.43$ & $1.93\text{\sc{e}-}4$ & $1.46$ & $2.18\text{\sc{e}-}4$ & $1.43$ 
          & $2.92\text{\sc{e}-}4$ & $1.30$ & $3.01\text{\sc{e}-}4$ & $1.36$ & $3.46\text{\sc{e}-}4$ & $1.30$ \\ 
 & $800$  & $8.45\text{\sc{e}-}5$ & $1.12$ & $8.86\text{\sc{e}-}5$ & $1.12$ & $9.78\text{\sc{e}-}5$ & $1.16$ 
          & $1.39\text{\sc{e}-}4$ & $1.07$ & $1.44\text{\sc{e}-}4$ & $1.06$ & $1.61\text{\sc{e}-}4$ & $1.10$ \\ 
 & $1600$ & $4.16\text{\sc{e}-}5$ & $1.02$ & $4.21\text{\sc{e}-}5$ & $1.07$ & $4.64\text{\sc{e}-}5$ & $1.08$ 
          & $7.01\text{\sc{e}-}5$ & $0.99$ & $7.12\text{\sc{e}-}5$ & $1.02$ & $7.89\text{\sc{e}-}5$ & $1.03$ \\ \hline
 \parbox[c]{4mm}{\multirow{5}{*}{\rotatebox[origin=c]{90}{$\phi^{\it \textrm{van Albada}}$}}} 
 & $100$  & $2.18\text{\sc{e}-}3$ &        & $2.29\text{\sc{e}-}3$ &        & $2.75\text{\sc{e}-}3$ &        
          & $2.65\text{\sc{e}-}3$ &        & $2.83\text{\sc{e}-}3$ &        & $3.42\text{\sc{e}-}3$ &        \\ 
 & $200$  & $6.44\text{\sc{e}-}4$ & $1.76$ & $7.00\text{\sc{e}-}4$ & $1.71$ & $7.86\text{\sc{e}-}4$ & $1.81$ 
          & $8.68\text{\sc{e}-}4$ & $1.61$ & $9.52\text{\sc{e}-}4$ & $1.57$ & $1.06\text{\sc{e}-}3$ & $1.69$ \\ 
 & $400$  & $2.17\text{\sc{e}-}4$ & $1.57$ & $2.32\text{\sc{e}-}4$ & $1.59$ & $2.66\text{\sc{e}-}4$ & $1.57$ 
          & $3.29\text{\sc{e}-}4$ & $1.40$ & $3.46\text{\sc{e}-}4$ & $1.46$ & $3.97\text{\sc{e}-}4$ & $1.42$ \\ 
 & $800$  & $9.19\text{\sc{e}-}5$ & $1.24$ & $9.85\text{\sc{e}-}5$ & $1.24$ & $1.10\text{\sc{e}-}4$ & $1.27$ 
          & $1.48\text{\sc{e}-}4$ & $1.15$ & $1.57\text{\sc{e}-}4$ & $1.14$ & $1.77\text{\sc{e}-}4$ & $1.17$ \\ 
 & $1600$ & $4.39\text{\sc{e}-}5$ & $1.06$ & $4.47\text{\sc{e}-}5$ & $1.14$ & $4.96\text{\sc{e}-}5$ & $1.15$ 
          & $7.38\text{\sc{e}-}5$ & $1.01$ & $7.54\text{\sc{e}-}5$ & $1.06$ & $8.41\text{\sc{e}-}5$ & $1.07$ \\ \hline
\end{tabular}
\setlength{\tabcolsep}{1pt}
}
\end{table}
\begin{table}\centering
\caption{$L_1$ errors and convergence rates by MUSCL-MOL to solve (\ref{eq:sec_case_1deuler}) and (\ref{eq:sec_case_ic}) using $\phi^{\it \textrm{van Leer}}_{A,B}$ or $\phi^{\it \textrm{van Albada}}_{A,B}$}
\label{tb:sec_alt_mod}
{\small
\setlength{\tabcolsep}{2pt}
\begin{tabular}{|c|c||cc|cc|cc|cc|cc|cc|}
\hline
 &
 & \multicolumn{6}{c|}{$r=0.2$} 
 & \multicolumn{6}{c|}{$r=0.3$} \\ \cline{3-14}
 &
 & \multicolumn{2}{c|}{$\rho$} & \multicolumn{2}{c|}{$u$} & \multicolumn{2}{c|}{$p$}
 & \multicolumn{2}{c|}{$\rho$} & \multicolumn{2}{c|}{$u$} & \multicolumn{2}{c|}{$p$} \\ \cline{2-14}
 & Mesh 
 & error & rate 
 & error & rate 
 & error & rate 
 & error & rate 
 & error & rate 
 & error & rate  \\ \hline
 \parbox[c]{4mm}{\multirow{5}{*}{\rotatebox[origin=c]{90}{$\phi^{\it \textrm{van Leer}}$}}} 
 & $100$  & $2.52\text{\sc{e}-}3$ &        & $2.80\text{\sc{e}-}3$ &        & $3.30\text{\sc{e}-}3$ &        
          & $2.50\text{\sc{e}-}3$ &        & $2.81\text{\sc{e}-}3$ &        & $3.29\text{\sc{e}-}3$ &        \\ 
 & $200$  & $7.35\text{\sc{e}-}4$ & $1.78$ & $7.03\text{\sc{e}-}4$ & $1.99$ & $8.25\text{\sc{e}-}4$ & $2.00$ 
          & $7.46\text{\sc{e}-}4$ & $1.74$ & $7.15\text{\sc{e}-}4$ & $1.97$ & $8.31\text{\sc{e}-}4$ & $1.98$ \\ 
 & $400$  & $1.41\text{\sc{e}-}4$ & $2.39$ & $1.59\text{\sc{e}-}4$ & $2.14$ & $1.74\text{\sc{e}-}4$ & $2.25$ 
          & $1.40\text{\sc{e}-}4$ & $2.41$ & $1.59\text{\sc{e}-}4$ & $2.17$ & $1.74\text{\sc{e}-}4$ & $2.26$ \\ 
 & $800$  & $3.15\text{\sc{e}-}5$ & $2.16$ & $3.24\text{\sc{e}-}5$ & $2.29$ & $3.65\text{\sc{e}-}5$ & $2.25$ 
          & $3.17\text{\sc{e}-}5$ & $2.15$ & $3.26\text{\sc{e}-}5$ & $2.25$ & $3.74\text{\sc{e}-}5$ & $2.21$ \\ 
 & $1600$ & $6.96\text{\sc{e}-}6$ & $2.18$ & $7.34\text{\sc{e}-}6$ & $2.14$ & $8.58\text{\sc{e}-}6$ & $2.09$ 
          & $7.06\text{\sc{e}-}6$ & $2.17$ & $7.39\text{\sc{e}-}6$ & $2.19$ & $8.61\text{\sc{e}-}6$ & $2.12$ \\ \hline
 \parbox[c]{4mm}{\multirow{5}{*}{\rotatebox[origin=c]{90}{$\phi^{\it \textrm{van Albada}}$}}} 
 & $100$  & $3.07\text{\sc{e}-}3$ &        & $3.54\text{\sc{e}-}3$ &        & $4.15\text{\sc{e}-}3$ &         
          & $3.04\text{\sc{e}-}3$ &        & $3.54\text{\sc{e}-}3$ &        & $4.11\text{\sc{e}-}3$ &        \\ 
 & $200$  & $9.00\text{\sc{e}-}4$ & $1.77$ & $8.82\text{\sc{e}-}4$ & $2.01$ & $1.04\text{\sc{e}-}3$ & $2.00$ 
          & $9.08\text{\sc{e}-}4$ & $1.74$ & $8.92\text{\sc{e}-}4$ & $1.99$ & $1.04\text{\sc{e}-}3$ & $1.98$ \\ 
 & $400$  & $2.09\text{\sc{e}-}4$ & $2.11$ & $2.09\text{\sc{e}-}4$ & $2.07$ & $2.58\text{\sc{e}-}4$ & $2.01$ 
          & $2.29\text{\sc{e}-}4$ & $1.99$ & $2.17\text{\sc{e}-}4$ & $2.04$ & $2.65\text{\sc{e}-}4$ & $1.98$ \\ 
 & $800$  & $4.96\text{\sc{e}-}5$ & $2.08$ & $4.73\text{\sc{e}-}5$ & $2.14$ & $6.35\text{\sc{e}-}5$ & $2.02$ 
          & $4.55\text{\sc{e}-}5$ & $2.33$ & $4.70\text{\sc{e}-}5$ & $2.21$ & $5.82\text{\sc{e}-}5$ & $2.19$ \\ 
 & $1600$ & $1.30\text{\sc{e}-}5$ & $1.94$ & $1.06\text{\sc{e}-}5$ & $2.17$ & $1.44\text{\sc{e}-}5$ & $2.14$ 
          & $1.50\text{\sc{e}-}5$ & $1.60$ & $1.18\text{\sc{e}-}5$ & $1.99$ & $1.73\text{\sc{e}-}5$ & $1.75$ \\ \hline
\end{tabular}
\setlength{\tabcolsep}{1pt}
}
\end{table}

Regarding the capacity-form differencing (Table~\ref{tb:sec_alt_cap}), as the mapping $x(\xi)$ becomes less smooth by either increasing $r$ or refining the mesh, the convergence rates decrease from near second-order to only first-order, as mentioned before.

For the MUSCL-MOL presented in this paper, results in Table~\ref{tb:sec_alt_mod} confirm that the enhanced limiters in Section~\ref{sec:limiter} recover second-order accuracy w.r.t. reference mesh sizes even on highly non-uniform grids.
Furthermore, by looking at the absolute values of the errors, they appear to be nearly independent of the perturbation level $r$.

\subsection{Blast-wave problem}
At last, the aforementioned methods are tested to solve the one-dimensional Woodward-Colella blast-wave problem~\cite{PWoodward:1984a}.
It is an Euler problem on the domain $x\in[0,1]$ with initial condition 
\begin{align}
\notag
&\left.p(x,0)\right|_{0<x<0.1} = 1000.0,\quad
 \left.p(x,0)\right|_{0.1<x<0.9} = 0.01,\quad
 \left.p(x,0)\right|_{0.9<x<1} = 100.0 \\
\label{eq:sec_alt_blst}
&\left.\rho(x,0)\right|_{0<x<1} = 1.0,\quad
 \left.u(x,0)\right|_{0<x<1} = 0.0
\end{align}
The two boundaries are fixed walls.
The solutions at $T=0.038$ exhibit strong shocks with two adjacent large density jumps.
The densities computed using different methods are plotted in Figures~\ref{fg:sec_alt_blast_den_va}.
All the solutions are computed using either the conventional or enhanced {\it van Albada} limiter, as indicated in the figure.
The same mesh is used for all the tests; it is composed of $400$ cells that are computed using $r=0.3$.
Figure~\ref{fg:sec_alt_blast_den_va} also includes a reference solution that is computed on a much finer uniform grid.
\begin{figure}\centering
\begin{subfigure}[b]{.48\textwidth}\centering
    \includegraphics[trim=0.5in 0.1in 0.5in 0.5in, clip, width=\textwidth]{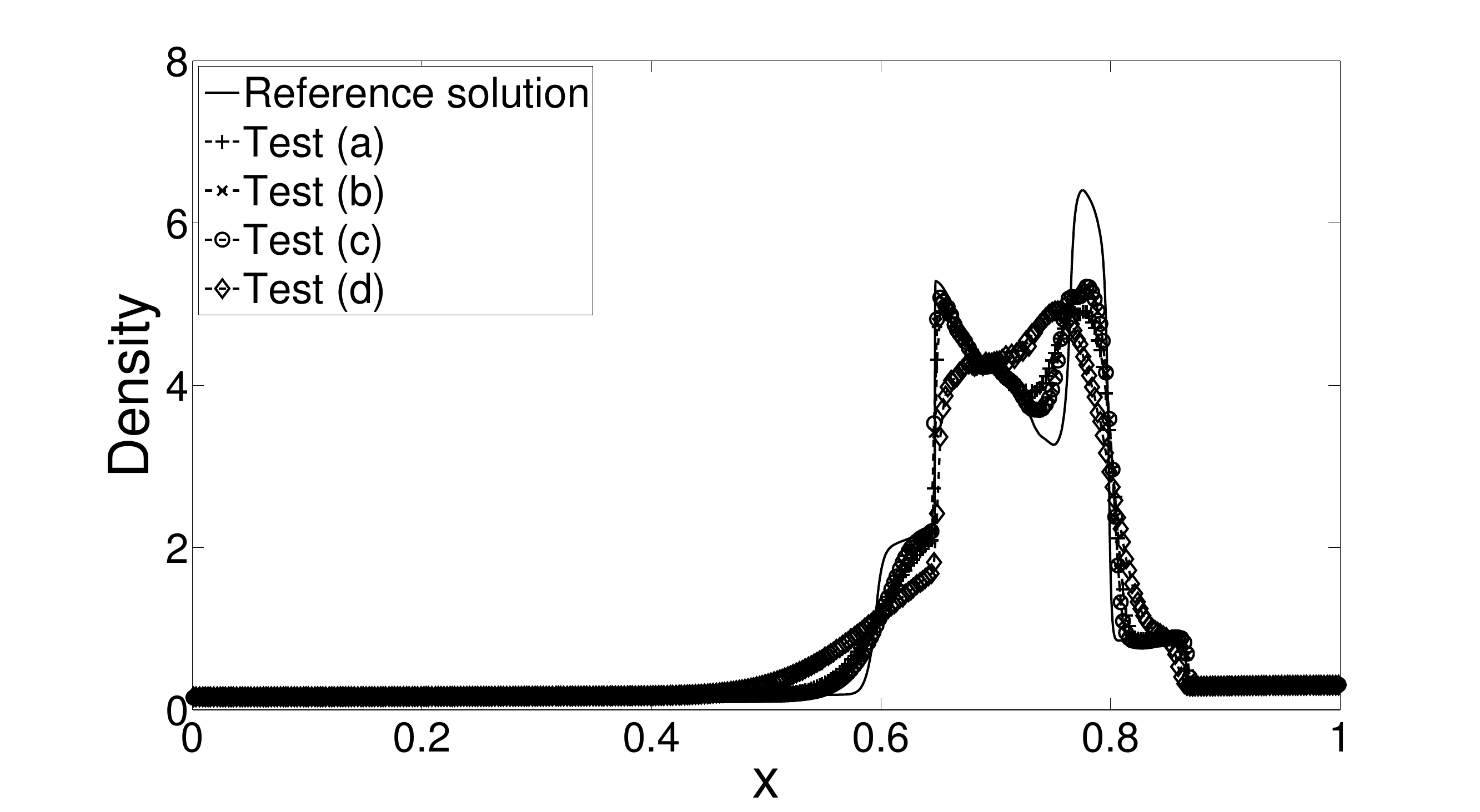}
    \vglue -0.05 truein
    \caption{Global view}
    \label{fg:sec_alt_blast_den_va_glob}
\end{subfigure}
\begin{subfigure}[b]{.48\textwidth}\centering
    \includegraphics[trim=0.5in 0.1in 0.5in 0.5in, clip, width=\textwidth]{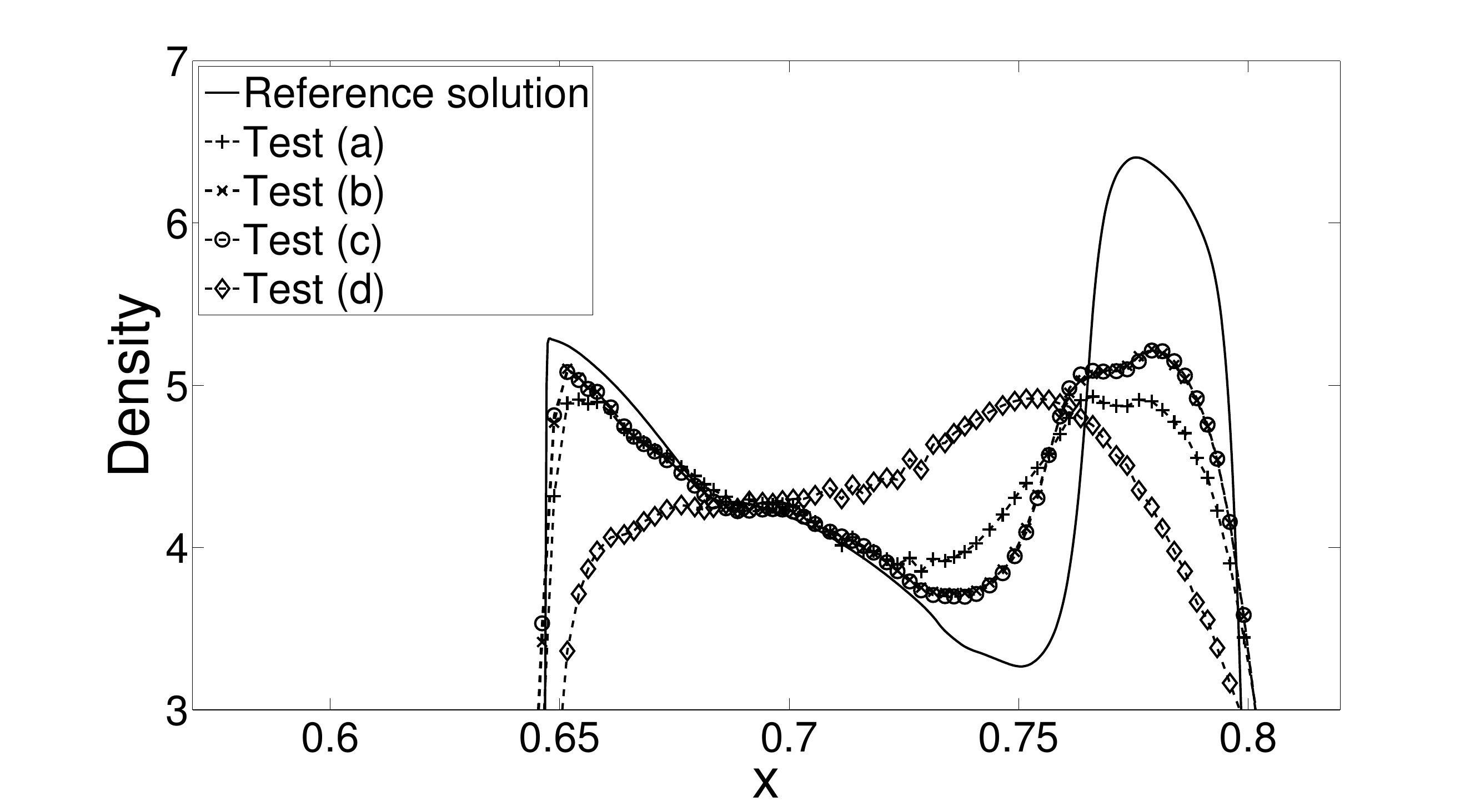}
    \vglue -0.05 truein
    \caption{Local view}
    \label{fg:sec_alt_blast_den_va_loc}
\end{subfigure}
\vglue -0.2 truein
\caption{Densities at $T=0.038$ of blast-wave problem in (\ref{fg:sec_alt_blast_den_va_glob}) global view and (\ref{fg:sec_alt_blast_den_va_loc}) local view, by: 
    test (a) -- slope~(\ref{eq:sec_case_slpe}) and $\phi^{\it \textrm{van Albada}}$, 
    test (b) -- slope~(\ref{eq:alt_consist_slpe}) and $\phi^{\it \textrm{van Albada}}$, 
    test (c) -- slope~(\ref{eq:sec_case_slpe}) and $\phi^{\it \textrm{van Albada}}_{A,B}$, 
    and test (d) -- capacity-form differencing and $\phi^{\it \textrm{van Alabda}}$}
\label{fg:sec_alt_blast_den_va}
\end{figure}
From the figures, there are the following observations 
{\small
\begin{enumerate}
\item The capacity-form differencing is the least accurate among the four tests, possibly due to the reason in Section~\ref{sec:alt_cap}.
\item Among the three MUSCL-MOL tests, the (\ref{eq:sec_case_slpe}) with $\phi^{\it \textrm{van Albada}}$ is less accurate than the other two, which are formally second-order accurate on irregular meshes.
\end{enumerate}
}

The conclusion from the one-dimensional example is that: within the MUSCL-MOL framework, the effects of applying conventional limiters (especially the nonlinear ones) on non-uniform meshes depend on the particular implementation, but it more or less reduces the formal accuracy or stability of the scheme.
The method present in this paper focuses on the MUSCL implemented by using (\ref{eq:sec_case_slpe}) and (\ref{eq:sec_case_monitor}), and proposes enhancements to conventional limiters so that the accuracy and stability are almost independent of the irregularity of the grids.

\section{Two-dimensional examples}
\label{sec:example}
The one-dimensional MUSCL-MOL method extends to 2D rectangular grids naturally by applying the 1D spatial operators to each direction.
The numerical performance of combining (\ref{eq:sec_case_slpe}) and (\ref{eq:sec_case_monitor}) with both conventional and enhanced limiter functions are assessed by solving several benchmark two-dimensional flow problems.
Note that the three enhanced piecewise linear limiters are the same as the improved ones by Berger, hence they are not tested here.

As mentioned at the end of Section~\ref{sec:alt_cap}, the 2D capacity-form differencing method for quadrilateral grids (details are found in numerous references, such as~\cite{RJLeVeque:2002a}) remains formally second-order accurate even on highly non-uniform rectilinear grids, possibly due to the explicit inclusion of transverse fluxes.
This is confirmed in Section~\ref{sec:example_vortex}.
However, similar as the MUSCL-MOL with consistent slopes in previous section, applying the capacity-form differencing with conventional limiter may reduce the stability of the scheme, in both cases of the {\it van Leer} limiter and the {\it van Albada} limiter, as demonstrated in Section~\ref{sec:example_bubble}.

To generate the non-uniform Cartesian meshes, the algorithm in Section~\ref{sec:case} is applied to each direction independently.
The resulting mesh is a rectilinear one with non-uniformity level of the spacing in both $x-$ and $y-$ directions controlled by the same parameter $r\in[0,0.5)$.

\subsection{Isentropic vortex advection}
\label{sec:example_vortex}
The isentropic vortex advection problem \cite{CWShu:2003a} is to solve the 2D Euler equations
\begin{equation}\label{eq:sec_num_2deuler}
\bs{w}_t + \nabla_{\bs{x}}\cdot\bs{F}(\bs{w}) = 0,\quad
\bs{w} = 
\left[\begin{array}{c}
\rho \\ \rho u \\ \rho v \\ E
\end{array}\right],\quad
\bs{F} = 
\left[\begin{array}{cc}
\rho u & \rho v \\
\rho u^2+p & \rho uv \\
\rho uv & \rho v^2+p \\
(E+p)u & (E+p)v 
\end{array}\right]
\end{equation}
where $u$ and $v$ are velocity components, and the total energy is $E = p/(\gamma-1)  + \rho(u^2+v^2)/2,\ \gamma=1.4$.
The computational domain is $(x,y)\in[-5,5]^2$, with periodic boundary condition imposed on all four edges.
Given the initial condition
\begin{align*}
u(x,y,0) &= 1-\frac{\epsilon y}{2\pi}\exp\left(\frac{1}{2}(1-r^2)\right),\quad v(x,y,0) = 1+\frac{\epsilon x}{2\pi}\exp\left(\frac{1}{2}(1-r^2)\right) \\
\rho(x,y,0) &= \left(1-\frac{(\gamma-1)\epsilon^2}{8\gamma \pi^2}\exp(1-r^2)\right)^{\frac{1}{\gamma-1}},\quad
p(x,y,0) = \rho(x,y,0)^\gamma
\end{align*}
with $r^2=x^2+y^2$ and $\epsilon=5$, the flow is an isentropic vortex with uniform entropy $p/\rho^\gamma\equiv1$ that moves at constant velocity $(1,1)$.
At $T=10$, the vortex moves to its original position; thus the initial condition serves as the reference solution.

Three tests are used to solve the problem: (a) MUSCL-MOL using conventional limiters, (b) MUSCL-MOL using enhanced limiters given in Section~\ref{sec:limiter}, and (c) capacity-form differencing.
Two sets of meshes generated by $r=0.2$ and $r=0.3$ are used to assess the convergence behaviors.
For each $r$, four meshes with sizes ranging from $20^2$ to $160^2$ are used.
The numerical errors at $T=10$ measured in $L_1$-norm as well as the convergence rates computed using the conventional or enhanced {\it van Leer} limiters are summarized in Table~\ref{tb:sec_example_vortex_vl}.
Similar results by using the {\it van Albada} limiters are presented in Table~\ref{tb:sec_example_vortex_va}.
\begin{table}\centering
\caption{$L_1$ errors and convergence rates: solving the vortex problem using {\it van Leer} limiters and various meshes --  
test (a) MUSCL-MOL and $\phi^{\it \textrm{van Leer}}$, test (b) MUSCL-MOL and $\phi^{\it \textrm{van Leer}}_{A,B}$, and test (c) capacity-form differencing}
\label{tb:sec_example_vortex_vl}
{\small
\setlength{\tabcolsep}{2pt}
\begin{tabular}{|c|c||cc|cc|cc|cc|cc|cc|}
\hline
 &
 & \multicolumn{4}{c|}{Test (a)} 
 & \multicolumn{4}{c|}{Test (b)} 
 & \multicolumn{4}{c|}{Test (c)} \\ \cline{3-14}
 &
 & \multicolumn{2}{c|}{$r=0.2$} 
 & \multicolumn{2}{c|}{$r=0.3$} 
 & \multicolumn{2}{c|}{$r=0.2$}
 & \multicolumn{2}{c|}{$r=0.3$}
 & \multicolumn{2}{c|}{$r=0.2$} 
 & \multicolumn{2}{c|}{$r=0.3$} \\ \cline{2-14}
 & Mesh 
 & error & rate 
 & error & rate 
 & error & rate 
 & error & rate 
 & error & rate 
 & error & rate  \\ \hline
 \parbox[c]{2mm}{\multirow{4}{*}{\rotatebox[origin=c]{90}{$\rho$}}} 
 & $20^2$  & $1.89\text{\sc{e}-}0$ &        & $1.96\text{\sc{e}-}0$ &        & $1.89\text{\sc{e}-}0$ &        
           & $1.89\text{\sc{e}-}0$ &        & $1.72\text{\sc{e}-}0$ &        & $1.73\text{\sc{e}-}0$ &        \\ 
 & $40^2$  & $8.96\text{\sc{e}-}1$ & $1.08$ & $1.11\text{\sc{e}-}0$ & $0.82$ & $7.44\text{\sc{e}-}1$ & $1.34$ 
           & $7.58\text{\sc{e}-}1$ & $1.32$ & $6.30\text{\sc{e}-}1$ & $1.45$ & $6.58\text{\sc{e}-}1$ & $1.39$ \\ 
 & $80^2$  & $2.98\text{\sc{e}-}1$ & $1.59$ & $4.82\text{\sc{e}-}1$ & $1.21$ & $1.50\text{\sc{e}-}1$ & $2.31$ 
           & $1.55\text{\sc{e}-}1$ & $2.29$ & $1.40\text{\sc{e}-}1$ & $2.17$ & $1.57\text{\sc{e}-}1$ & $2.06$ \\ 
 & $160^2$ & $1.15\text{\sc{e}-}1$ & $1.37$ & $2.28\text{\sc{e}-}1$ & $1.08$ & $3.21\text{\sc{e}-}2$ & $2.22$ 
           & $3.44\text{\sc{e}-}2$ & $2.17$ & $3.59\text{\sc{e}-}2$ & $1.98$ & $4.10\text{\sc{e}-}2$ & $1.94$ \\ \hline
 \parbox[c]{2mm}{\multirow{4}{*}{\rotatebox[origin=c]{90}{$u$}}} 
 & $20^2$  & $4.13\text{\sc{e}-}0$ &        & $4.39\text{\sc{e}-}0$ &        & $4.08\text{\sc{e}-}0$ &        
           & $4.15\text{\sc{e}-}0$ &        & $3.80\text{\sc{e}-}0$ &        & $3.88\text{\sc{e}-}0$ &        \\ 
 & $40^2$  & $1.67\text{\sc{e}-}0$ & $1.30$ & $2.17\text{\sc{e}-}0$ & $1.02$ & $1.39\text{\sc{e}-}0$ & $1.55$ 
           & $1.44\text{\sc{e}-}0$ & $1.52$ & $1.26\text{\sc{e}-}0$ & $1.59$ & $1.38\text{\sc{e}-}0$ & $1.49$ \\ 
 & $80^2$  & $5.45\text{\sc{e}-}1$ & $1.60$ & $8.43\text{\sc{e}-}1$ & $1.36$ & $4.01\text{\sc{e}-}1$ & $1.79$ 
           & $4.25\text{\sc{e}-}1$ & $1.76$ & $3.52\text{\sc{e}-}1$ & $1.84$ & $4.03\text{\sc{e}-}1$ & $1.78$ \\ 
 & $160^2$ & $1.95\text{\sc{e}-}1$ & $1.49$ & $3.73\text{\sc{e}-}1$ & $1.18$ & $1.08\text{\sc{e}-}1$ & $1.89$ 
           & $1.16\text{\sc{e}-}1$ & $1.88$ & $9.36\text{\sc{e}-}2$ & $1.91$ & $1.09\text{\sc{e}-}1$ & $1.88$ \\ \hline
 \parbox[c]{2mm}{\multirow{4}{*}{\rotatebox[origin=c]{90}{$v$}}} 
 & $20^2$  & $4.15\text{\sc{e}-}0$ &        & $4.44\text{\sc{e}-}0$ &        & $4.11\text{\sc{e}-}0$ &        
           & $4.15\text{\sc{e}-}0$ &        & $3.76\text{\sc{e}-}0$ &        & $3.83\text{\sc{e}-}0$ &        \\ 
 & $40^2$  & $1.60\text{\sc{e}-}0$ & $1.38$ & $2.05\text{\sc{e}-}0$ & $1.11$ & $1.32\text{\sc{e}-}0$ & $1.63$ 
           & $1.36\text{\sc{e}-}0$ & $1.61$ & $1.19\text{\sc{e}-}0$ & $1.67$ & $1.26\text{\sc{e}-}0$ & $1.60$ \\ 
 & $80^2$  & $5.03\text{\sc{e}-}1$ & $1.67$ & $8.35\text{\sc{e}-}1$ & $1.30$ & $3.45\text{\sc{e}-}1$ & $1.94$ 
           & $3.64\text{\sc{e}-}1$ & $1.90$ & $3.14\text{\sc{e}-}1$ & $1.92$ & $3.53\text{\sc{e}-}1$ & $1.84$ \\ 
 & $160^2$ & $1.95\text{\sc{e}-}1$ & $1.37$ & $3.95\text{\sc{e}-}1$ & $1.08$ & $8.90\text{\sc{e}-}2$ & $1.96$ 
           & $9.47\text{\sc{e}-}2$ & $1.95$ & $7.96\text{\sc{e}-}2$ & $1.98$ & $9.25\text{\sc{e}-}2$ & $1.93$ \\ \hline
 \parbox[c]{2mm}{\multirow{4}{*}{\rotatebox[origin=c]{90}{$p$}}} 
 & $20^2$  & $2.55\text{\sc{e}-}0$ &        & $2.64\text{\sc{e}-}0$ &        & $2.53\text{\sc{e}-}0$ &        
           & $2.55\text{\sc{e}-}0$ &        & $2.38\text{\sc{e}-}0$ &        & $2.40\text{\sc{e}-}0$ &        \\ 
 & $40^2$  & $1.18\text{\sc{e}-}0$ & $1.11$ & $1.49\text{\sc{e}-}0$ & $0.82$ & $9.60\text{\sc{e}-}1$ & $1.40$ 
           & $9.79\text{\sc{e}-}1$ & $1.38$ & $8.48\text{\sc{e}-}1$ & $1.49$ & $8.93\text{\sc{e}-}1$ & $1.43$ \\ 
 & $80^2$  & $3.83\text{\sc{e}-}1$ & $1.62$ & $6.40\text{\sc{e}-}1$ & $1.22$ & $1.91\text{\sc{e}-}1$ & $2.33$ 
           & $1.98\text{\sc{e}-}1$ & $2.31$ & $1.87\text{\sc{e}-}1$ & $2.18$ & $2.13\text{\sc{e}-}1$ & $2.07$ \\ 
 & $160^2$ & $1.55\text{\sc{e}-}1$ & $1.31$ & $3.12\text{\sc{e}-}1$ & $1.03$ & $4.14\text{\sc{e}-}2$ & $2.21$ 
           & $4.45\text{\sc{e}-}2$ & $2.15$ & $4.66\text{\sc{e}-}2$ & $2.01$ & $5.55\text{\sc{e}-}2$ & $1.94$ \\ \hline
\end{tabular}
\setlength{\tabcolsep}{1pt}
}
\end{table}
\begin{table}\centering
\caption{$L_1$ errors and convergence rates: solving vortex problem using {\it van Albada} limiters and various meshes --  
test (a) MUSCL-MOL and $\phi^{\it \textrm{van Albada}}$, test (b) MUSCL-MOL and $\phi^{\it \textrm{van Albada}}_{A,B}$, and test (c) capacity-form differencing}
\label{tb:sec_example_vortex_va}
{\small
\setlength{\tabcolsep}{2pt}
\begin{tabular}{|c|c||cc|cc|cc|cc|cc|cc|}
\hline
 &
 & \multicolumn{4}{c|}{Test (a)} 
 & \multicolumn{4}{c|}{Test (b)} 
 & \multicolumn{4}{c|}{Test (c)} \\ \cline{3-14}
 &
 & \multicolumn{2}{c|}{$r=0.2$} 
 & \multicolumn{2}{c|}{$r=0.3$} 
 & \multicolumn{2}{c|}{$r=0.2$}
 & \multicolumn{2}{c|}{$r=0.3$}
 & \multicolumn{2}{c|}{$r=0.2$} 
 & \multicolumn{2}{c|}{$r=0.3$} \\ \cline{2-14}
 & Mesh 
 & error & rate 
 & error & rate 
 & error & rate 
 & error & rate 
 & error & rate 
 & error & rate  \\ \hline
 \parbox[c]{2mm}{\multirow{4}{*}{\rotatebox[origin=c]{90}{$\rho$}}} 
 & $20^2$  & $2.08\text{\sc{e}-}0$ &        & $2.14\text{\sc{e}-}0$ &        & $2.04\text{\sc{e}-}0$ &        
           & $2.04\text{\sc{e}-}0$ &        & $1.94\text{\sc{e}-}0$ &        & $1.94\text{\sc{e}-}0$ &        \\ 
 & $40^2$  & $1.12\text{\sc{e}-}0$ & $0.89$ & $1.35\text{\sc{e}-}0$ & $0.66$ & $9.14\text{\sc{e}-}1$ & $1.16$ 
           & $9.28\text{\sc{e}-}1$ & $1.14$ & $8.20\text{\sc{e}-}1$ & $1.24$ & $8.51\text{\sc{e}-}1$ & $1.19$ \\ 
 & $80^2$  & $4.16\text{\sc{e}-}1$ & $1.44$ & $6.61\text{\sc{e}-}1$ & $1.03$ & $1.95\text{\sc{e}-}1$ & $2.23$ 
           & $1.96\text{\sc{e}-}1$ & $2.23$ & $1.83\text{\sc{e}-}1$ & $2.16$ & $2.05\text{\sc{e}-}1$ & $2.05$ \\ 
 & $160^2$ & $1.69\text{\sc{e}-}1$ & $1.30$ & $3.29\text{\sc{e}-}1$ & $1.01$ & $3.71\text{\sc{e}-}2$ & $2.39$ 
           & $3.88\text{\sc{e}-}2$ & $2.35$ & $4.28\text{\sc{e}-}2$ & $2.10$ & $5.14\text{\sc{e}-}2$ & $2.00$ \\ \hline
 \parbox[c]{2mm}{\multirow{4}{*}{\rotatebox[origin=c]{90}{$u$}}} 
 & $20^2$  & $4.67\text{\sc{e}-}0$ &        & $4.90\text{\sc{e}-}0$ &        & $4.53\text{\sc{e}-}0$ &        
           & $4.60\text{\sc{e}-}0$ &        & $4.37\text{\sc{e}-}0$ &        & $4.45\text{\sc{e}-}0$ &        \\ 
 & $40^2$  & $2.05\text{\sc{e}-}0$ & $1.19$ & $2.66\text{\sc{e}-}0$ & $0.88$ & $1.59\text{\sc{e}-}0$ & $1.51$ 
           & $1.64\text{\sc{e}-}0$ & $1.49$ & $1.48\text{\sc{e}-}0$ & $1.56$ & $1.58\text{\sc{e}-}0$ & $1.49$ \\ 
 & $80^2$  & $6.66\text{\sc{e}-}1$ & $1.62$ & $1.13\text{\sc{e}-}0$ & $1.24$ & $4.35\text{\sc{e}-}1$ & $1.87$ 
           & $4.56\text{\sc{e}-}1$ & $1.85$ & $4.01\text{\sc{e}-}1$ & $1.88$ & $4.48\text{\sc{e}-}1$ & $1.82$ \\ 
 & $160^2$ & $2.63\text{\sc{e}-}1$ & $1.34$ & $5.36\text{\sc{e}-}1$ & $1.07$ & $1.20\text{\sc{e}-}1$ & $1.86$ 
           & $1.27\text{\sc{e}-}1$ & $1.85$ & $1.09\text{\sc{e}-}1$ & $1.89$ & $1.24\text{\sc{e}-}1$ & $1.85$ \\ \hline
 \parbox[c]{2mm}{\multirow{4}{*}{\rotatebox[origin=c]{90}{$v$}}} 
 & $20^2$  & $4.75\text{\sc{e}-}0$ &        & $4.99\text{\sc{e}-}0$ &        & $4.59\text{\sc{e}-}0$ &        
           & $4.64\text{\sc{e}-}0$ &        & $4.37\text{\sc{e}-}0$ &        & $4.42\text{\sc{e}-}0$ &        \\ 
 & $40^2$  & $1.97\text{\sc{e}-}0$ & $1.27$ & $2.54\text{\sc{e}-}0$ & $0.97$ & $1.56\text{\sc{e}-}0$ & $1.56$ 
           & $1.59\text{\sc{e}-}0$ & $1.55$ & $1.44\text{\sc{e}-}0$ & $1.60$ & $1.51\text{\sc{e}-}0$ & $1.55$ \\ 
 & $80^2$  & $6.71\text{\sc{e}-}1$ & $1.56$ & $1.16\text{\sc{e}-}0$ & $1.13$ & $3.94\text{\sc{e}-}1$ & $1.98$ 
           & $4.11\text{\sc{e}-}1$ & $1.95$ & $3.79\text{\sc{e}-}1$ & $1.93$ & $4.16\text{\sc{e}-}1$ & $1.86$ \\ 
 & $160^2$ & $2.78\text{\sc{e}-}1$ & $1.27$ & $5.72\text{\sc{e}-}1$ & $1.02$ & $1.05\text{\sc{e}-}1$ & $1.91$ 
           & $1.10\text{\sc{e}-}1$ & $1.90$ & $9.83\text{\sc{e}-}2$ & $1.94$ & $1.11\text{\sc{e}-}1$ & $1.90$ \\ \hline
 \parbox[c]{2mm}{\multirow{4}{*}{\rotatebox[origin=c]{90}{$p$}}} 
 & $20^2$  & $2.76\text{\sc{e}-}0$ &        & $2.82\text{\sc{e}-}0$ &        & $2.72\text{\sc{e}-}0$ &        
           & $2.73\text{\sc{e}-}0$ &        & $2.62\text{\sc{e}-}0$ &        & $2.64\text{\sc{e}-}0$ &        \\ 
 & $40^2$  & $1.47\text{\sc{e}-}0$ & $0.90$ & $1.81\text{\sc{e}-}0$ & $0.64$ & $1.17\text{\sc{e}-}0$ & $1.21$ 
           & $1.20\text{\sc{e}-}0$ & $1.19$ & $1.10\text{\sc{e}-}0$ & $1.25$ & $1.14\text{\sc{e}-}0$ & $1.20$ \\ 
 & $80^2$  & $5.35\text{\sc{e}-}1$ & $1.46$ & $8.82\text{\sc{e}-}1$ & $1.04$ & $2.38\text{\sc{e}-}1$ & $2.29$ 
           & $2.46\text{\sc{e}-}1$ & $2.28$ & $2.44\text{\sc{e}-}1$ & $2.17$ & $2.74\text{\sc{e}-}1$ & $2.06$ \\ 
 & $160^2$ & $2.25\text{\sc{e}-}1$ & $1.25$ & $4.51\text{\sc{e}-}1$ & $0.97$ & $4.81\text{\sc{e}-}2$ & $2.32$ 
           & $5.06\text{\sc{e}-}2$ & $2.28$ & $5.64\text{\sc{e}-}2$ & $2.11$ & $6.94\text{\sc{e}-}2$ & $1.98$ \\ \hline
\end{tabular}
\setlength{\tabcolsep}{1pt}
}
\end{table}
In Figure~\ref{fg:sec_example_vortex}, the pressure contours obtained by the three methods using the {\it van Albada} limiters on the finest mesh with $r=0.3$ are plotted.
Clearly, the test (b) and test (c) show similar accuracy property, which are much better than the test (a).
Using {\it van Leer} limiters lead to similar results.
\begin{figure}\centering
\begin{subfigure}[b]{0.32\textwidth}\centering
\includegraphics[trim=0.7in 0.0in 1.3in 0.5in, clip, width=\textwidth]{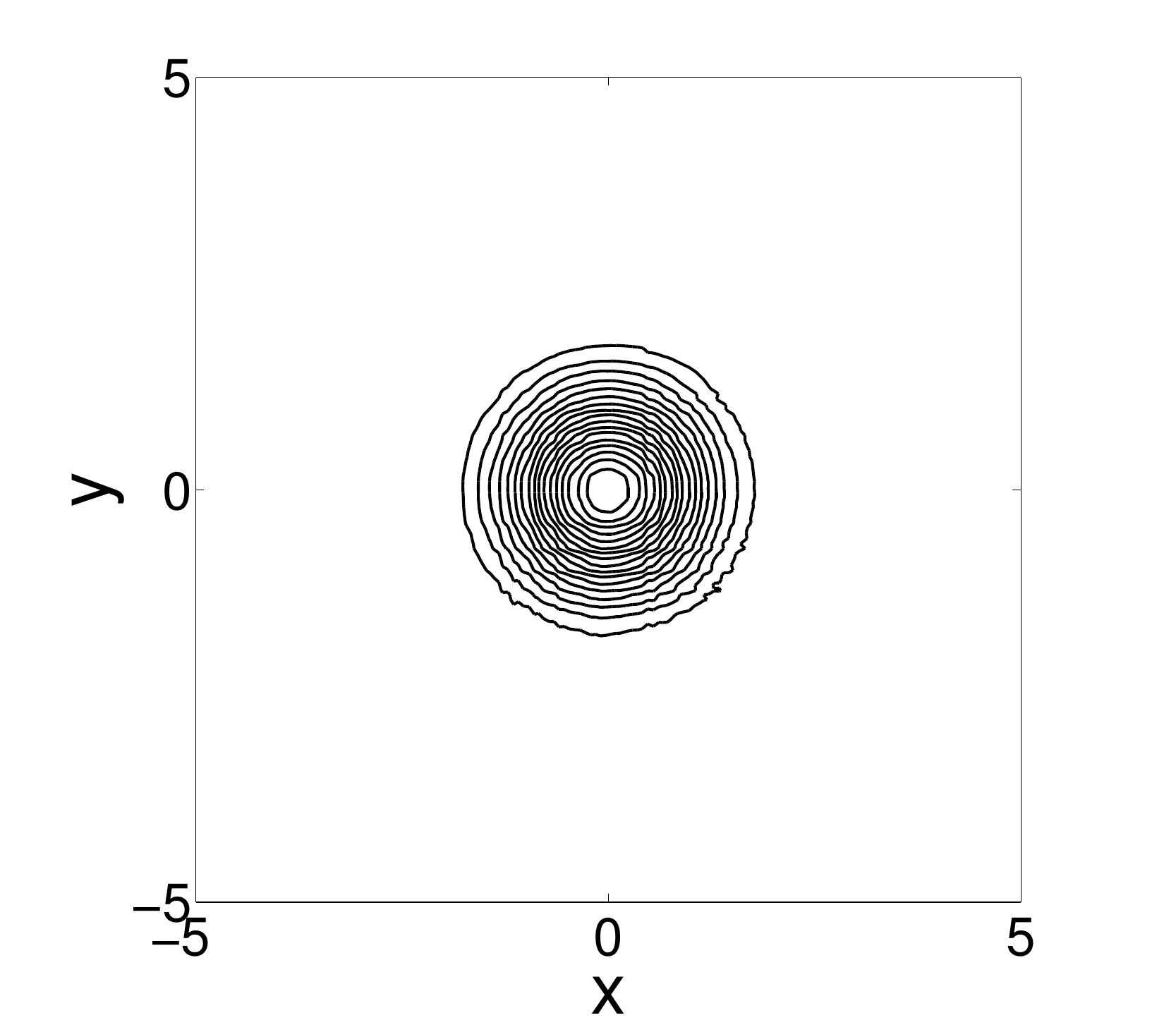}
\vglue -0.05 truein
\caption{Test (a)}
\label{fg:sec_example_vortex_org}
\end{subfigure}
\begin{subfigure}[b]{0.32\textwidth}\centering
\includegraphics[trim=0.7in 0.0in 1.3in 0.5in, clip, width=\textwidth]{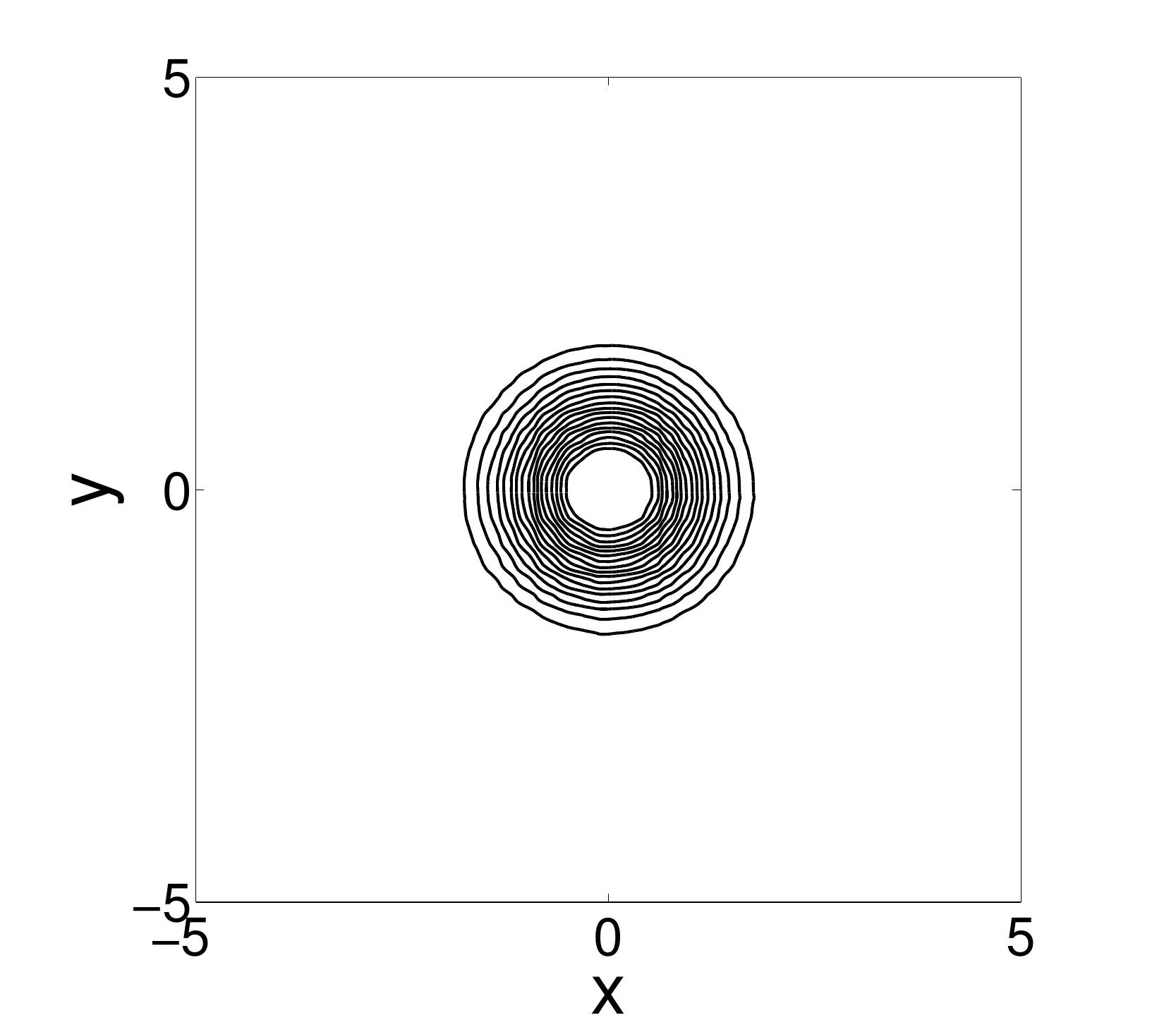}
\vglue -0.05 truein
\caption{Test (b)}
\label{fg:sec_example_vortex_lim}
\end{subfigure}
\begin{subfigure}[b]{0.32\textwidth}\centering
\includegraphics[trim=0.7in 0.0in 1.3in 0.5in, clip, width=\textwidth]{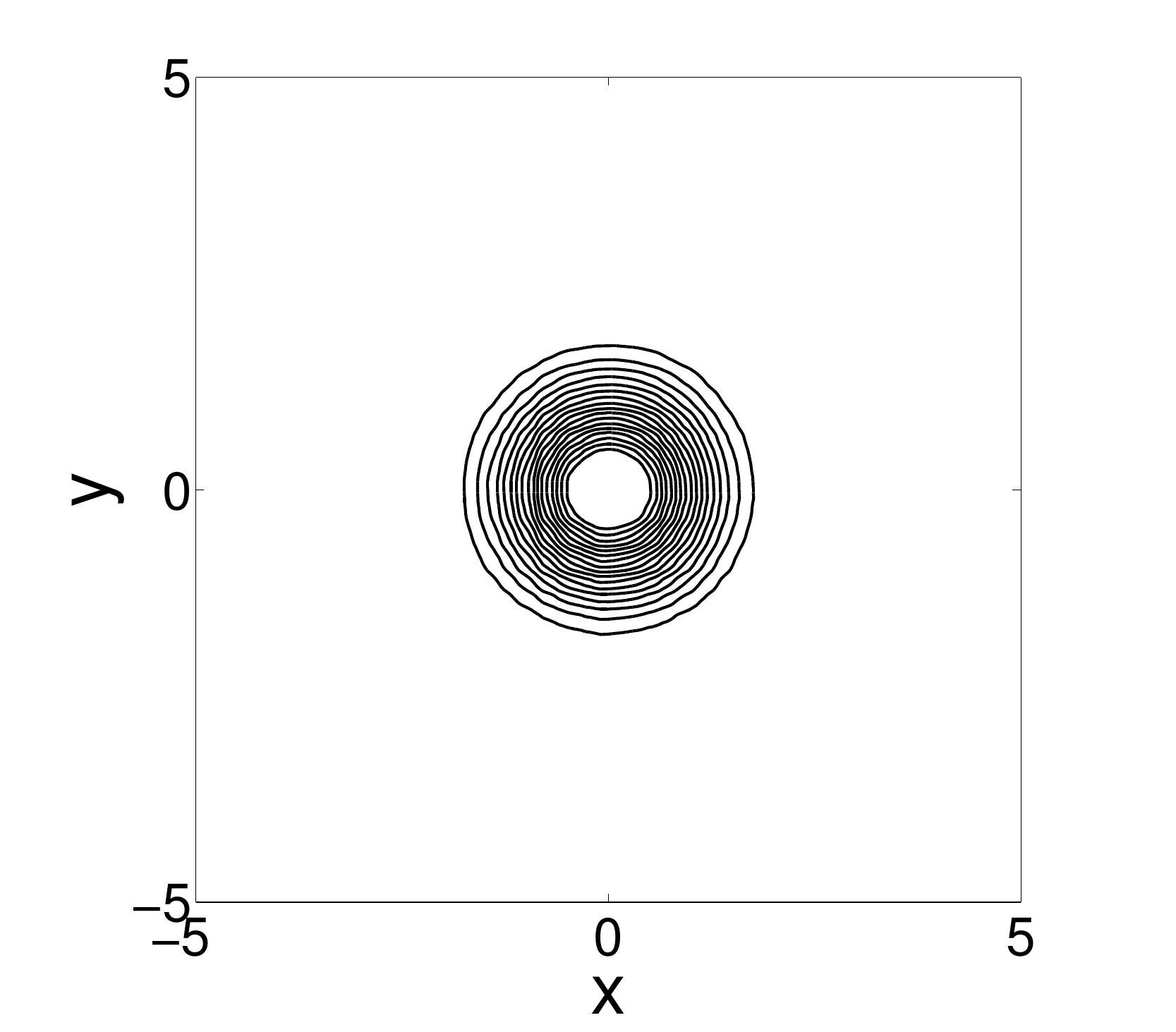}
\vglue -0.05 truein
\caption{Test (c)}
\label{fg:sec_example_vortex_cap}
\end{subfigure}
\vglue -0.2 truein
\caption{Pressure contours on the $160^2$ mesh ($r=0.3$) by: 
    (\ref{fg:sec_example_vortex_org}) MUSCL-MOL with $\phi^{\it \textrm{van Albada}}$, 
    (\ref{fg:sec_example_vortex_lim}) MUSCL-MOL with $\phi^{\it \textrm{van Albada}}_{A,B}$, 
    (\ref{fg:sec_example_vortex_cap}) capacity-form differencing} 
\label{fg:sec_example_vortex}
\end{figure}

These results confirm that the enhanced limiter greatly improves the accuracy of the MUSCL-MOL on highly non-uniform grids;
and the numerical errors are almost $r$-independent.
The results also suggest that the capacity-form differencing method is formally second-order accurate on these highly non-uniform grids.

\subsection{Shock-bubble interaction}
\label{sec:example_bubble}
The shock-bubble interaction problem~\cite{MCada:2009a} solves the Euler equations to simulate the interaction between a moving shock and a low-density bubble, indicated by Figure~\ref{fg:sec_example_bubble_setup}.
Because of the symmetry of the problem, only the upper half of the domain is used in computation.
The typical density at $T=0.4$ is shown in Figure~\ref{fg:sec_example_bubble_den}, which is computed using MUSCL-MOL and a uniform $340\times100$ grid and the {\it van Albada} limiter.
In this figure, the lower half of the data is obtained by mirroring the upper half.
\begin{figure}\centering
\begin{subfigure}[b]{.48\textwidth}\centering
    \begin{tikzpicture}[line width=1.6]
        \draw (-0.3,-1.5) rectangle (4.8,1.5);
        \draw [line width=0.8] (0.0,-1.5) -- (0.0,1.5);
        \draw [line width=0.8] (0.9,0.0) circle (0.6);
        \draw (-0.3,-1.5) node [left]  {\tiny $-0.5$};
        \draw (-0.3,1.5)  node [left]  {\tiny $0.5$};
        \draw (-0.3,-1.5) node [below] {\tiny $-0.1$};
        \draw (0.0,1.5)   node [above] {\tiny $0.0$};
        \draw (4.8,-1.5)  node [below] {\tiny $1.6$};
        \draw (2.4, 0.25) node [right] {\tiny $\rho = 1.0$};
        \draw (2.4, 0.0) node [right] {\tiny $u = 0.0$};
        \draw (2.4,-0.25) node [right] {\tiny $p = 1.0$};
        \draw (2.25,-1.5) node [below] {\tiny wall};
        \draw (2.25, 1.5) node [above] {\tiny wall};
        \node [rotate=90] at (-0.5,0.0) {\tiny in-flow};
        \node [rotate=90] at (5.0,0.0) {\tiny out-flow};
        \draw (0.35,0.25)  node [right] {\tiny $\rho = 0.1$};
        \draw (0.35,0.0)   node [right] {\tiny $u = 0.0$};
        \draw (0.35,-0.25) node [right] {\tiny $p = 1.0$};
        \draw [densely dashed, line width=0.8] (0.9,-0.6) -- (0.9,-1.5);
        \draw [densely dashed, line width=0.8] (1.5, 0.0) -- (1.5,-1.5);
        \draw (0.9,-1.5) node [below] {\tiny $0.3$};
        \draw (1.5,-1.5) node [below] {\tiny $0.5$};
        \node [rotate=90] at (-0.15,0.0) {\tiny $\rho=3.81, u=2.85, p=10$};
        \draw [color=white] (0.9,-2.0) -- (1.5,-2.0);
    \end{tikzpicture}
    \vglue -0.05 truein
    \caption{Initial problem setup: $v\equiv0.0$}
    \label{fg:sec_example_bubble_setup}
\end{subfigure}
\begin{subfigure}[b]{.48\textwidth}\centering
    \includegraphics[trim=1.0in 0.1in 1.0in 0.3in, clip, width=\textwidth]{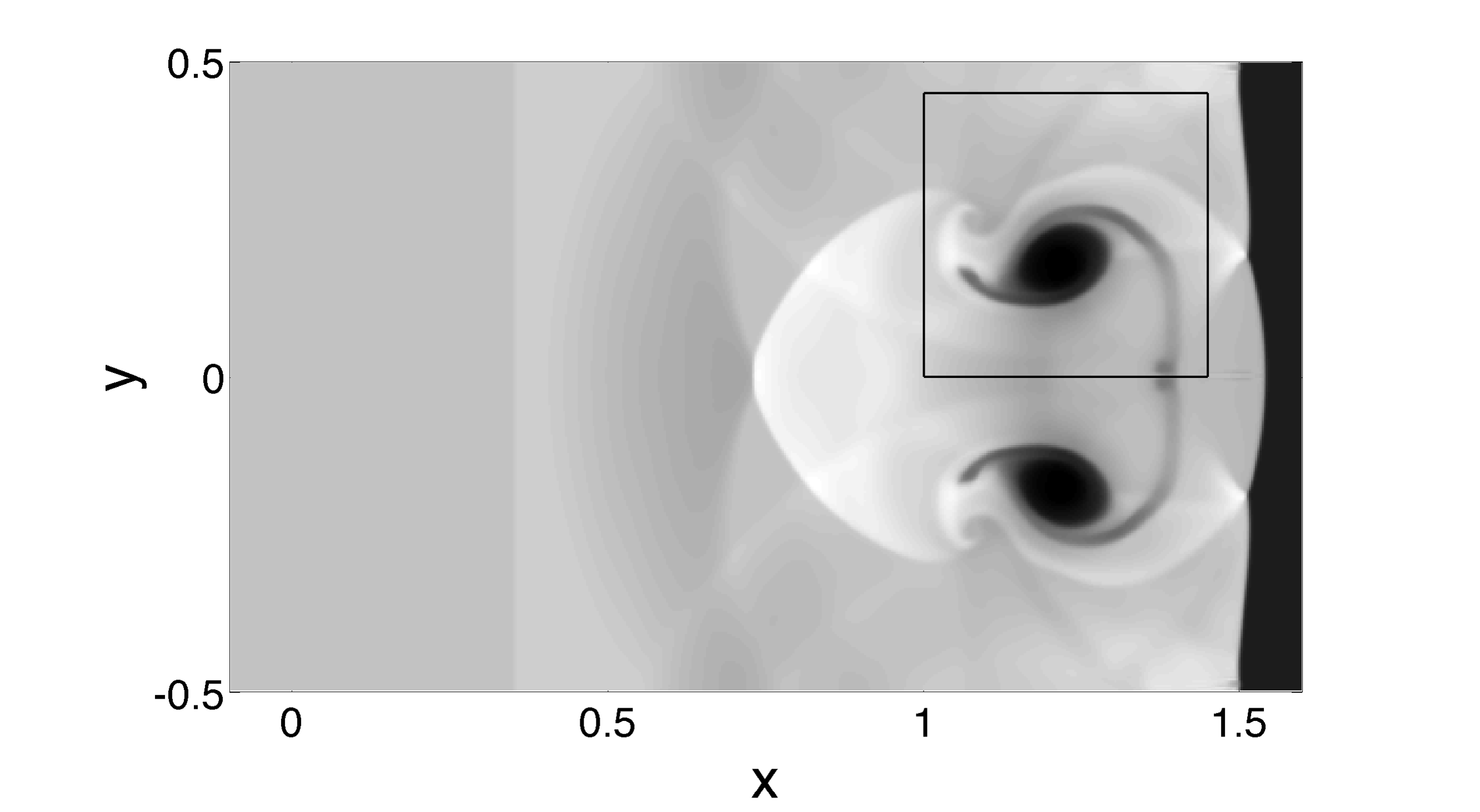}
    \vglue -0.05 truein
    \caption{Density at $T=0.4$}
    \label{fg:sec_example_bubble_den}
\end{subfigure}
\vglue -0.2 truein
\caption{Problem setup and density at $T=0.4$ on a uniform $340\times100$ grid}
\label{fg:sec_example_bubble}
\end{figure}
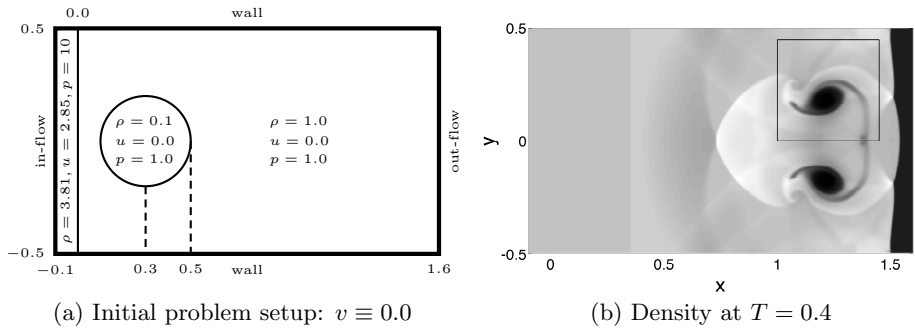

The densities in the whole computational domain and in the region indicated by the black box in Figure~\ref{fg:sec_example_bubble_den} computed by different methods using {\it van Leer} limiters are shown in the left column of Figure~\ref{fg:sec_example_bubble_den_glob} and Figure~\ref{fg:sec_example_bubble_den_loc_vl}, respectively.
Similar results computed using the {\it van Albada} limiters are plotted in the right column of Figure~\ref{fg:sec_example_bubble_den_glob} and Figure~\ref{fg:sec_example_bubble_den_loc_va} for the global views and local views, respectively.
\begin{figure}\centering
\begin{subfigure}[b]{.48\textwidth}\centering
    \includegraphics[trim=0.8in 1.4in 1.1in 2.4in, clip, width=\textwidth]{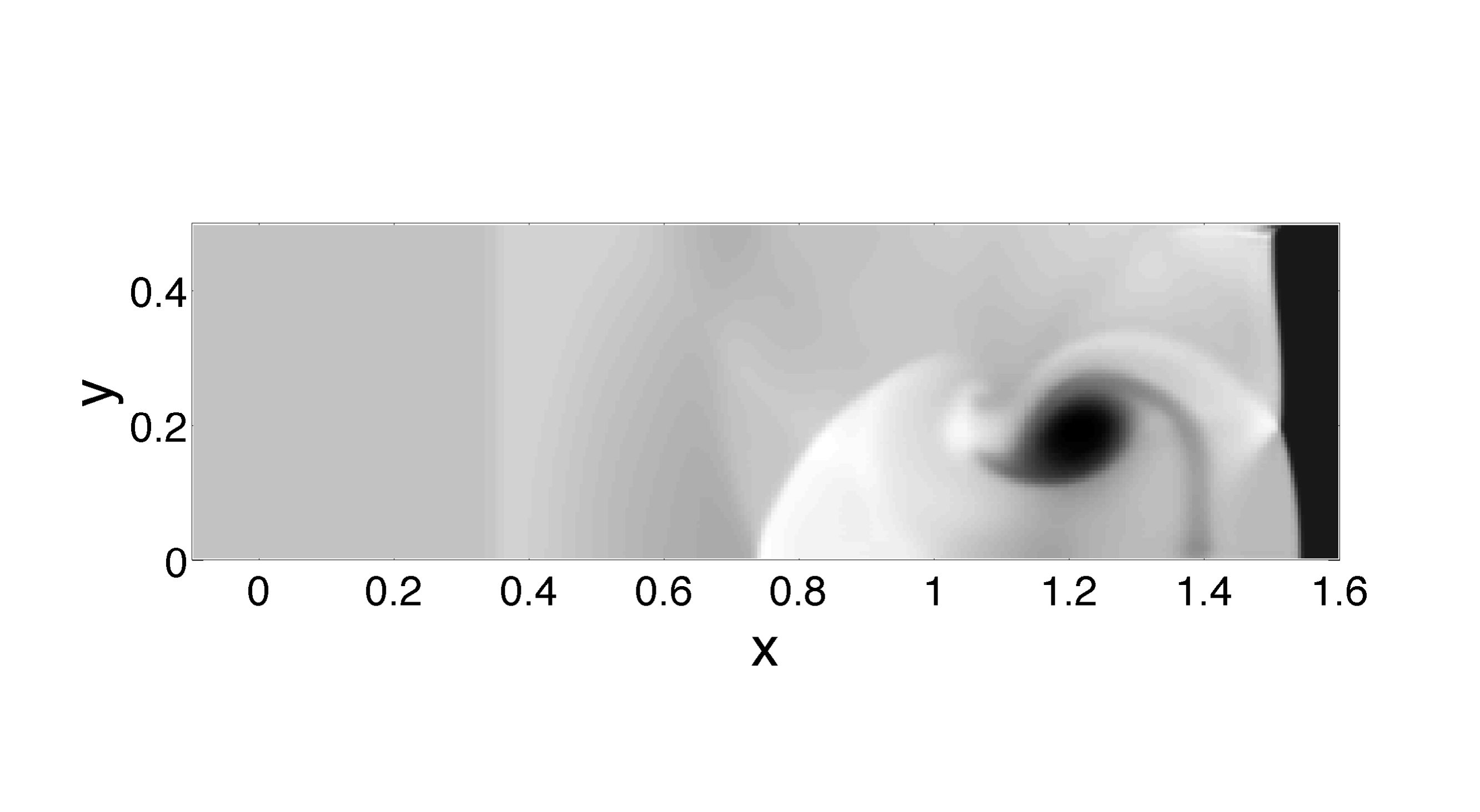}
    \vglue -0.1 truein
    \caption{Test (a): $\phi^{\it \textrm{van Leer}}$}
    \label{fg:sec_example_bubble_den_glob_vl_org}
\end{subfigure}
\begin{subfigure}[b]{.48\textwidth}\centering
    \includegraphics[trim=0.8in 1.4in 1.1in 2.4in, clip, width=\textwidth]{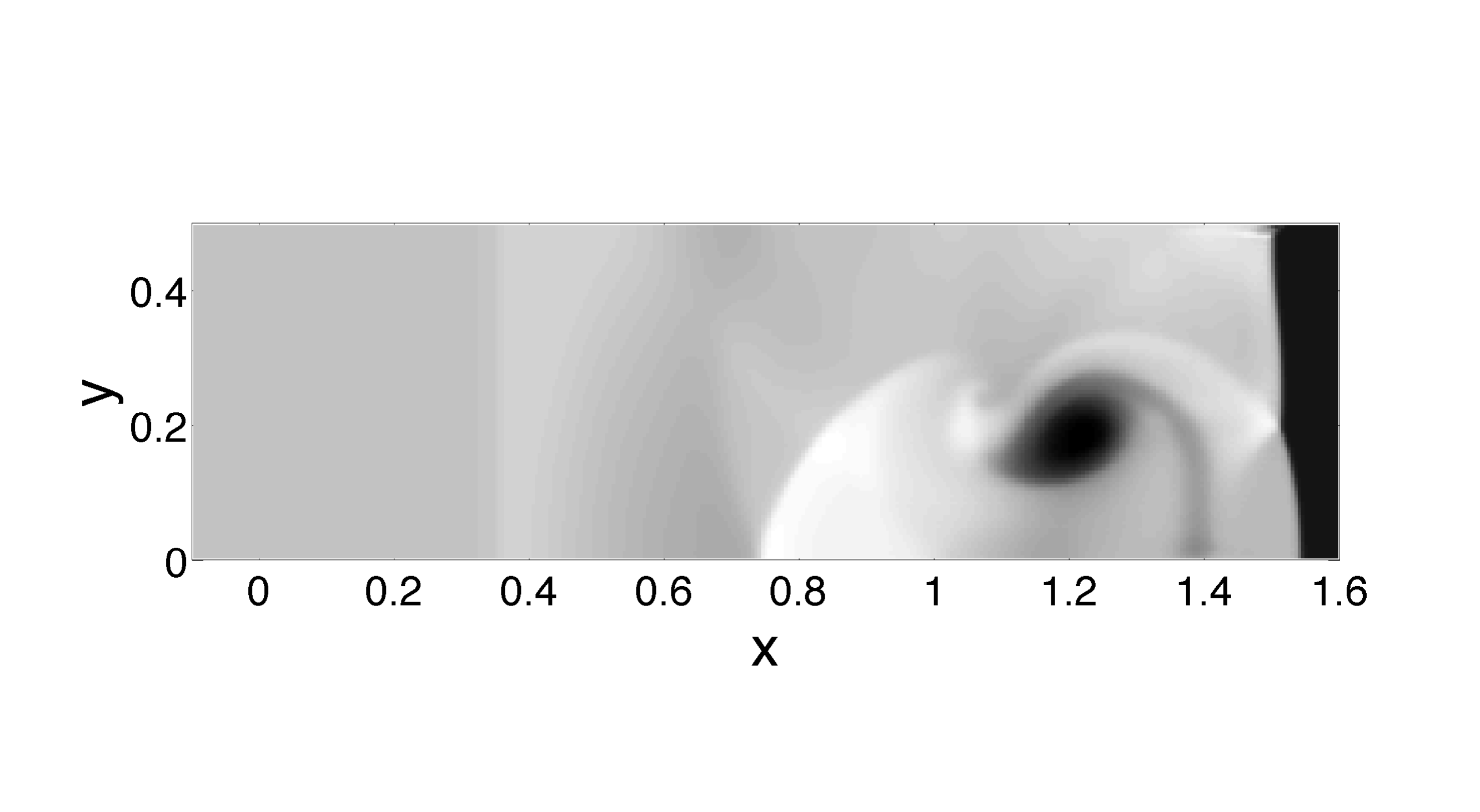}
    \vglue -0.1 truein
    \caption{Test (a): $\phi^{\it \textrm{van Albada}}$}
    \label{fg:sec_example_bubble_den_glob_va_org}
\end{subfigure}
\begin{subfigure}[b]{.48\textwidth}\centering
    \includegraphics[trim=0.8in 1.4in 1.1in 2.4in, clip, width=\textwidth]{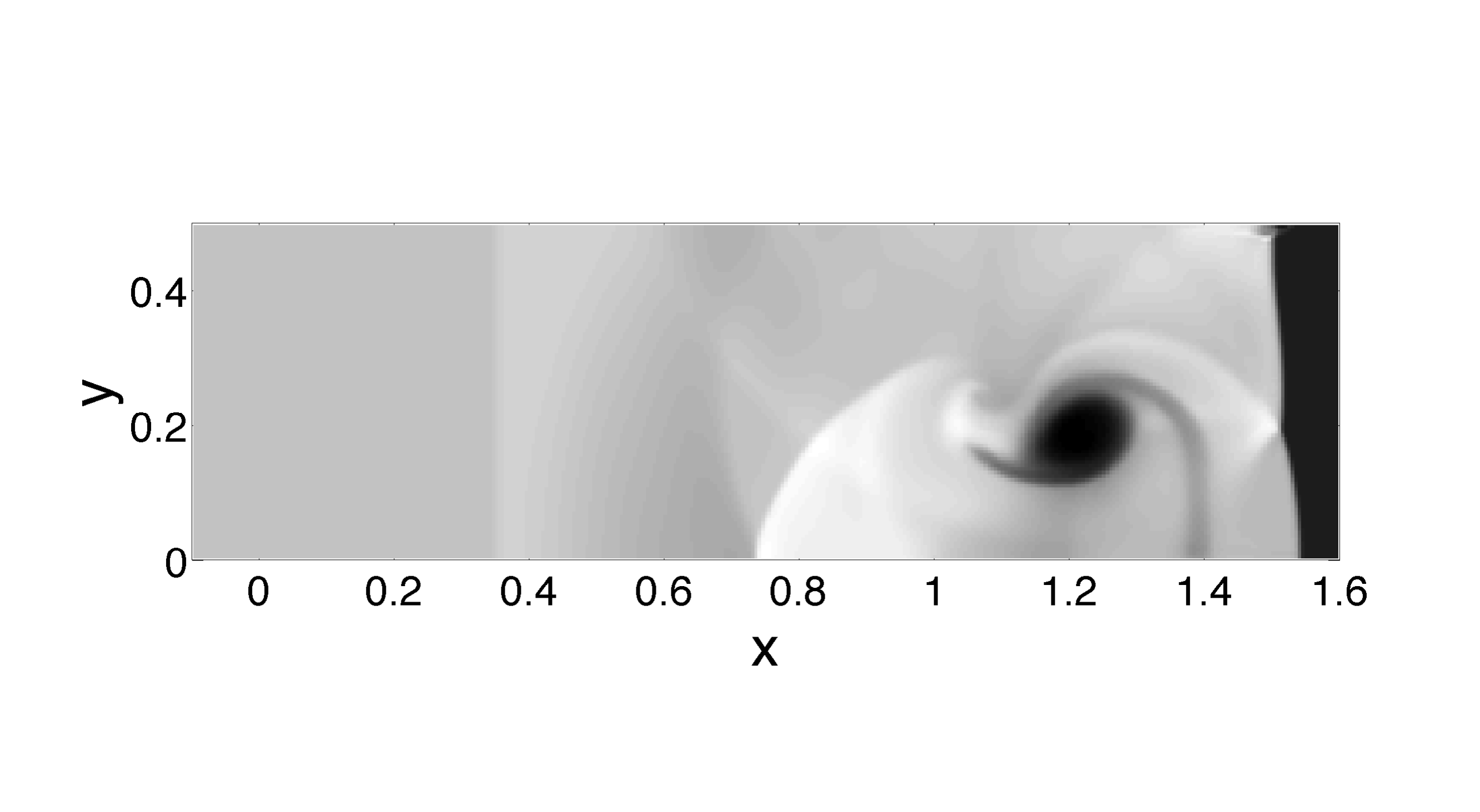}
    \vglue -0.1 truein
    \caption{Test (b): $\phi^{\it \textrm{van Leer}}_{A,B}$}
    \label{fg:sec_example_bubble_den_glob_vl_lim}
\end{subfigure}
\begin{subfigure}[b]{.48\textwidth}\centering
    \includegraphics[trim=0.8in 1.4in 1.1in 2.4in, clip, width=\textwidth]{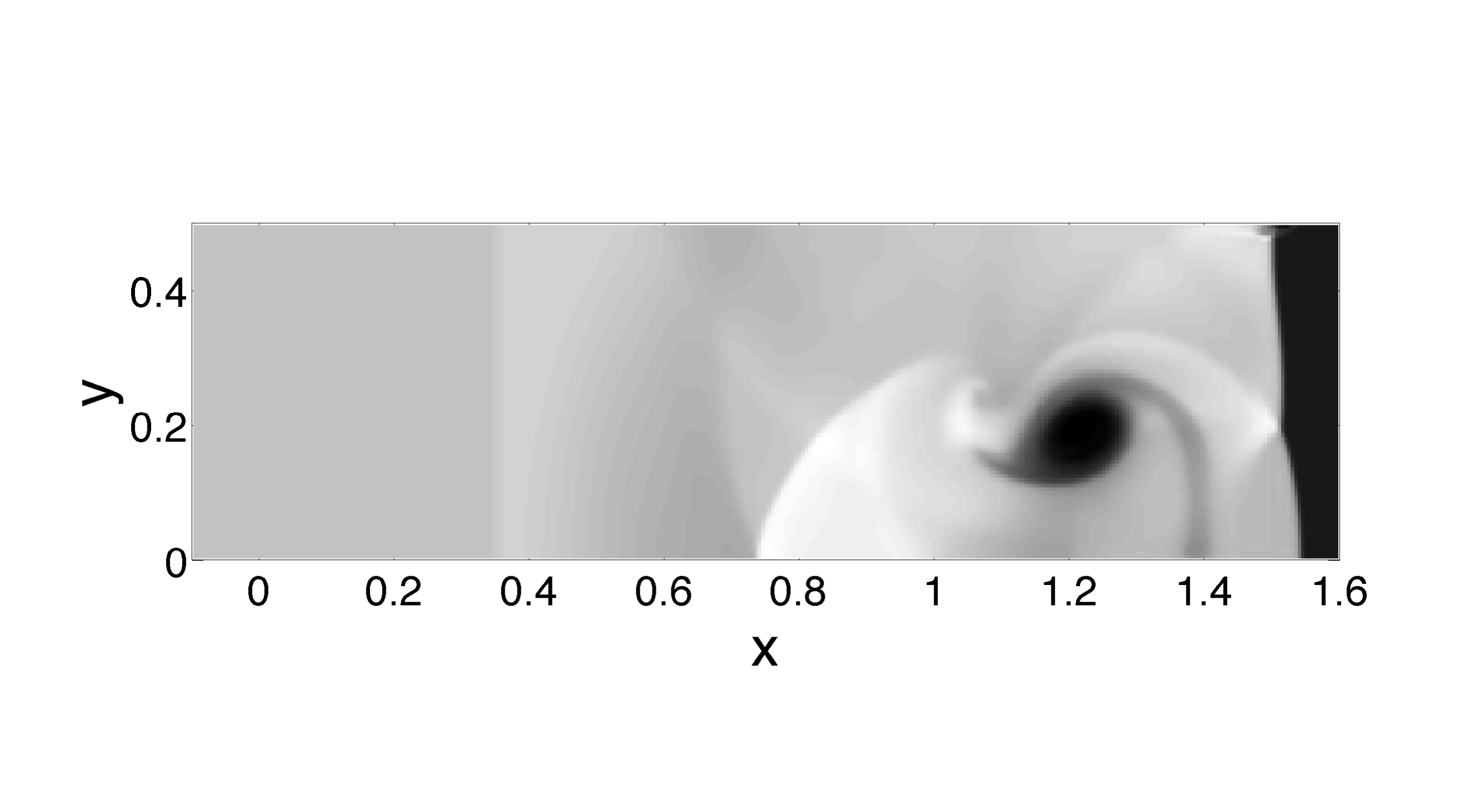}
    \vglue -0.1 truein
    \caption{Test (b): $\phi^{\it \textrm{van Albada}}_{A,B}$}
    \label{fg:sec_example_bubble_den_glob_va_lim}
\end{subfigure}
\begin{subfigure}[b]{.48\textwidth}\centering
    \includegraphics[trim=0.8in 1.4in 1.1in 2.4in, clip, width=\textwidth]{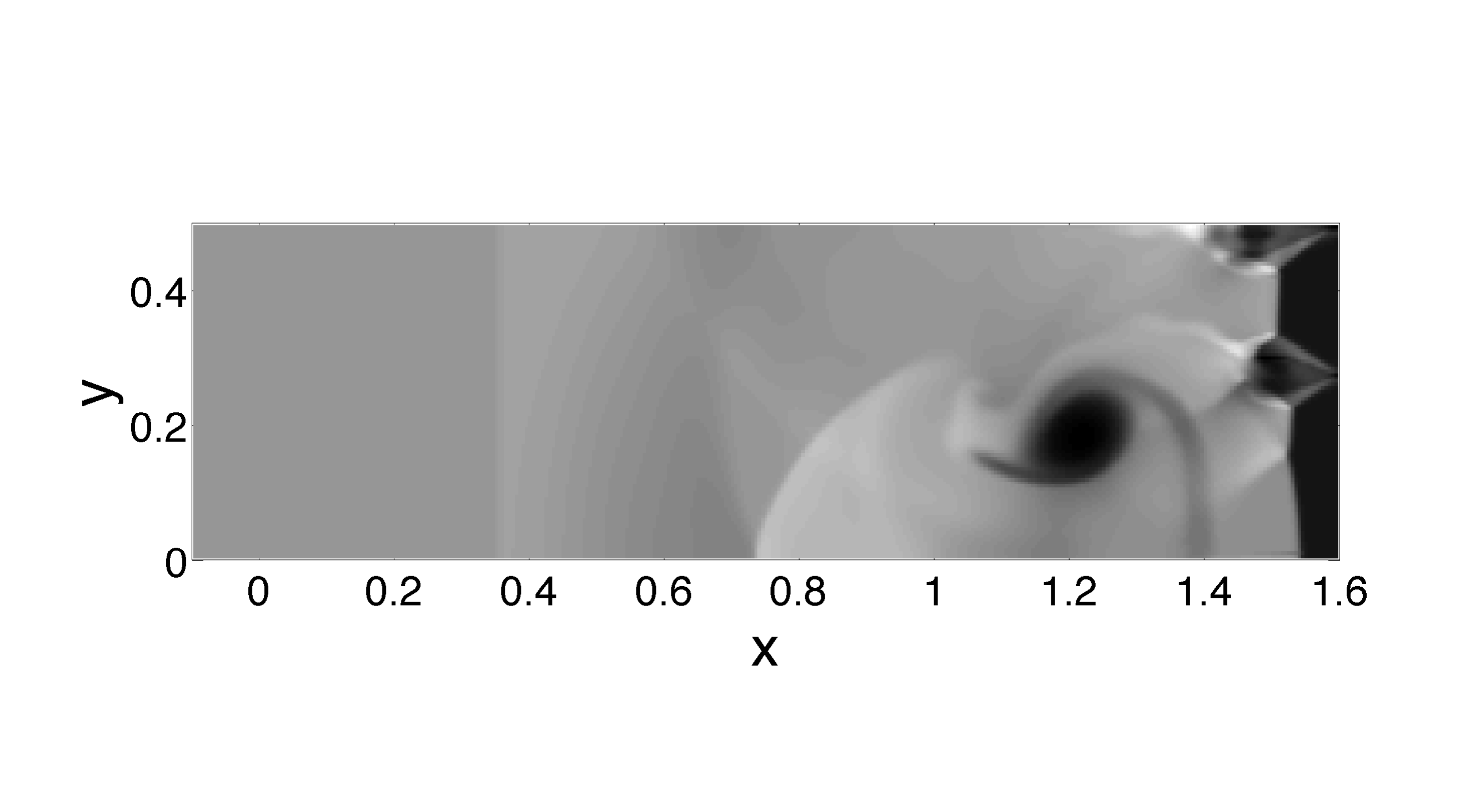}
    \vglue -0.1 truein
    \caption{Test (c): $\phi^{\it \textrm{van Leer}}$}
    \label{fg:sec_example_bubble_den_glob_vl_cap}
\end{subfigure}
\begin{subfigure}[b]{.48\textwidth}\centering
    \includegraphics[trim=0.8in 1.4in 1.1in 2.4in, clip, width=\textwidth]{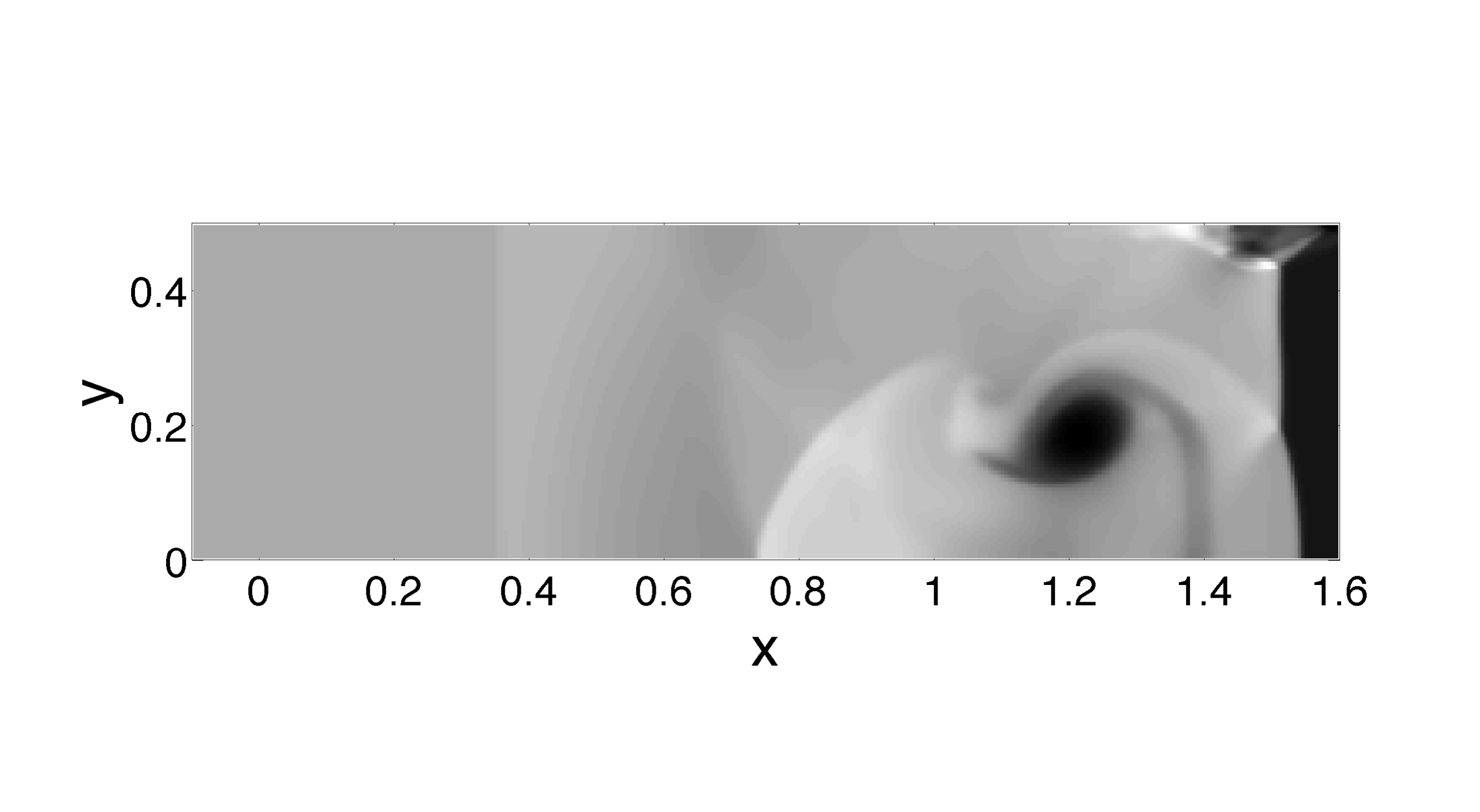}
    \vglue -0.1 truein
    \caption{Test (c): $\phi^{\it \textrm{van Albada}}$}
    \label{fg:sec_example_bubble_den_glob_va_cap}
\end{subfigure}
\vglue -0.2 truein
\caption{Densities at $T=0.4$ on a $340\times100$ mesh ($r=0.3$): 
(left column) {\it van Leer} limiters, (right column) {\it van Albada} limiters; 
(\ref{fg:sec_example_bubble_den_glob_vl_org}--\ref{fg:sec_example_bubble_den_glob_va_org}) MUSCL-MOL and conventional limiters,
(\ref{fg:sec_example_bubble_den_glob_vl_lim}--\ref{fg:sec_example_bubble_den_glob_va_lim}) MUSCL-MOL and enhanced limiters,
(\ref{fg:sec_example_bubble_den_glob_vl_cap}--\ref{fg:sec_example_bubble_den_glob_va_cap}) capacity-form differencing and conventional limiters
}
\label{fg:sec_example_bubble_den_glob}
\end{figure}
\begin{figure}\centering
\begin{subfigure}[b]{0.32\textwidth}\centering
\includegraphics[trim=0.7in 0.0in 1.3in 0.5in, clip, width=\textwidth]{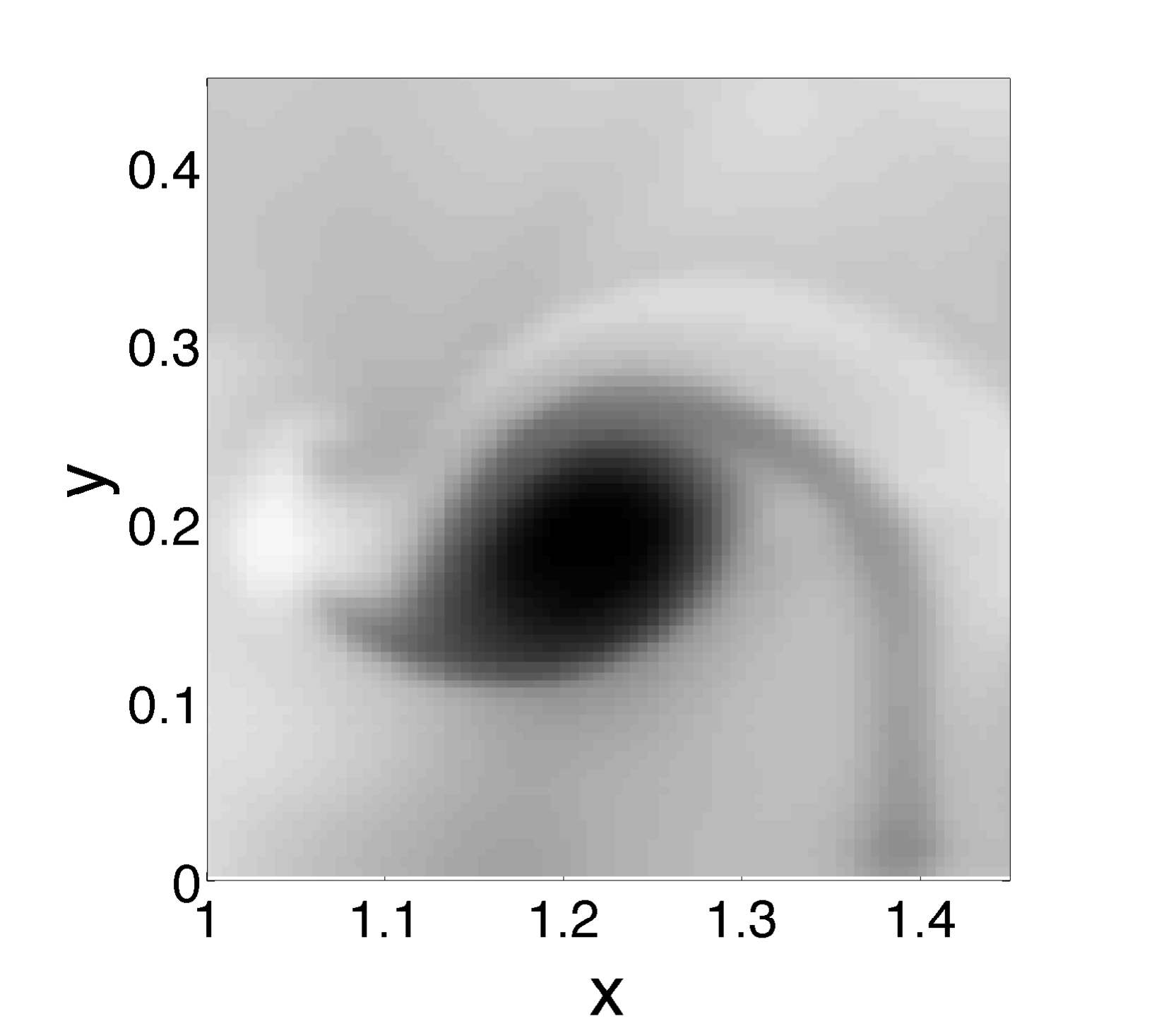}
\vglue -0.05 truein
\caption{Test (a)}
\label{fg:sec_example_bubble_den_loc_vl_org}
\end{subfigure}
\begin{subfigure}[b]{0.32\textwidth}\centering
\includegraphics[trim=0.7in 0.0in 1.3in 0.5in, clip, width=\textwidth]{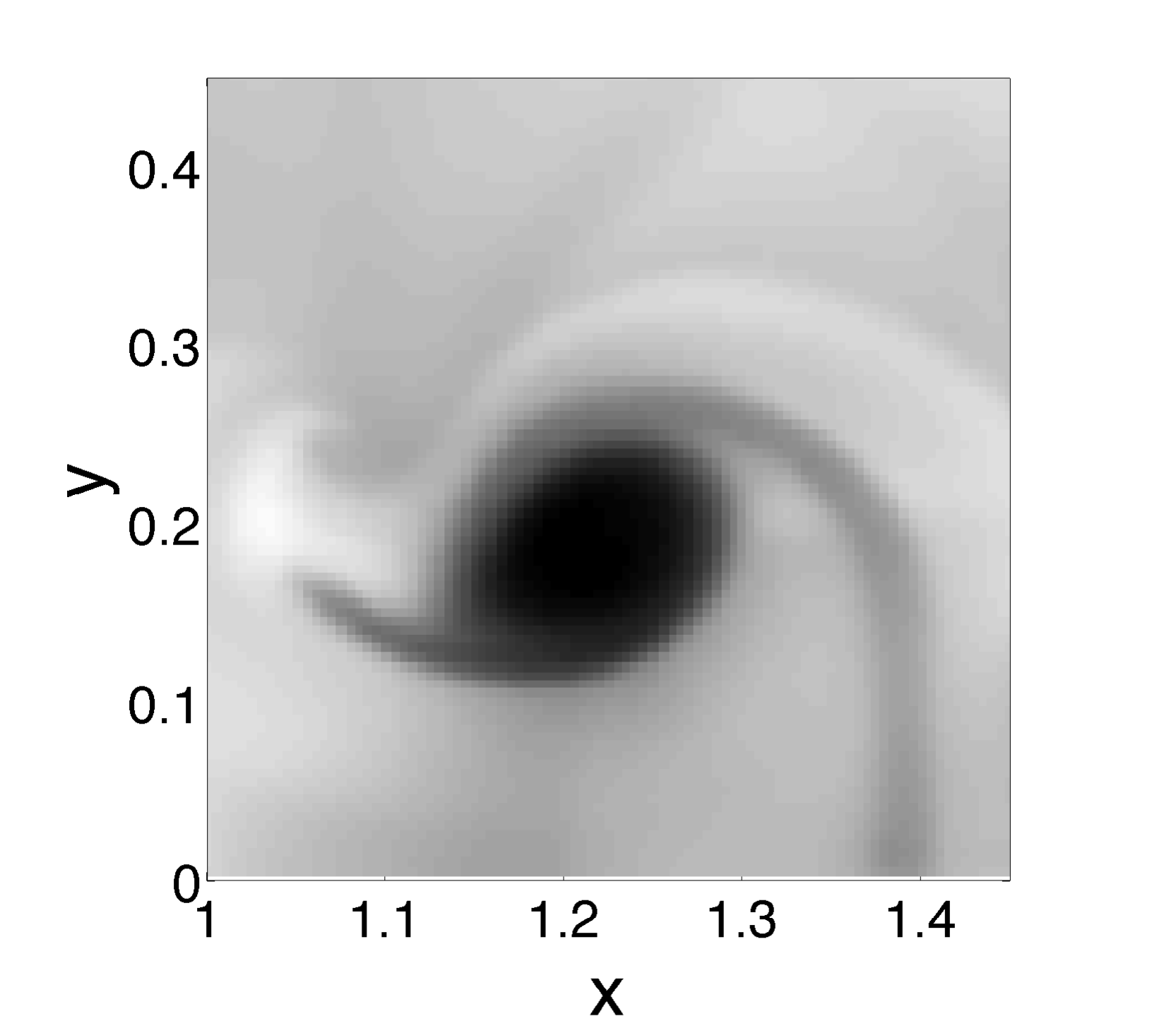}
\vglue -0.05 truein
\caption{Test (b)}
\label{fg:sec_example_bubble_den_loc_vl_lim}
\end{subfigure}
\begin{subfigure}[b]{0.32\textwidth}\centering
\includegraphics[trim=0.7in 0.0in 1.3in 0.5in, clip, width=\textwidth]{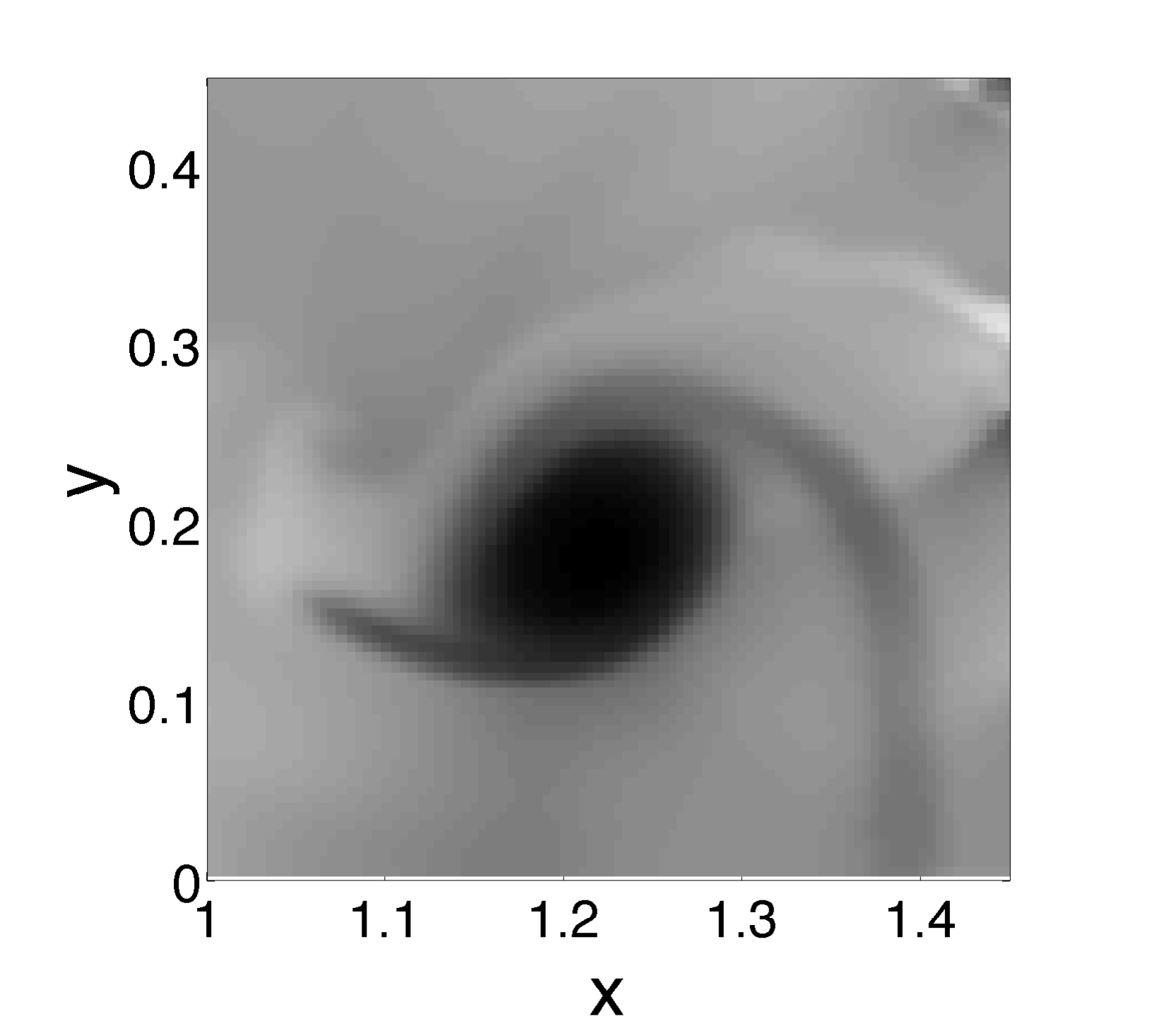}
\vglue -0.05 truein
\caption{Test (c)}
\label{fg:sec_example_bubble_den_loc_vl_cap}
\end{subfigure}
\vglue -0.2 truein
\caption{Local views of density at $T=0.4$ on a $340\times100$ mesh ($r=0.3$): 
    (\ref{fg:sec_example_bubble_den_loc_vl_org}) MUSCL-MOL and $\phi^{\it \textrm{van Leer}}$, 
    (\ref{fg:sec_example_bubble_den_loc_vl_lim}) MUSCL-MOL and $\phi^{\it \textrm{van Leer}}_{A,B}$,
and (\ref{fg:sec_example_bubble_den_loc_vl_cap}) capacity-form differencing and $\phi^{\it \textrm{van Leer}}$}
\label{fg:sec_example_bubble_den_loc_vl}
\end{figure}
\begin{figure}\centering
\begin{subfigure}[b]{0.32\textwidth}\centering
\includegraphics[trim=0.7in 0.0in 1.3in 0.5in, clip, width=\textwidth]{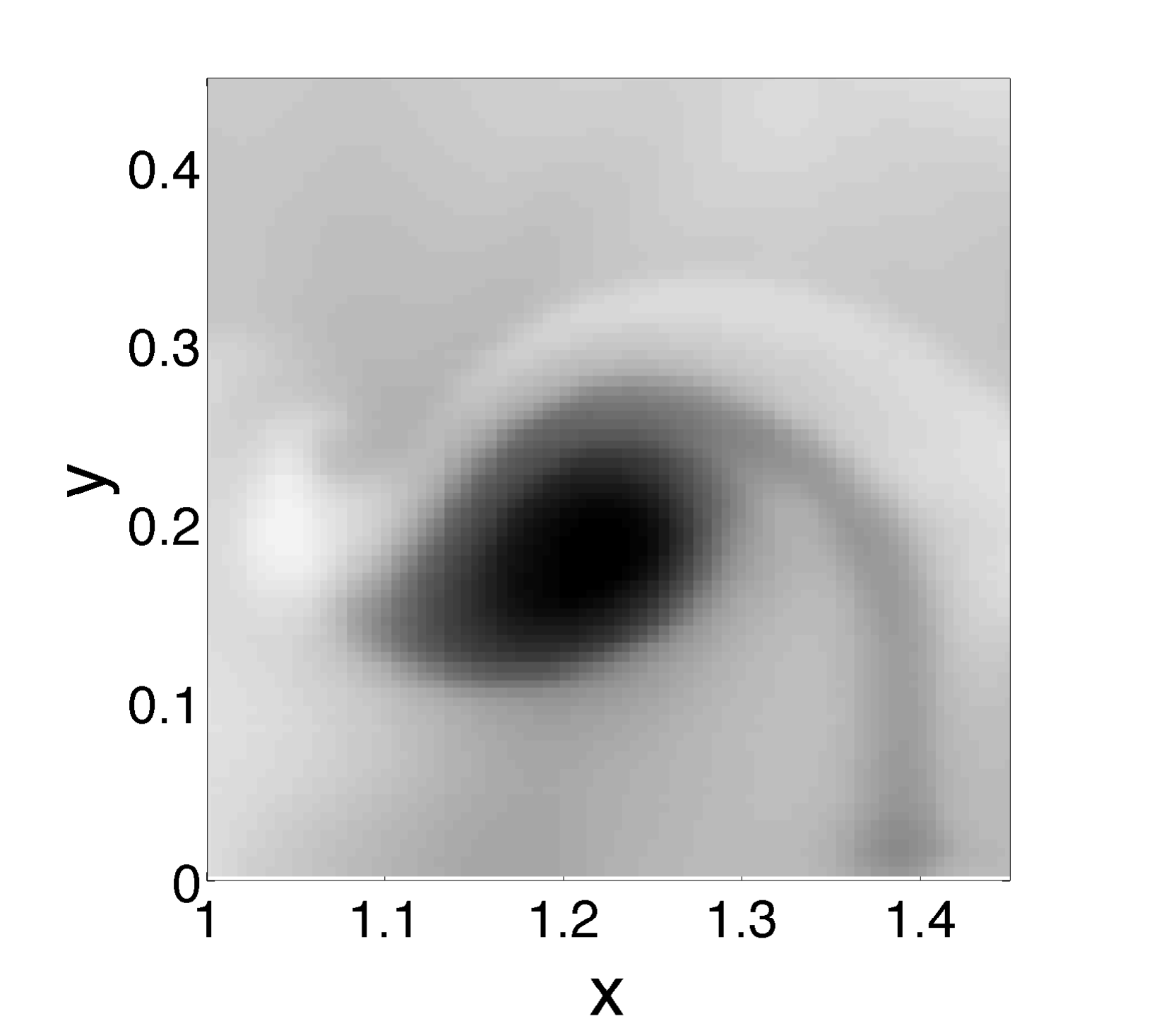}
\vglue -0.05 truein
\caption{Test (a)}
\label{fg:sec_example_bubble_den_loc_va_org}
\end{subfigure}
\begin{subfigure}[b]{0.32\textwidth}\centering
\includegraphics[trim=0.7in 0.0in 1.3in 0.5in, clip, width=\textwidth]{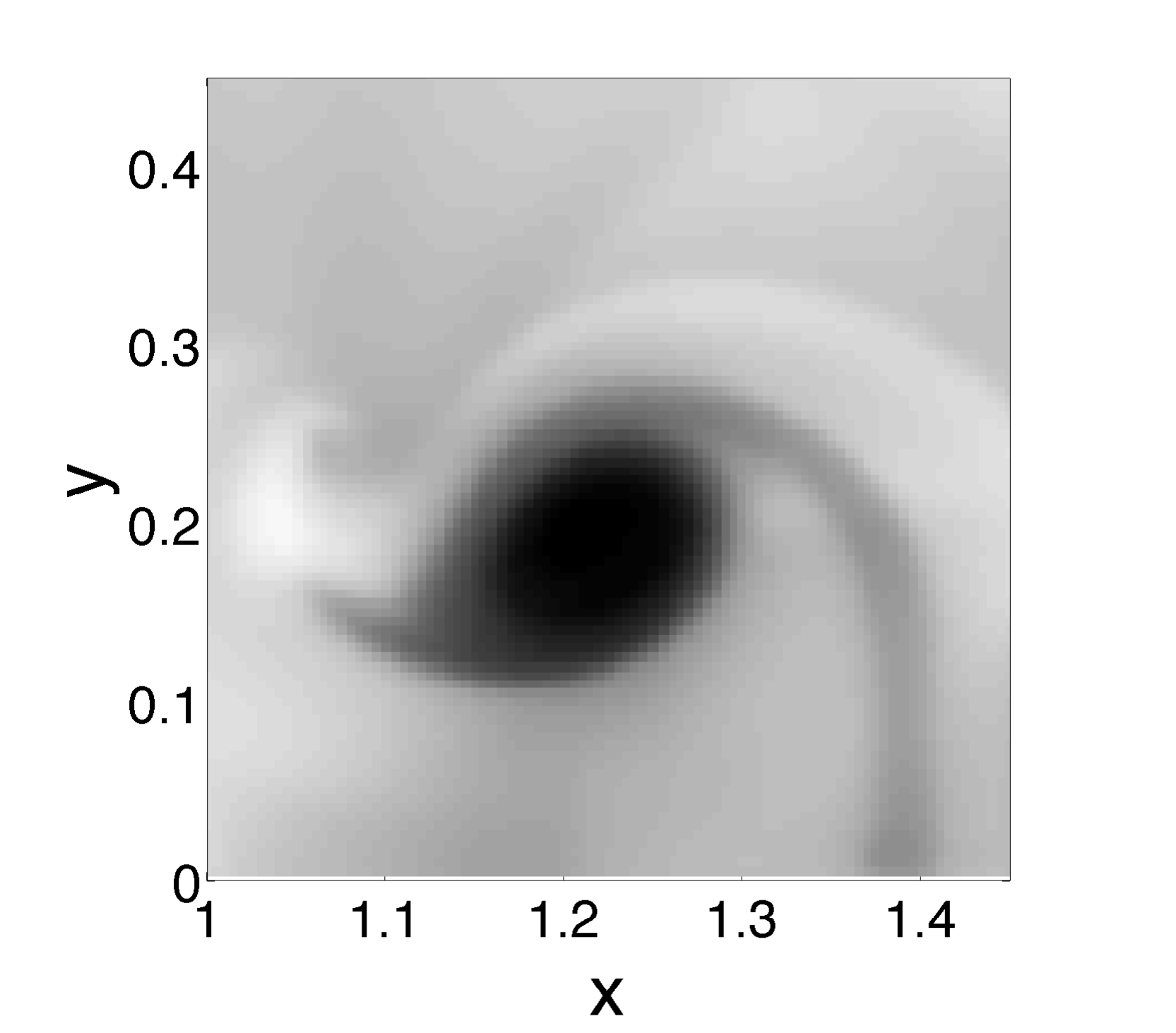}
\vglue -0.05 truein
\caption{Test (b)}
\label{fg:sec_example_bubble_den_loc_va_lim}
\end{subfigure}
\begin{subfigure}[b]{0.32\textwidth}\centering
\includegraphics[trim=0.7in 0.0in 1.3in 0.5in, clip, width=\textwidth]{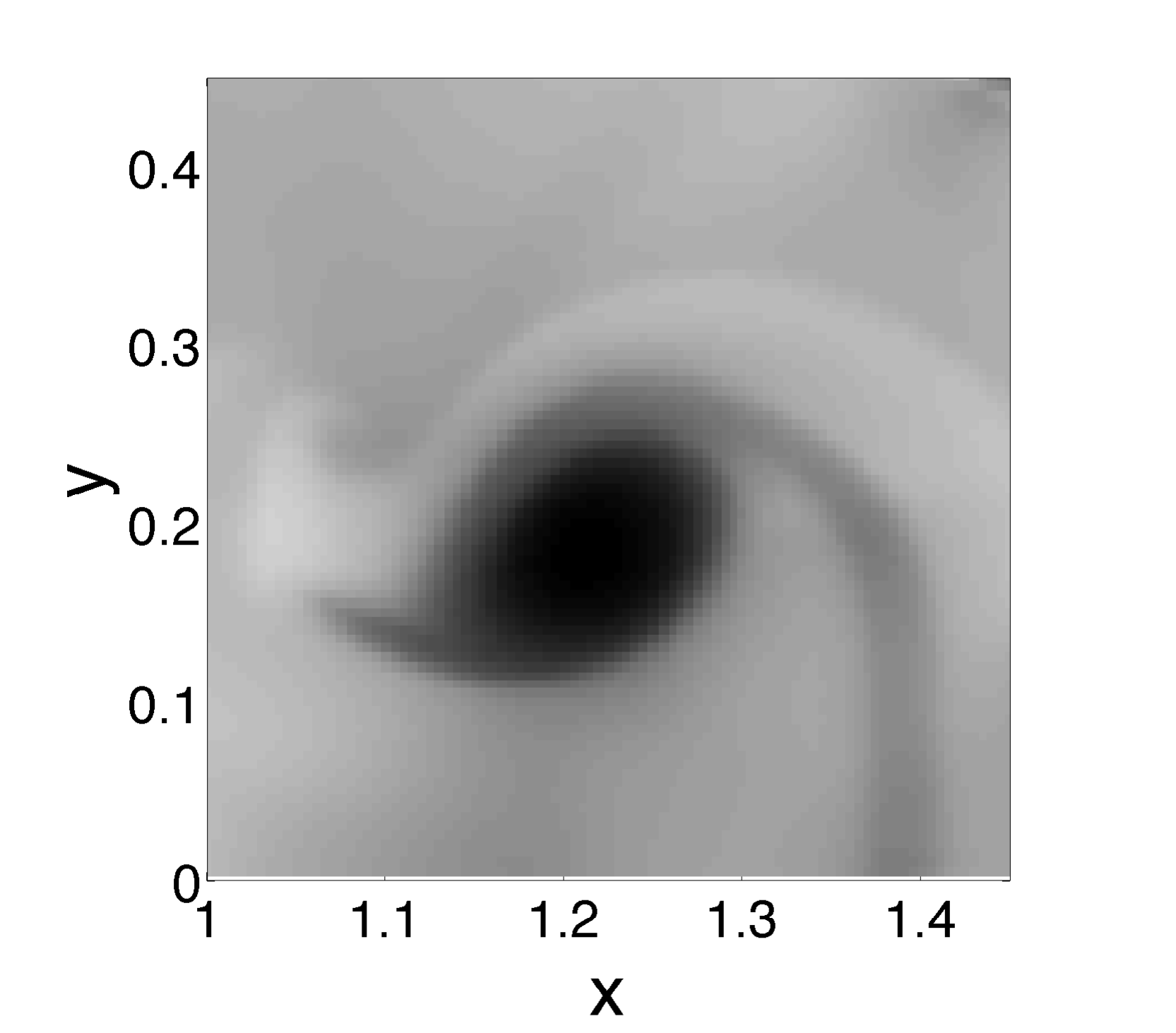}
\vglue -0.05 truein
\caption{Test (c)}
\label{fg:sec_example_bubble_den_loc_va_cap}
\end{subfigure}
\vglue -0.2 truein
\caption{Local views of density at $T=0.4$ on a $340\times100$ mesh ($r=0.3$): 
    (\ref{fg:sec_example_bubble_den_loc_va_org}) MUSCL-MOL and $\phi^{\it \textrm{van Albada}}$, 
    (\ref{fg:sec_example_bubble_den_loc_va_lim}) MUSCL-MOL and $\phi^{\it \textrm{van Albada}}_{A,B}$,
and (\ref{fg:sec_example_bubble_den_loc_va_cap}) capacity-form differencing and $\phi^{\it \textrm{van Albada}}$}
\label{fg:sec_example_bubble_den_loc_va}
\end{figure}

On the one hand, by comparing the local views (Figures~\ref{fg:sec_example_bubble_den_loc_vl} and~\ref{fg:sec_example_bubble_den_loc_va}) with the black box in the reference solution (Figure~\ref{fg:sec_example_bubble_den}), one may conclude that the MUSCL-MOL using enhanced limiters are as accurate as the capacity-form differencing, which are both more accurate than MUSCL-MOL with conventional limiters.
On the other hand, the upper-right corners of the global solutions (Figure~\ref{fg:sec_example_bubble_den_glob}) show that the capacity-form differencing with conventional limiters tend to be unstable or unphysical on highly irregular grids, especially if the {\it van Leer} limiter is used.
Note that unlike the 1D example in Section~\ref{sec:alt_consist}, the {\it van Albada} limiter does not save the capacity-form differencing from being unstable (Figure~\ref{fg:sec_example_bubble_den_glob_va_cap}).

\section{Conclusions}
\label{sec:concl}
This paper raises the issue of limitations of using conventional slope limiter functions in finite volume methods using highly non-uniform rectilinear grids.
In particular, depending on particular FVM implementations, one may lose either second-order accuracy or TVD stability in these situations, especially when the limiter functions are smooth.
Using the MUSCL in space and method of lines in time, this paper analyzes and enhances the limiter functions on general 1D grids using an extension of the REP procedure.
These enhanced limiter functions, including the {\it van Leer} and the {\it van Albada} limiters, satisfy sufficient conditions for a limiter function to lead to a formally second-order accurate, TVD stable, and symmetry-preserving methods.
Their numerical performances are assessed by solving various 1D and 2D benchmark problems; 
the results are compared to alternative FVM strategies to solve these problems on non-uniform grids.

\section*{Acknowledgments}
The author thanks the support by a Stanford Graduate Fellowship.
The author also gratefully thanks Professor Charbel Farhat for reviewing the manuscript and his insightful suggestions.

\bibliographystyle{siam}

\appendix

\section{Computing convergence rates on non-uniform meshes}
\label{app:rates}
Given two randomly generated irregular meshes of the domain $[a,b]$ 
\begin{align*}
\textrm{Mesh I:}\quad& a=x_{1/2}^I<x_{3/2}^I<\cdots<x_{(N_1-1)+1/2}^I<x_{N_1+1/2}^I=b \\
\textrm{Mesh II:}\quad& a=x_{1/2}^{II}<x_{3/2}^{II}<\cdots<x_{(N_2-1)+1/2}^{II}<x_{N_2+1/2}^{II}=b
\end{align*}
let the reference cell-sizes be $h^I = (b-a)/N_1$ and $h^{II} = (b-a)/N_2$, respectively.

Supposing that the dependent variable is a scalar $u$, define $u^{I}_i$ and $u^{II}_i$ to be the numerical solutions at $T$ computed by the same scheme on the two meshes, respectively.
Letting $u^{\textrm{ref}}(x,T)$ be the reference solution, the $L_1$ norm of the computed errors are calculated as
\begin{displaymath}
E^{I,II}  = \sum_{i=1}^{N_1}(x_{i+1/2}^{I,II}-x_{i-1/2}^{I,II})\left|u^{I,II}_i-u^{\textrm{ref}}(x_i^{I,II},T)\right|
\end{displaymath}
Then the convergence rates is estimated by
\begin{equation}\label{eq:app_rates_rate}
R = \frac{\log(E^{I})-\log(E^{II})}{\log(h^{I})-\log(h^{II})}
\end{equation}
This equation has wide usage for estimating convergence rates on uniform meshes, but its validity for the same purpose on randomly generated non-uniform meshes is open for discussion. 
Nevertheless, the numerical results presented in Section~\ref{sec:example} justifies this choice for irregular meshes generated in the way described in Section~\ref{sec:case}.

\section{The conventional {\it \textbf{van Leer}} limiter}
\label{app:vanleer}
Consider the next general strategy to construct the alternative unlimited slopes and smoothness monitors
\begin{equation}\label{eq:vanleer_alt}
\theta_i^{alt} = \alpha_i\theta_i^p,\quad
D_xu_i = \beta_i\theta_i^q\frac{u_{i+1}-u_i}{\Delta x_i}
\end{equation}
in which $\theta_i$ is the smoothness monitor~(\ref{eq:sec_case_monitor}), $\alpha_i$ and $\beta_i$ are two positive numbers that only depends on the local cell sizes.
The other two numbers are $p\in\{-1,1\}$ and $q\in\{0,1\}$, so that (\ref{eq:vanleer_alt}) allows biased differencing in both directions.

Thus by setting $\alpha_i=\beta_i\equiv1$ and using $p=1, q=0$, one obtains the strategy chosen in this paper;
whereas setting $\theta_i\equiv1$, $p=1$, $q=0$, and $\beta_i = 2\Delta x_i/(\Delta x_i+\Delta x_{i+1})$, one has the consistent numerical slope~(\ref{eq:alt_consist_slpe}).
For the general formula~(\ref{eq:vanleer_alt}), the next theorem holds for the conventional {\it van Leer} limiter.
\begin{theorem}\label{thm:app_vanleer}
Using (\ref{eq:vanleer_alt}) and the conventional limiter function $\phi^{\it \textrm{van Leer}}$, 
there exists a function $\phi^{\it \textrm{van Leer}}_{i}$ for each $i$ such that the limited numerical slopes satisfy
\begin{equation}\label{eq:vanleer_equiv}
\sigma_i \eqdef \phi^{\it \textrm{van Leer}}(\theta_i^{alt})D_xu_i = \phi_{i}^{\it \textrm{van Leer}}(\theta_i)\frac{u_{i+1}-u_i}{\Delta x_i}
\end{equation}
These $\phi_i^{\it \textrm{van Leer}}$ cannot satisfy both (\ref{eq:sec_math_2nd}) and (\ref{eq:sec_math_tvd}) for arbitrary non-uniform grids.
\end{theorem}

Before proving the theorem, the idea is that any alternative strategy that may be written as (\ref{eq:vanleer_alt}), one can construct a ``modified'' limiter function such that the analysis of Section~\ref{sec:math} fits.
In particular, the ``modified'' limiter function corresponding to $\phi^{\it \textrm{van Leer}}$ cannot satisfy both the (\ref{eq:sec_math_2nd}) and (\ref{eq:sec_math_tvd}) simultaneously for arbitrary grids: 
thus there always be some irregular grids, such that this alternative strategy leads to the loss of either second-order accuracy or the TVD stability.
\begin{proof}
For simplicity, $\theta_i>0$ is always supposed in the proof.
The general formula~(\ref{eq:vanleer_alt}) together with conventional {\it van Leer} limiter leads to the following slopes
\begin{equation}\label{eq:vanleer_alt_slpe}
\sigma_i = \phi^{\it \textrm{van Leer}}(\alpha_i\theta_i^p)\beta_i\theta_i^q\frac{u_{i+1}-u_i}{\Delta x_i}
\end{equation}
Thus (\ref{eq:vanleer_equiv}) is established by defining the ``modified'' limiters as
\begin{displaymath}
\phi_i = \beta_i\theta_i^q\phi^{\it \textrm{van Leer}}(\alpha_i\theta_i^p)
\end{displaymath}
Plugging in the explicit expression of $\phi^{\it \textrm{van Leer}}$, one has the ``modified'' limiter function
\begin{equation}\label{eq:vanleer_alt_lim}
\phi_i = \phi_i^{\it \textrm{van Leer}}(\theta_i);\qquad
\phi_i^{\it \textrm{van Leer}}(\theta) \eqdef \frac{\theta^s}{a+b\theta},\quad s\in\{0,1,2\}
\end{equation}
where $a = 1/(2\alpha_i\beta_i)$, $b = 1/(2\beta_i)$, $s=1+q$ if $p=1$, and $a = 1/(2\beta_i)$, $b=1/(2\alpha_i\beta_i)$, $s=q$ if $p=-1$.
Considering the two numbers $A=\frac{\Delta x_{i-1}+\Delta x_i}{\Delta x_i+\Delta x_{i+1}}$ and $B=\frac{2\Delta x_i}{\Delta x_i+\Delta x_{i+1}}$ instead of $\{\Delta x_j\}$, it is then sufficient to show that $\phi_i^{\it \textrm{van Leer}}$ cannot satisfy both (\ref{eq:sec_math_2nd}) and (\ref{eq:sec_math_tvd}) for all $A,B$ such that (\ref{eq:sec_math_a_b}) and (\ref{eq:sec_math_lim_a_b}) hold.

Indeed, suppose for some $a, b$, and $s$, $\phi_i$ of (\ref{eq:vanleer_alt_lim}) satisfies both (\ref{eq:sec_math_2nd}) and (\ref{eq:sec_math_tvd})
\begin{displaymath}
0\le\frac{\theta^s}{a+b\theta}\le2\min(1,\theta),\quad\forall\theta>0\quad\Rightarrow\quad
t=1\ \textrm{ and }\ a,b\ge\frac{1}{2}
\end{displaymath}
Setting $t=1$, $\phi_i^{\it \textrm{van Leer}}(A)=B$ leads to
\begin{displaymath}
\frac{A}{B} = 2a + 2b A \ge 1 + A
\end{displaymath}
which is, however, not true for all $A,B:\ 0 < B < 2\min(1,A)$, contradiction.
\end{proof}

\end{document}